\documentclass[12pt]{article}
\usepackage{mathrsfs}
\usepackage{mathrsfs}
\usepackage{amsmath, amsthm, amsfonts, amssymb}
\usepackage{mathrsfs}
\usepackage{bbm}
\usepackage{amssymb}
\usepackage{txfonts}

\newlength{\defbaselineskip}
\newcommand{\setlinespacing}[1]%
           {\setlength{\baselineskip}{#1 \defbaselineskip}}

\newcounter{marnote}
\newcommand\marginnote[1]{\stepcounter{marnote}$^{\bullet\,\themarnote}$\marginpar{\tiny$\bullet\,\themarnote$:\,#1}}

\theoremstyle{plain}
\newtheorem{theorem}{Theorem}[section]
\newtheorem{corollary}[theorem]{Corollary}
\newtheorem{lemma}[theorem]{Lemma}
\newtheorem{prop}[theorem]{Proposition}
\theoremstyle{definition}
\newtheorem{defn}{Definition}[section]
\theoremstyle{remark}
\newtheorem{remark}{Remark}[section]
\numberwithin{equation}{section}

\begin{document}


\title{Gradient Estimates for Parabolic Systems from Composite Material}

\author{Haigang Li\footnote{School of Mathematical Sciences, Beijing Normal University, Laboratory of Mathematics and Complex Systems, Ministry of
   Education, Beijing 100875,
China. Email: hgli@bnu.edu.cn}\quad and\quad Yanyan
Li\footnote{Department of Mathematics, Rutgers University, 110
Frelinghuysen Rd, Piscataway, NJ 08854, USA. Email:
yyli@math.rutgers.edu}}

\date{}

\maketitle

\begin{abstract}

In this paper we derive $W^{1,\infty}$ and piecewise $C^{1,\alpha}$
estimates for solutions, and their $t-$derivatives, of divergence
form parabolic systems with coefficients piecewise H\"older
continuous in space variables $x$ and smooth in $t$.  This is an
extension to parabolic systems of results of Li and Nirenberg on
elliptic systems. These estimates depend on the shape and the size
of the surfaces of discontinuity of the coefficients, but are
independent of the distance between these surfaces.

\end{abstract}

\section{Introduction and Main Results}\label{sec1}

The purpose of this paper is to establish gradient estimates for
some parabolic systems of divergence form, which arise from the
study of composite material. Babu\u{s}ka et al. \cite{basl} were
interested in elliptic systems arising in elasticity. They observed
numerically that, for certain homogeneous isotropic linear systems
of elasticity, $|\nabla{u}|$ stay bounded independently of the
distance between the regions. Bonnetier and Vogelius \cite{bv}
proved the boundedness of $|\nabla{u}|$ for a scalar elliptic
equation in bounded domains with two unit balls touching at a point.
This result was extended by Li and Vogelius in \cite{lv} to general
second order elliptic equations with piecewise H\"{o}lder
coefficients, where stronger $C^{1,\alpha}$ estimates were
established. Later, $C^{1,\alpha}$ estimates were obtained by Li and
Nirenberg in \cite{ln} for general second order elliptic systems
including systems of elasticity, with an improved H\"{o}lder
exponent $\alpha$. A remaining open problem is to determine the
optimal regularity; see the open problem on page 894 of \cite{ln}.
In this paper we extend the interior $ C^{1,\alpha}$ estimates in
\cite{ln} to parabolic systems. For parabolic problems, related
results were given by Almgren and Wang \cite{aw} and Dong \cite{d}.

Let ${D}\subset\mathbb{R}^{n}$ $(n\geq1)$ be a bounded domain that contains $L$
disjoint subdomains ${D}_{1},\cdots,{D}_{L}$, with
${D}=(\cup\overline{{D}}_{m})\setminus\partial{D}$. Suppose that
their boundaries $\partial{D}$, $\partial{D}_{m}$ are $C^{1,\alpha}$
for some $0<\alpha<1$. Denote $Q_{T}={D}\times(0,T)$, for some
$T>0$.

We study interior gradient estimates of solutions of the linear
parabolic systems
\begin{equation}\label{IBVP}
(u^{i})_{t}-D_{\alpha}\bigg(A_{ij}^{\alpha\beta}(x,t)D_{\beta}u^{j}\bigg)=-D_{\alpha}g^{\alpha\,i}+f^{i}\quad\quad\hbox{in}\quad
Q_{T}.
\end{equation}
Here $u(x,t)=(u^{1}(x,t),\cdots,u^{N}(x,t))$ is a vector-valued
function, and we use $D_{i}u$ for $\partial{u}/\partial{x}_{i}$
while we use $u_{t}$ (or sometimes $D_{t}u$) for
$\partial{u}/\partial{t}$. We also write $Du$ (or sometimes
$\nabla{u}$) for $(D_{1}u,\cdots,D_{2}u)$. Throughout this paper we
use the usual summation convention over repeated indices: $\alpha$
and $\beta$ are summed from $1$ to $n$, while $i$ and $j$ are summed
from $1$ to $N$.

The coefficients $A_{ij}^{\alpha\beta}(x,t)$, often also denoted by
$A$,
\begin{itemize}
\item [(a)]are measurable and bounded,
\begin{equation}\label{coeff1}
|A_{ij}^{\alpha\beta}|\leq\Lambda_{0};
\end{equation}
\item [(b)] and $A_{ij}^{\alpha\beta}(x,t)$
satisfy the following (weak) parabolic condition: for some constant
$\lambda>0$,
\begin{equation}\label{coeff2}
\int_{{D}}A_{ij}^{\alpha\beta}(x,t)\partial_{\alpha}\xi^{i}\partial_{\beta}\xi^{j}dx\geq\lambda\int_{{D}}|\nabla\xi|^{2}dx,
\ \ \forall
\xi\in\overset{\circ}{W}\,^{1,0}_{2}(Q_{T};\mathbb{R}^{N}),\ \forall
t\in(0,T);
\end{equation}

\item [(c)] Furthermore, $A_{ij}^{\alpha\beta}(x,t)$ is of class
$C^{\mu,\infty}(\overline{{D}}_{m}\times(0,T))$, that is, for some
constants $0<\mu<1$, and $C$ such that
$$
\big|A(x,t)-A(y,t)\big|\leq C\big|x-y\big|^{\mu},\quad\forall
(x,t),(y,t)\in{D}_{m}\times(0,T), m=1,2,\cdots,L;
$$
and, for every integer $k\geq 1$, there exists $\Lambda_{2k}$,
depending on $k$, such that
\begin{equation}\label{coeff31}
\begin{array}{cl}
\sum\limits_{s=0}^{k}\big|D_{t}^{s}A(x,t)\big|\leq\Lambda_{2k},\\
\sum\limits_{s=0}^{k}\big|D_{t}^{s}A(x,t)-D_{t}^{s}A(y,t)\big|\leq
\Lambda_{2k}\big|x-y\big|^{\mu},
\end{array}~\mbox{in}~D_{m}\times(0,T),m=1,\cdots,L.
\end{equation}
\end{itemize}

Finally, we assume that $f\in{L}^{\infty}(Q_{T})$, and
$g\in{C}^{\mu,0}(\overline{{D}}_{m}\times[0,T])$. Here
$C^{\mu,k}(\overline{{D}}_{m}\times[0,T])$ denotes the Banach space
of functions $g(x,t)$ that are $C^{k}$ continuous in $t$ and
H\"{o}lder continuous in $x$ with exponent $\mu\in[0,1]$, and having
finite norms
$$\|g\|_{C^{\mu,k}}=\sum_{s=0}^{k}\sup_{\overline{{D}}_{m}\times[0,T]}|D_{t}^{s}g|
+\sup_{\scriptstyle x,y\in{D}_{m} \atop \scriptstyle
t\in[0,T]}\frac{|g(x,t)-g(y,t)|}{|x-y|^{\mu}}.$$ Throughout the
paper, unless otherwise stated, we will always assume that these
hypotheses hold.

We define weak solutions of \eqref{IBVP} as in \cite{ch}. Denote by
$V(Q_{T};\mathbb{R}^{N})$ the set of all
$u\in{L}^{2}(Q_{T};\mathbb{R}^{N})$ such that
$Du\in{L}^{2}(Q_{T};\mathbb{R}^{N})$,
$u(\cdot,t)\in{L}^{2}(D;\mathbb{R}^{N})$ for a.e. $t\in[0,T]$, and
having finite norms
$$\|u\|_{V(Q_{T})}:=\left(\int_{0}^{T}\int_{{D}}|Du|^{2}dxdt+\textrm{ess}\sup\limits_{0\leq\tau\leq\,T}\int_{{D}}u^{2}(x,\tau)dx\right)^{\frac{1}{2}}.$$
$W^{1,1}_{2}(Q_{T};\mathbb{R}^{N})$ denotes the Hilbert space with
the inner product
$$\langle u,v\rangle_{W^{1,1}_{2}}=\int_{0}^{T}\int_{{D}}(uv+D_{x}uD_{x}v+D_{t}uD_{t}v)dxdt.$$
By $\overset{\circ}{V}(Q_{T};\mathbb{R}^{N})$ and
$\overset{\circ}{W}\,^{1,1}_{2}(Q_{T};\mathbb{R}^{N})$ we denote the
subset of $V(Q_{T};\mathbb{R}^{N})$ and
$W^{1,1}_{2}(Q_{T};\mathbb{R}^{N})$, respectively, with the elements
satisfying $u(\cdot,t)|_{\partial{D}}=0$ a.e. $t\in(0,T)$.

\begin{defn}
For $f\in{L}^{\infty}(Q_{T})$,
$g\in{C}^{\mu,0}(\overline{{D}}_{m}\times[0,T])$, we say that $u$ is
a weak solution of \eqref{IBVP}, if
$u\in\overset{\circ}{V}(Q_{T};\mathbb{R}^{N})$ satisfies, for a.e
$\tau\in(0,T)$ and for all
$\zeta\in\overset{\circ}{W}\,^{1,1}_{2}(Q_{T};\mathbb{R}^{N})$ that
vanish at $t=0$,  the identity
\begin{align}\label{weaksoln1}
\int_{{D}}(u\zeta)(\cdot,\tau)\,dx-\int_{0}^{\tau}\int_{{D}}u\zeta_{t}dxdt&+\int_{0}^{\tau}\int_{{D}}A^{\alpha\beta}_{ij}D_{\beta}u^{j}D_{\alpha}\zeta^{i}\,dxdt
\nonumber\\
&=\int_{0}^{\tau}\int_{{D}}(f\zeta+g^{\alpha}D_{\alpha}\zeta)dxdt.
\end{align}
\end{defn}
For $\epsilon>0$ small, set
$${D}_{\epsilon}=\bigg\{x\in{D}~\big|~\mathrm{dist}(x,\partial{D})>\epsilon\bigg\}.$$
The first of our main results concerns the H\"{o}lder interior
estimates for the gradient.
\begin{theorem}\label{thm1}
Under the assumptions on ${D},A,f,g,\varphi$ mentioned above, let
$u\in{V}(Q_{T};\mathbb{R}^{N})$ be a weak solution of \eqref{IBVP}.
Then for any $0<\epsilon<\frac{1}{2}$ and
$\alpha'<\min\{\mu,\frac{\alpha}{2(1+\alpha)}\}$,
\begin{align}\label{thm1.1}
\|u&\|_{L^{\infty}({D}_{\epsilon}\times(\epsilon{T},T))}+
\|D_{x}u\|_{C^{\alpha',0}(({D}_{\epsilon}\cap\overline{{D}}_{m})\times(\epsilon{T},T))}\nonumber\\
&\leq\,C\left(\|u\|_{L^{2}(Q_{T})}
+\|f\|_{L^{\infty}(Q_{T})}+\max_{1\leq\,m\leq\,L}\|g\|_{C^{\alpha',0}(\overline{{D}}_{m}\times[0,T])}\right),
\end{align}
where $C$ depends only on
$n,N,L,\alpha,\epsilon,\lambda,\Lambda_{0},\mu,T$,$\|A\|_{C^{\alpha',1}(\overline{{D}}_{m}\times[0,T])}$
and the $C^{1,\alpha}$ norm of ${D}_{m}$. In particular,
\begin{equation}\label{thm1.2}
\|D_{x}u\|_{L^{\infty}({D}_{\epsilon}\times(\epsilon{T},T))}\leq\,C\left(\|u\|_{L^{2}(Q_{T})}
+\|f\|_{L^{\infty}(Q_{T})}+\max_{1\leq\,m\leq\,L}\|g\|_{C^{\alpha',0}(\overline{{D}}_{m}\times[0,T])}\right).
\end{equation}
\end{theorem}

\begin{remark}
Theorem \ref{thm1} in the case $A(x,t)\equiv{A}(x)$, independent of
$t$, was included as a part of the thesis of the first author, see
\cite{l}.
\end{remark}

Further, if we suppose that $A_{ij}^{\alpha\beta}$ is of class
$C^{\mu,k+1}(\overline{{D}}_{m}\times[0,T])$, and
$f\in{C}^{0,k}(\overline{{D}}_{m}\times[0,T])$,
$g\in{C}^{\mu,k}(\overline{{D}}_{m}\times[0,T])$, then we have the
following estimates on $u'$s higher order derivatives.

\begin{theorem}\label{thm2}
Under the assumptions on ${D},A,f,g,\varphi$ mentioned above, let
$u\in{V}(Q_{T};\mathbb{R}^{N})$ be a weak solution of \eqref{IBVP}.
Then for any $l\leq{k}$, for any $0<\epsilon<\frac{1}{2}$ and
$\alpha'<\min\{\mu,\frac{\alpha}{2(1+\alpha)}\}$,
\begin{align}\label{thm2.1}
\|D_{t}^{l}u&\|_{L^{\infty}({D}_{\epsilon}\times(\epsilon{T},T))}+
\|D_{t}^{l}D_{x}u\|_{C^{\alpha',0}(({D}_{\epsilon}\cap\overline{{D}}_{m})\times(\epsilon{T},T))}\nonumber\\
&\leq\,C\left(\|u\|_{L^{2}(Q_{T})}
+\sum_{s=1}^{l}\|D_{t}^{s}f\|_{L^{\infty}(Q_{T})}+\max_{1\leq\,m\leq\,L}\|g\|_{C^{\alpha',l}(\overline{{D}}_{m}\times[0,T])}\right),
\end{align}
where $C$ depends only on
$n,N,L,l,\alpha,\epsilon,\lambda,\Lambda_{2l},\mu,T$, diam$({D})$,
the $C^{1,\alpha}$ norm of ${D}_{m}$, and
$\|A\|_{C^{\alpha',k+1}(\overline{{D}}_{m}\times[0,T])}$.
\end{theorem}

We draw attention to some closely related results in \cite{fknn} by
J. Fan, K. Kim, S. Nagayasu and G. Nakamura.

This paper is organized as follows. In Section \ref{sec2}, we
introduce some notations and present some standard $L^{2}$ estimates
for readers' convenience. A result of Chipot, Kinderlehrer, and
Vergara-Caffarelli \cite{ck} for laminar elliptic systems is
extended to laminar parabolic systems in Section \ref{sec3}. We
present a general perturbation result, an extension of Li-Vogelius
and Li-Nirenberg for elliptic equations and systems, in Section
\ref{sec5}, and apply it to establish uniform gradient estimates in
Section \ref{gradient}. Finally, the proofs of main theorems are
given in Section \ref{sec7}.

\bigskip
\section{Preliminary Results}\label{sec2}
In this section we mainly follow the notations and definitions  of
\cite{ch} and \cite{lsu} and list some standard $L^{2}$ estimates
for readers' convenience.

\subsection{Function Spaces}
Let ${D}$ be a bounded open set of $\mathbb{R}^{n}$.
$H^{1,p}({D};\mathbb{R}^{N})$ and $H_{0}^{1,p}({D};\mathbb{R}^{N})$
are the usual Sobolev spaces of the vector-valued function
$u:D\rightarrow\mathbb{R}^{N}$; if $p=2$ we shall write more briefly
$H^{1}$ and $H^{1}_{0}$.

We shall also use $V(Q_{T})$ and $\overset{\circ}{V}(Q_{T})$ to
denote $V(Q_{T};\mathbb{R}^{N})$ and
$\overset{\circ}{V}(Q_{T};\mathbb{R}^{N})$, respectively, when there
is no ambiguity. Let $W^{1,0}_{2}(Q_{T})$, $W^{1,1}_{2}(Q_{T})$, and
$W_{2}^{0,k}(Q_{T})$ denote the Hilbert spaces with the inner
product
$$\langle u,v\rangle_{W^{1,0}_{2}}=\int_{0}^{T}\int_{{D}}\bigg(uv+D_{x}uD_{x}v\bigg)dxdt,$$
$$\langle u,v\rangle_{W^{1,1}_{2}}=\int_{0}^{T}\int_{{D}}\bigg(uv+D_{x}uD_{x}v+D_{t}uD_{t}v\bigg)dxdt,$$
and
$$\langle u,v\rangle_{W_{2}^{0,k}}=\int_{0}^{T}\int_{{D}}\left(\sum_{s\leq{k}}D^{s}_{t}uD^{s}_{t}v\right)dxdt,$$
respectively. Denote by $V^{1,0}(Q_{T})$ the Banach space consisting
of all elements of $V(Q_{T})$ satisfying
\begin{equation}\label{v10}
\lim_{h\rightarrow 0}\|u(\cdot,t+h)-u(\cdot,t)\|_{L^{2}(D)}=0,
\quad\mbox{uniformly for}\ t,t+h\in [0,T].
\end{equation}
In fact
$$W_{2}^{1,1}(Q_{T})\subset\,V^{1,0}(Q_{T})\subset\,V(Q_{T}),$$
and $V^{1,0}(Q_{T})$ is the completion of $W_{2}^{1,1}(Q_{T})$ with
respect to the norm $\|\cdot\|_{V(Q_{T})}$. Similarly, a zero over
$W^{1,0}_{2}(Q_{T})$, $W^{1,1}_{2}(Q_{T})$, $W_{2}^{0,k}(Q_{T})$,
$V^{1,0}(Q_{T})$ and $V(Q_{T})$ means that only those elements of
the spaces are taken which satisfy $u(\cdot,t)|_{\partial{D}}=0$
a.e. $t\in(0,T)$.

For $u(x,t)\in{V}^{1,0}(Q_{T})$, we will consider the Steklov
average, i.e. $u'$s average in $t$,
$$u_{h}(x,t):=\frac{1}{h}\int_{t}^{t+h}u(x,\tau)d\tau,$$
for $0<t<t+h<T$. It is clear that if
$u\in\overset{\circ}{V}\,^{1,0}(Q_{T})$, then for any $0<h<\delta$
$(0<\delta<T)$, $u_{h}\in
\overset{\circ}{W}\,^{1,1}_{2}(Q_{T-\delta})$, and
\begin{equation}\label{uh}
\|u_{h}-u\|_{V(Q_{T-\delta})}\rightarrow 0,\quad\mbox{as}\quad
h\rightarrow 0,
\end{equation}
as shown below. First,
$$\|u_{h}(\cdot,t)-u(\cdot,t)\|^{2}_{L^{2}({D})}\leq\frac{1}{h}\int_{t}^{t+h}\|u(x,\tau)-u(\cdot,t)\|^{2}_{L^{2}({D})}d\tau
\leq\sup_{t\leq\tau\leq\,t+h}\|u(x,\tau)-u(\cdot,t)\|^{2}_{L^{2}({D})}.$$
Since $u\in{V}^{1,0}(Q_{T})$, it follows that the right hand side
converges to zero as $h\rightarrow 0$. On the other hand,
\begin{align*}
\|Du_{h}(\cdot,t)-Du(\cdot,t)\|^{2}_{L^{2}(Q_{T-\delta})}&\leq\frac{1}{h}\int_{0}^{h}\|Du(x,t+s)-Du(x,t)\|^{2}_{L^{2}(Q_{T-\delta})}ds\\
&\leq\sup_{0\leq\,s\leq\,h}\|Du(x,t+s)-Du(x,t)\|^{2}_{L^{2}(Q_{T-\delta})}.
\end{align*}
Since $L^{2}(Q_{T})$ functions are continuous with respect to
translations, this supremum tends to zero as $h\rightarrow 0$.

\subsection{$L^{2}$ Estimates for the Initial Boundary Value
Problem}

Let $D$ and $f,g$ be defined in Section \ref{sec1}. We assume that
$A^{\alpha\beta}_{ij}(x,t)$ satisfy \eqref{coeff1} \eqref{coeff2},
and \eqref{coeff31} and $f,g$ are smooth in $t$. We consider the
following initial boundary value problem of the parabolic systems
\begin{equation}\label{VPIND}
\left\{
  \begin{array}{ll}
    (u^{i})_{t}-D_{\alpha}\left(A_{ij}^{\alpha\beta}(x,t)D_{\beta}u^{j}\right)=-D_{\alpha}g^{\alpha\,i}+f^{i} & \hbox{in}\ {D}\times(0,T),\\
    u=0 & \hbox{on}\ \partial{D}\times(0,T),\\
    u=\varphi(x) & \hbox{on}\ {D}\times\{0\},
  \end{array}
\right.
\end{equation}
where $\varphi\in{L}^{2}(D)$. We now define a weak solution of
problem \eqref{VPIND}.

\begin{defn}\label{definition1}
For $f\in{L}^{\infty}(Q_{T})$,
$g\in{C}^{\mu,0}(\overline{{D}}_{m}\times[0,T])$, and $\varphi\in
L^{2}({D})$, we say that $u$ is a weak solution of \eqref{VPIND}, if
$u\in\overset{\circ}{V}(Q_{T})$ satisfies, for a.e $\tau\in(0,T)$
and for all $\zeta\in\overset{\circ}{W}\,^{1,1}_{2}(Q_{T})$, the
identity
\begin{align}\label{weaksoln}
\int_{{D}}(u\zeta)(\cdot,\tau)\,dx-\int_{{D}}(\varphi\zeta)(\cdot,0)\,dx&-\int_{0}^{\tau}\int_{{D}}u\zeta_{t}dxdt
+\int_{0}^{\tau}\int_{{D}}A^{\alpha\beta}_{ij}D_{\beta}u^{j}D_{\alpha}\zeta^{i}\,dxdt\nonumber\\
&=\int_{0}^{\tau}\int_{{D}}(f\zeta+g^{\alpha}D_{\alpha}\zeta)dxdt.
\end{align}
\end{defn}

The following lemmas and their proofs follow book \cite{ch} and
\cite{lsu}. We present them here for readers' convenience.

\begin{lemma}\label{lem31}
Assume the above. Suppose $u\in\overset{\circ}{V}(Q_{T})$ is a weak
solution of problem \eqref{VPIND}, then $u$ belongs to
$\overset{\circ}{V}\,^{1,0}(Q_{T})$.
\end{lemma}

\begin{remark}\label{rem1}
Suppose $u\in{V}(Q_{T})$ is a weak solution of \eqref{IBVP}. Then
$u\in{V}^{1,0}(Q^{'}_{T})$, for any $Q'_{T}=D'\times(0,T)$,
${D}'\subset\subset{D}$. The proof of this assertion is similar to
the proof of Lemma \ref{lem31}.
\end{remark}

\begin{proof}[Proof of Lemma \ref{lem31}.]
By Definition \ref{definition1}, we have, for a.e. $\tau\in(0,T)$
and for any $\zeta\in\overset{\circ}{W}\,^{1,1}_{2}(Q_{T})$, $u$
satisfies the identity \eqref{weaksoln}. Since $u\in{V}(Q_{T})$, it
is obvious that
\begin{equation}\label{lem31.2}
\|u(\cdot,\tau)\|_{L^{2}({D})}\leq\|u\|_{V(Q_{T})},\quad\quad\mbox{a.e.}\
\tau\in(0,T).
\end{equation}
Denote $I=\big\{\tau\in(0,T)\ |\ \eqref{weaksoln}\ \mbox{and}\
\eqref{lem31.2}\ \mbox{hold simultaneously}\big\}$, then the measure
of $([0,T]\setminus{I})$ is zero. For $\tau\in[0,T]\setminus{I}$, we
can redefine $u(x,\tau)$ such that \eqref{weaksoln} and
\eqref{lem31.2} still hold. Hence, without loss of generality, we
assume that $u(x,\tau)$ satisfies \eqref{weaksoln} and
\eqref{lem31.2} for every $\tau\in[0,T]$.

For $\tau\in(0,T)$, and $\Delta{t}>0$, denote
$Q_{\tau,\tau+\Delta{t}}={D}\times(\tau,\tau+\Delta{t})$. By
\eqref{weaksoln}, we have
\begin{align}\label{lem31.3}
\int_{{D}}(u\zeta)(\cdot,\tau+\Delta\,t)\,dx-\int_{{D}}(u\zeta)(\cdot,\tau)\,dx&-\int_{\tau}^{\tau+\Delta\,t}\int_{{D}}u\zeta_{t}dxdt
\nonumber\\+\int_{\tau}^{\tau+\Delta\,t}\int_{{D}}A^{\alpha\beta}_{ij}D_{\beta}u^{j}D_{\alpha}\zeta^{i}\,dxdt
&=\int_{\tau}^{\tau+\Delta\,t}\int_{{D}}(f\zeta+g^{\alpha}D_{\alpha}\zeta)dxdt.
\end{align}
Picking the test function $\zeta(x,t)=\zeta(x)\in{H}_{0}^{1}(D)$,
the usual Sobolev space, then we have
\begin{align*}
&\int_{{D}}(u(x,\tau+\Delta\,t)-u(x,\tau))\zeta\,dx\nonumber\\
&=\int_{\tau}^{\tau+\Delta\,t}\int_{{D}}f\zeta\,dxdt+
\int_{\tau}^{\tau+\Delta\,t}\int_{{D}}(g^{\alpha}-A^{\alpha\beta}_{ij}D_{\beta}u)D_{\alpha}\zeta\,dxdt.
\end{align*}
It is clear that
\begin{equation}\label{lem31.4}
\lim_{\Delta\,t\rightarrow{0}}\int_{{D}}(u(x,\tau+\Delta\,t)-u(x,\tau))\zeta(x)\,dx=0,\quad\mbox{for}~\zeta\in{H}_{0}^{1}({D}).
\end{equation}

For $0\leq\tau<\tau+\Delta{t}<T$, consider $\widetilde{u}$
(depending on $\tau$),
$$\widetilde{u}(x,t)=\begin{cases}
u(x,\tau+\Delta{t}),&t\geq\tau+\Delta{t},\\
u(x,t),&0\leq t\leq\tau+\Delta{t},\\
u(x,\tau),&t\leq\tau.
\end{cases}
$$
Similarly as the Steklov average, for $0<h<\Delta{t}$, we define
$$\widetilde{u}_{h}(x,t):=\frac{1}{2h}\int_{t-h}^{t+h}\widetilde{u}(x,s)ds.$$
Then
$\widetilde{u}_{h}\in\overset{\circ}{W}\,^{1,1}_{2}(Q_{\tau,\tau+\Delta{t}})$.
Take $\zeta=\widetilde{u}_{h}$ in \eqref{lem31.3}, and compute
\begin{align*}
&\int_{\tau}^{\tau+\Delta\,t}\int_{{D}}u(\widetilde{u}_{h})_{t}dxdt\\
&=\frac{1}{2h}\int_{{D}}dx\bigg\{\int_{\tau}^{\tau+\Delta\,t}u(x,t)\widetilde{u}(x,t+h)dt-\int_{\tau}^{\tau+\Delta\,t}u(x,t)\widetilde{u}(x,t-h)dt\bigg\}\\
&=\frac{1}{2h}\int_{{D}}dx\bigg\{\int_{\tau+\Delta\,t}^{\tau+\Delta\,t+h}u(x,t-h)\widetilde{u}(x,t)dt-\int_{\tau}^{\tau+h}u(x,t)\widetilde{u}(x,t-h)dt\bigg\}\\
&=\frac{1}{2h}\int_{{D}}dx\bigg\{\int_{\tau+\Delta\,t}^{\tau+\Delta\,t+h}u(x,t-h)u(x,\tau+\Delta\,t)dt-\int_{\tau}^{\tau+h}u(x,t)u(x,\tau)dt\bigg\}\\
&=\frac{1}{2h}\int_{{D}}dx\bigg\{\int_{\tau+\Delta\,t}^{\tau+\Delta\,t+h}\bigg(u(x,t-h)-u(x,\tau+\Delta\,t)\bigg)u(x,\tau+\Delta\,t)dt\\
&\quad\quad-\int_{\tau}^{\tau+h}\bigg(u(x,t)-u(x,\tau)\bigg)u(x,\tau)dt\bigg\}
+\frac{1}{2}\int_{{D}}u^{2}(x,\tau+\Delta\,t)dx-\frac{1}{2}\int_{{D}}u^{2}(x,\tau)dx.
\end{align*}
By \eqref{lem31.4}, it follows that
\begin{equation}\label{lem31.5}
\lim_{h\rightarrow
0}\int_{\tau}^{\tau+\Delta\,t}\int_{{D}}u(\widetilde{u}_{h})_{t}dxdt=
\frac{1}{2}\int_{{D}}u^{2}(x,\tau+\Delta\,t)dx-\frac{1}{2}\int_{{D}}u^{2}(x,\tau)dx.
\end{equation}
Furthermore, we have
$$\lim_{h\rightarrow0}\int_{{D}}u\widetilde{u}_{h}(x,\tau)dx=\int_{{D}}u^{2}(x,\tau)dx,
~\mbox{and}~\lim_{h\rightarrow0}\int_{{D}}u\widetilde{u}_{h}(x,\tau+\Delta{t})dx=\int_{{D}}u^{2}(x,\tau+\Delta{t})dx.$$
On the other hand,
$$\lim_{h\rightarrow 0}\int_{\tau}^{\tau+\Delta\,t}\int_{{D}}A^{\alpha\beta}_{ij}D_{\beta}uD_{\alpha}\widetilde{u}_{h}dxdt
=\int_{\tau}^{\tau+\Delta\,t}\int_{{D}}A^{\alpha\beta}_{ij}D_{\beta}uD_{\alpha}udxdt.$$
Indeed, it reduces to show
\begin{align*}
|I|&=\left|\int_{\tau}^{\tau+\Delta\,t}\int_{{D}}A^{\alpha\beta}_{ij}D_{\beta}u(D_{\alpha}\widetilde{u}_{h}-D_{\alpha}u)dxdt\right|\\
&=\left|\int_{\tau}^{\tau+\Delta\,t}\int_{{D}}A^{\alpha\beta}_{ij}D_{\beta}u
\left(\frac{1}{2h}\int_{t-h}^{t+h}(D_{\alpha}u(x,s)-D_{\alpha}u(x,t))ds\right)dxdt\right|\\
&=\left|\int_{\tau}^{\tau+\Delta\,t}\int_{{D}}A^{\alpha\beta}_{ij}(x,t)D_{\beta}u(x,t)
\left(\frac{1}{2}\int_{-1}^{1}\left(D_{\alpha}u(x,t+hs')-D_{\alpha}u(x,t)\right)ds'\right)dxdt\right|\\
&\leq\,C(\Lambda)\|Du\|_{L^{2}((\tau,\tau+\Delta\,t)\times{D})}\sup_{-h<s<h}\|Du(x,t+s)-Du(x,t)\|_{L^{2}((\tau,\tau+\Delta\,t)\times{D})}\\
&\rightarrow0\quad\mbox{as}\quad h\rightarrow0,\quad(\mbox{since}\ \
Du\in\,L^{2}(Q_{T})).
\end{align*}
Similarly, we have
$$\lim_{h\rightarrow 0}\iint g^{\alpha}D_{\alpha}\widetilde{u}_{h}=\iint g^{\alpha}D_{\alpha}u.$$
Then taking $h\rightarrow 0$ in identity \eqref{lem31.3} with
$\zeta=\widetilde{u}_{h}$, we have
\begin{equation}\label{lem31.6}
\frac{1}{2}\int_{{D}}u^{2}(x,\tau+\Delta\,t)dx-\frac{1}{2}\int_{{D}}u^{2}(x,\tau)dx
=\int_{\tau}^{\tau+\Delta\,t}\int_{{D}}\left(fu+(g^{\alpha}-A^{\alpha\beta}_{ij}D_{\beta}u)D_{\alpha}u\right)dxdt.
\end{equation}
This implies that
\begin{equation}\label{lem31.7}
\lim_{\Delta\,t\rightarrow0}\|u(\cdot,\tau+\Delta\,t)\|^{2}_{L^{2}({D})}=\|u(\cdot,\tau)\|^{2}_{L^{2}({D})}.
\end{equation}

Now we can prove the continuity of $u(\cdot,t)$ in $t$ in the norm
of $L^{2}({D})$. Indeed,
$$\|u(\cdot,\tau+\Delta\,t)-u(\cdot,t)\|^{2}_{L^{2}({D})}=\|u(\cdot,\tau+\Delta\,t)\|^{2}_{L^{2}({D})}
+\|u(x,\tau)\|^{2}_{L^{2}({D})}-2\int_{{D}}u(x,\tau+\Delta\,t)u(x,\tau)dx.$$
Then using \eqref{lem31.4} and \eqref{lem31.7}, it follows that
$$\lim_{\Delta\,t\rightarrow0}\|u(\cdot,\tau+\Delta\,t)-u(\cdot,\tau)\|^{2}_{L^{2}({D})}=0.$$
Hence $u\in\overset{\circ}{V}\,^{1,0}(Q_{T})$.
\end{proof}

\begin{lemma}\label{lem32}
Assume the above. Suppose $u\in\overset{\circ}{V}(Q_{T})$ is a weak
solution of problem \eqref{VPIND}, then we have
\begin{align}\label{lem32.1}
\|u\|^{2}_{V(Q_{T})}\leq\,C\bigg(\|\varphi\|^{2}_{L^{2}({D})}+\|f\|^{2}_{L^{2}(Q_{T})}+\|g\|^{2}_{L^{2}(Q_{T})}\bigg),
\end{align}
where $C$ depends only on $\lambda$ and $T$.
\end{lemma}

\begin{proof}
By Lemma \ref{lem31}, $u(x,t)\in\overset{\circ}{V}\,^{1,0}(Q_{T})$.
For $0<\tau<T$, let
$$\overline{u}(x,t)=\begin{cases}
u(x,\tau),&t\geq\tau,\\
u(x,t),&0\leq t\leq\tau,\\
u(x,0),&t\leq0,
\end{cases}
\quad\mbox{and}\quad\overline{u}_{h}(x,t):=\frac{1}{2h}\int_{t-h}^{t+h}\overline{u}(x,s)ds.$$
Replacing $\zeta$ by $\overline{u}_{h}$ in \eqref{weaksoln} gives
\begin{align}\label{lem32.2}
\int_{{D}}(u\overline{u}_{h})(\cdot,\tau)\,dx-\int_{{D}}\varphi\overline{u}_{h}(\cdot,0)\,dx
&-\int_{0}^{\tau}\int_{{D}}u(\overline{u}_{h})_{t}dxdt\nonumber\\
+\int_{0}^{\tau}\int_{{D}}A^{\alpha\beta}_{ij}D_{\beta}uD_{\alpha}(\overline{u}_{h})\,dxdt
&=\int_{0}^{\tau}\int_{{D}}\big(f\overline{u}_{h}+g^{\alpha}D_{\alpha}(\overline{u}_{h})\big)dxdt.
\end{align}
Similar to the derivation of \eqref{lem31.5}, we obtain
\begin{equation}\label{lem32.3}
\lim_{h\rightarrow0}\int_{0}^{\tau}\int_{{D}}u(\overline{u}_{h})_{t}dxdt=
\frac{1}{2}\int_{{D}}u^{2}(x,\tau)dx-\frac{1}{2}\int_{{D}}u^{2}(x,0)dx.
\end{equation}
In fact, the derivation is much easier since we now have
$u\in\overset{\circ}{V}\,^{1,0}(Q_{T})$. Since
$\|\overline{u}_{h}-u\|_{V(Q_{T})}\rightarrow 0$ as $h\rightarrow0$,
and combining with \eqref{lem32.3}, then sending $h\rightarrow 0$ in
\eqref{lem32.2}, it follows that
$$\frac{1}{2}\int_{{D}}u^{2}(x,\tau)dx+\int_{0}^{\tau}\int_{{D}}A^{\alpha\beta}_{ij}D_{\beta}uD_{\alpha}udxdt
=\int_{0}^{\tau}\int_{{D}}\bigg(fu+g^{\alpha}D_{\alpha}u\bigg)dxdt+\frac{1}{2}\int_{{D}}\varphi^{2}(x)dx.$$
Using the parabolic condition and Cauchy inequality, we have
\begin{align*}
&\frac{1}{2}\int_{{D}}u^{2}(x,\tau)dx+\lambda\int_{0}^{\tau}\int_{{D}}|Du|^{2}dxdt\\
&\leq\frac{\lambda}{2}\int_{0}^{\tau}\int_{{D}}|Du|^{2}dxdt+C\int_{0}^{\tau}\int_{{D}}|g|^{2}dxdt
\\
&+\frac{1}{2}\int_{0}^{\tau}\int_{{D}}|u|^{2}dxdt+\frac{1}{2}\int_{0}^{\tau}\int_{{D}}|f|^{2}dxdt
+\frac{1}{2}\int_{{D}}\varphi^{2}(x)dx,
\end{align*}
where $C$ depends only on $\lambda$. That is,
\begin{align}\label{lem32.4}
&\int_{{D}}u^{2}(x,\tau)dx+\int_{0}^{\tau}\int_{{D}}|Du|^{2}dxdt\nonumber\\
&\leq\,C\left(\int_{0}^{\tau}\int_{{D}}|u|^{2}dxdt+\int_{{D}}\varphi^{2}(x)dx
+\int_{0}^{\tau}\int_{{D}}|g|^{2}dxdt+\int_{0}^{\tau}\int_{{D}}|f|^{2}dxdt\right),
\end{align}
where $C$ depends only on $\lambda$. Denote
$E(\tau)=\int_{0}^{\tau}\int_{{D}}u^{2}dxdt$ and
$F(\tau)=\|\varphi\|^{2}_{L^{2}({D})}+\|g\|_{L^{2}(Q_{\tau})}+\|f\|_{L^{2}(Q_{\tau})}$,
where $Q_{\tau}=D\times(0,\tau)$. By \eqref{lem32.4},
$$\frac{dE(\tau)}{d\tau}\leq\,C\bigg(E(\tau)+F(\tau)\bigg),\quad0<\tau<T,$$
where $C$ depends only on $\lambda$. This implies that
$$E(\tau)\leq\,e^{C\tau}E(0)+e^{C\tau}F(\tau),\quad0<\tau<T.$$
Lemma \ref{lem32} follows from the above and \eqref{lem32.4}.
\end{proof}

Now we consider the $L^{2}$ estimate for $u_{t}$.

\begin{lemma}\label{lem33}
Under \eqref{coeff1} \eqref{coeff2} and \eqref{coeff31}, suppose
$u(x,t)\in\overset{\circ}{V}(Q_{T})$ is a weak solution of Problem
\eqref{VPIND}. Then for any small $\delta>0$, we have
\begin{align}\label{lem33.0}
\sup_{\delta\leq\tau\leq\,T}\int_{{D}}|&Du(x,\tau)|^{2}dx+\int_{\delta}^{T}\int_{{D}}|u_{t}|^{2}dxdt\nonumber\\
&\leq{C}\bigg(\|\varphi\|_{L^{2}({D})}^{2}
+\|f\|^{2}_{L^{2}(Q_{T})}+\sum_{s\leq1}\|D_{t}^{s}g\|^{2}_{L^{2}(Q_{T})}\bigg),
\end{align}
where $C$ depends only on $n,N,\lambda,\Lambda_{2},T$ and $\delta$.
\end{lemma}

\begin{proof}
{\bf{STEP 1.}} By Lemma \ref{lem31},
$u(x,t)\in\overset{\circ}{V}\,^{1,0}(Q_{T})$. Let $\tau_{0}\in(0,T)$
and $h\in(0,T-\tau_{0})$. Take $\zeta_{h}$ as the test function in
\eqref{weaksoln}, where $\zeta$ is an arbitrary element of
$\overset{\circ}{W}\,^{1,1}_{2}({D}\times(0,T+h))$ with
$\zeta(x,t)=0$ if $t\leq{h}$ and $t\geq{T}$. A simple calculation
shows that $(\zeta_{h})_{t}=(\zeta_{t})_{h}$, and hence
\begin{align*}
\int_{0}^{T}\int_{{D}}u(\zeta_{h})_{t}dxdt&=\int_{0}^{T}\int_{{D}}u(\zeta_{t})_{h}dxdt
=\int_{0}^{T}\int_{{D}}u_{-h}\zeta_{t}dxdt=-\int_{0}^{T}\int_{{D}}(u_{-h})_{t}\zeta\,dxdt,
\end{align*}
where the notation
$$\zeta_{-h}(x,t)=\frac{1}{h}\int_{t-h}^{t}\zeta(x,\tau)d\tau.$$
For all the other terms in \eqref{weaksoln} with $\zeta=\zeta_{h}$,
we also transfer the average $(\cdot)_{h}$ from $\zeta$ to the
coefficients, taking into account the permutability of this
averaging with differentiation with respect to $x$. This gives the
identity
\begin{align}\label{lem33.1}
\int_{h}^{T}\int_{{D}}(u_{-h})_{t}\zeta\,dxdt
+\int_{h}^{T}\int_{{D}}\big(A^{\alpha\beta}_{ij}D_{\beta}u^{j}\big)_{-h}D_{\alpha}\zeta\,dxdt
=\int_{h}^{T}\int_{{D}}(f_{-h}\zeta+g^{\alpha}_{-h}D_{\alpha}\zeta)dxdt,
\end{align}
for any $\zeta\in\overset{\circ}{W}\,^{1,1}_{2}(Q_{T+h})$ which
vanishes for $t\leq{h}$ and $t\geq{T}$. This identity is actually
valid for any $\zeta$ that is equal to zero for $t>\tau$
$(\tau\leq{T})$ and is equal to some function
$\hat{\zeta}\in\overset{\circ}{V}\,^{1,0}(Q_{\tau})$ for
$t\in[h,\tau]$. Indeed, the set
$\overset{\circ}{W}\,^{1,1}_{2}(Q_{T})$ is dense in the space
$\overset{\circ}{V}\,^{1,0}(Q_{T})$. Thus for any
$\hat{\zeta}(x,t)\in\overset{\circ}{V}\,^{1,0}(Q_{T})$ there is a
sequence of functions
$\zeta_{m}\in\overset{\circ}{W}\,^{1,1}_{2}(Q_{T})$, that is
strongly convergent to $\hat{\zeta}$ as $m\rightarrow\infty$ in the
norm of $\overset{\circ}{V}\,^{1,0}(Q_{T})$. We denote $\chi_{k}(t)$
the continuous piecewise-linear functions
$$\chi_{k}(t)=
\begin{cases}
kt,&\mbox{in}\ (0,\frac{1}{k});\\
1,&\mbox{in}\ [\frac{1}{k},\tau-\frac{1}{k}];\\
k(\tau-t),&\mbox{in}\ (\tau-\frac{1}{k},\tau);\\
0,&\mbox{in}\ (-\infty,0]\cup[\tau,+\infty).
\end{cases}
$$
Then identity \eqref{lem33.1} is established for
$\zeta_{m,k}=\zeta_{m}\chi_{k}$ for $\tau<T$. One can pass to the
limit as $m\rightarrow\infty$ and $k\rightarrow\infty$, thereby
\begin{equation}\label{lem33.2}
\int_{0}^{\tau}\int_{{D}}(u_{-h})_{t}\hat{\zeta}\,dxdt+\int_{0}^{\tau}\int_{{D}}(A^{\alpha\beta}_{ij}D_{\beta}u^{j})_{-h}D_{\alpha}\hat{\zeta}\,dxdt
=\int_{0}^{\tau}\int_{{D}}f_{-h}\zeta+g^{\alpha}_{-h}D_{\alpha}\hat{\zeta}\,dxdt,
\end{equation}
for any $\hat{\zeta}\in\overset{\circ}{V}\,^{1,0}(Q_{\tau})$,
$\tau\leq{T}$.

{\bf{STEP 2.}} For $h<\frac{\delta}{2}$, replacing $\hat{\zeta}$ by
$\overline{\eta}^{2}(u_{-h})_{t}$ in \eqref{lem33.2}, where
$\overline{\eta}(x,t)$ is a smooth cutoff function,
 satisfying $0\leq\overline{\eta}\leq1$, and for any $0<\delta<T$,
\begin{equation}\label{etabar}
\overline{\eta}(t)=
\begin{cases}1,&\mbox{in}~(\delta,\tau),\\
0,&\mbox{in}~(0,T)\setminus(\frac{\delta}{2},\tau),
\end{cases}\quad|\overline{\eta}_{t}|\leq\frac{C}{\delta},
\end{equation}
it follows that
\begin{align*}
&\int_{0}^{\tau}\int_{{D}}(u_{-h})_{t}\overline{\eta}^{2}(u_{-h})_{t}dxdt
+\int_{0}^{\tau}\int_{{D}}(A^{\alpha\beta}_{ij}D_{\beta}u^{j})_{-h}\overline{\eta}^{2}D_{\alpha}(u_{-h})_{t}dxdt\\
=&\int_{0}^{\tau}\int_{{D}}f_{-h}\overline{\eta}^{2}(u_{-h})_{t}+g^{\alpha}_{-h}\overline{\eta}^{2}D_{\alpha}(u_{-h})_{t}dxdt.
\end{align*}
Then we have
\begin{align*}
&LHS\\
=&\int_{0}^{\tau}\int_{{D}}|(u_{-h})_{t}|^{2}\overline{\eta}^{2}dxdt
-\frac{1}{2}\int_{0}^{\tau}\int_{{D}}\left((A^{\alpha\beta}_{ij})_{-h}\right)_{t}D_{\beta}u^{j}(x,t-h)\overline{\eta}^{2}D_{\alpha}u^{i}_{-h}dxdt\\
&+\frac{1}{2}\int_{{D}}(A^{\alpha\beta}_{ij}D_{\beta}u^{j})_{-h}\overline{\eta}^{2}D_{\alpha}u^{i}_{-h}dx\bigg|_{t=0}^{t=\tau}
-\int_{0}^{\tau}\int_{{D}}\left(A^{\alpha\beta}_{ij}D_{\beta}u^{j}\right)_{-h}\overline{\eta}_{t}\overline{\eta}D_{\alpha}u^{i}_{-h}\\
=&\int_{0}^{\tau}\int_{{D}}|u_{ht}|^{2}\overline{\eta}^{2}dxdt
-\frac{1}{2}\int_{0}^{\tau}\int_{{D}}\left((A^{\alpha\beta}_{ij})_{-h}\right)_{t}D_{\beta}u^{j}(x,t-h)\overline{\eta}^{2}D_{\alpha}u^{i}_{-h}dxdt\\
&+\frac{1}{2}\int_{{D}}\left((A^{\alpha\beta}_{ij}D_{\beta}u^{j})_{h}\overline{\eta}^{2}D_{\alpha}u^{i}_{h}\right)(x,\tau)dxdt
-\int_{0}^{\tau}\int_{{D}}\left(A^{\alpha\beta}_{ij}D_{\beta}u^{j}\right)_{-h}\overline{\eta}_{t}\overline{\eta}D_{\alpha}u^{i}_{-h},
\end{align*}
here using $\overline{\eta}(0)=0$, and similarly,
\begin{align*}
RHS&=\int_{0}^{\tau}\int_{{D}}f_{-h}\overline{\eta}^{2}(u_{-h})_{t}+\int_{{D}}\left(\overline{\eta}^{2}g^{\alpha}_{-h}D_{\alpha}u_{-h}\right)(x,\tau)dx\\
&-\int_{0}^{\tau}\int_{{D}}(g^{\alpha}_{-h})_{t}\overline{\eta}^{2}D_{\alpha}u_{-h}dxdt
-2\int_{0}^{\tau}\int_{{D}}g^{\alpha}_{-h}\overline{\eta}_{t}\overline{\eta}D_{\alpha}u_{-h}dxdt.
\end{align*}
Combining them, by Cauchy inequality, and sending $h\rightarrow0$,
we have
\begin{align*}
&\frac{1}{2}\int_{0}^{\tau}\int_{{D}}\overline{\eta}^{2}|u_{t}|^{2}dxdt
+\frac{1}{2}\int_{{D}}\left(A^{\alpha\beta}_{ij}\overline{\eta}D_{\beta}u^{j}\overline{\eta}D_{\alpha}u^{i}\right)(x,\tau)dx\\
&\leq\int_{0}^{\tau}\int_{{D}}(D_{t}A^{\alpha\beta}_{ij})(\overline{\eta}D_{\beta}u^{j})(\overline{\eta}D_{\alpha}u^{i})dxdt
+\frac{1}{2}\int_{0}^{\tau}\int_{{D}}|\overline{\eta}f|^{2}dxdt\\
&+\int_{{D}}\left(\overline{\eta}^{2}g^{\alpha}D_{\alpha}u\right)(x,\tau)dx
-\int_{0}^{\tau}\int_{{D}}\overline{\eta}^{2}(g^{\alpha})_{t}D_{\alpha}udxdt
-2\int_{0}^{\tau}\int_{{D}}\overline{\eta}_{t}\overline{\eta}g^{\alpha}D_{\alpha}udxdt.
\end{align*}
Then using the parabolic condition and Cauchy inequality, we obtain
\begin{align}\label{lem4-1}
&\int_{\delta}^{\tau}\int_{{D}}|u_{t}|^{2}dxdt
+\lambda\int_{{D}}|Du(x,\tau)|^{2}dx\nonumber\\
&\leq\,C(n,N,\Lambda_{2})\int_{0}^{\tau}\int_{{D}}|Du|^{2}dxdt
+\int_{0}^{\tau}\int_{{D}}|f|^{2}dxdt\nonumber\\
&+C(\lambda)\int_{{D}}|g(x,\tau)|^{2}dx+C\int_{0}^{\tau}\int_{{D}}|g|^{2}dxdt
+C\int_{0}^{\tau}\int_{{D}}|g_{t}|^{2}dxdt.
\end{align}
Since
\begin{align*}
&\|g(x,\tau)\|^{2}_{L^{2}({D})}-\frac{1}{\tau}\int_{0}^{\tau}\|g(x,t)\|^{2}_{L^{2}({D})}dt\\
=&\frac{1}{\tau}\int_{0}^{\tau}\int_{{D}}|g(x,\tau)|^{2}-|g(x,t)|^{2}dxdt\\
=&\frac{1}{\tau}\int_{0}^{\tau}\int_{{D}}\int_{t}^{\tau}D_{s}\left(g(x,s)\right)^{2}dsdxdt\\
\leq&2\int_{0}^{\tau}\int_{{D}}|g(x,t)g_{t}(x,t)|dxdt\\
\leq&\int_{0}^{\tau}\int_{{D}}|g(x,t)|^{2}dxdt+\int_{0}^{\tau}\int_{{D}}|g_{t}(x,t)|^{2}dxdt,
\end{align*}
it follows that
$$\|g(x,\tau)\|^{2}_{L^{2}({D})}\leq\,C\bigg(\int_{0}^{\tau}\int_{{D}}|g(x,t)|^{2}dxdt+\int_{0}^{\tau}\int_{{D}}|g_{t}(x,t)|^{2}dxdt\bigg),$$
where $C$ depends only on $\delta$. Substituting it into
\eqref{lem4-1} gives
\begin{align}\label{lem4-2}
&\sup_{\delta\leq\tau\leq{T}}\int_{{D}}|Du(x,\tau)|^{2}dx+\int_{\delta}^{\tau}\int_{{D}}|u_{t}|^{2}dxdt\nonumber\\
&\leq\,C\bigg(\int_{0}^{\tau}\int_{{D}}|Du|^{2}dxdt+\int_{0}^{\tau}\int_{{D}}|f|^{2}dxdt+\int_{0}^{\tau}\int_{{D}}\left(|g|^{2}+|g_{t}|^{2}\right)dxdt\bigg),
\end{align}
where $C$ depends $n,N,\lambda,\Lambda_{2},\delta$ and $\tau$. In
combination with Lemma \ref{lem32}, \eqref{lem33.0} is established.
\end{proof}

Since $A,f,g$ are smooth in $t$, it is standard to use difference
quotients in $t$ to estimate higher derivatives.

\begin{lemma}\label{lem34}
Under \eqref{coeff1} \eqref{coeff2} and \eqref{coeff31}, if
$u\in{V}(Q_{T})$ is a weak solution of problem \eqref{VPIND}, then
for any $0<\delta<T$ and $0<\epsilon<1$, we have
\begin{align}\label{lem34-2}
&\sup_{\delta\leq\tau\leq\,T}\sum_{|\gamma'|\leq\,k}\int_{{D}_{\epsilon}}|u_{t}(x,\tau)|^{2}dx
+\sum_{|\gamma'|\leq\,k}\int_{\delta}^{T}\int_{{D}_{\epsilon}}|Du_{t}|^{2}dxdt\nonumber\\
&\leq\,C\bigg(\|\varphi\|_{L^{2}({D})}^{2}
+\sum_{s\leq1}\|D_{t}^{s}f\|_{L^{2}(Q_{T})}^{2}
+\sum_{s\leq1}\|D_{t}^{s}g\|^{2}_{L^{2}({D}_{m}\times(0,T))}\bigg),
\end{align}\marginnote{Delete the two $\sum$.}
where $C$ depends only on $n,N,\lambda,\Lambda_{2},T,\epsilon$, and
$\delta$.
\end{lemma}

\begin{proof}
For $h<\delta$, we apply $\triangle^{h}_{t}$ to \eqref{VPIND}, thus
\begin{align}\label{lem34-1}
\Delta_{t}^{h}u_{t}&-D_{\alpha}(A^{\alpha\beta}_{ij}D_{\beta}\Delta_{t}^{h}u^{j})
=D_{\alpha}\bigg(\left(\Delta_{t}^{h}A^{\alpha\beta}_{ij}\right)(D_{\beta}u^{j})-\Delta_{t}^{h}g^{\alpha\,i}\bigg)+\Delta_{t}^{h}f^{i},
\end{align}
where
$$\Delta_{t}^{h}u^{i}=\frac{u^{i}(x,t)-u^{i}(x,t-h)}{h}.$$
Similarly as in Lemma \ref{lem32}, multiplying
$\overline{\eta}^{2}\Delta_{t}^{h}u$ on both sides of
\eqref{lem34-1}, where $\overline{\eta}$ is defined by
\eqref{etabar}, integrating by parts, and using the property of
difference quotients, we have
\begin{align*}
\sup_{\delta\leq\tau\leq\,T}\int_{{D}}&|u_{t}(x,\tau)|^{2}dx
+\int_{\delta}^{T}\int_{{D}}|Du_{t}|^{2}dxdt\nonumber\\
&\leq\,C\bigg(\|Du\|_{L^{2}(Q_{T})}^{2}+\|u_{t}\|_{L^{2}(Q_{T})}^{2}
+\|D_{t}f\|_{L^{2}(Q_{T})}^{2} +\|D_{t}g\|^{2}_{L^{2}(Q_{T})}\bigg),
\end{align*}
where $C$ depends only on $n,N,\lambda,\Lambda_{2}$ and $\delta$. In
combination with \eqref{lem32.1} and \eqref{lem33.0}, we established
\eqref{lem34-2}.
\end{proof}

\bigskip
\section{Estimates for Laminar Systems}\label{sec3}
\bigskip

In this section, we extend results for laminar elliptic systems due
to Chipot, Kinderlehrer, and Vergara-Caffarelli \cite{ck} to laminar
parabolic systems. In subsection \ref{subibvp}, we first consider
the initial boundary value problem of laminar parabolic systems.
Then in subsection \ref{subsD}, we give the interior estimates for
equations \eqref{IBVP} in general domain $D\times(0,T)$, similarly
as in section \ref{sec2}. In subsection \ref{subslaminar}, we
consider the laminar parabolic systems in $\omega\times(0,T)$.

\subsection{Estimates for the Initial Boundary Value
Problem}\label{subibvp}

Let $D$ be the unit cube $\omega$,
$$\omega=\bigg\{x\in\mathbb{R}^{n}:|x_{i}|<\frac{1}{2}, 1\leq i\leq n\bigg\},$$
divided into $\omega_{m}$. However, the $\omega_{m}$ are different;
they are strips:
 $$\omega_{m}=\bigg\{x\in\omega:c_{m-1}<x_{n}<c_{m}\bigg\},$$
where the $c_{m}$ are increasing constants lying between
$-\frac{1}{2}$ and $\frac{1}{2}$. There may be infinitely many
strips; if so, we set $c_{-\infty}=-\frac{1}{2}$ and
$c_{\infty}=\frac{1}{2}$. We consider the following initial boundary
value problem of the laminar parabolic systems
\begin{equation}\label{IBVP0}
\left\{
  \begin{array}{ll}
    (u^{i})_{t}-D_{\alpha}\left(A_{ij}^{\alpha\beta}(x,t)D_{\beta}u^{j}\right)=-D_{\alpha}g^{\alpha\,i}+f^{i} & \hbox{in}\ \omega\times(0,T),\\
    u=0 & \hbox{on}\ \partial\omega\times(0,T),\\
    u=\varphi(x) & \hbox{on}\ \omega\times\{0\}.
  \end{array}
\right.
\end{equation}
Here $g,f$ are smooth in each $\overline{\omega}_{m}\times[0,T]$,
and $\varphi\in C^{\infty}(\overline{\omega}_{m})$. We assume that
$A^{\alpha\beta}_{ij}(x,t)$ satisfy \eqref{coeff1} \eqref{coeff2},
and further, for any nonnegative integer $r,s$,
\begin{equation}\label{coeff3}
\sum_{r+2s\leq{l}}|D_{x}^{r}D_{t}^{s}A_{ij}^{\alpha\beta}(x,t)|\leq\Lambda_{l},\quad\forall~
(x,t)\in\omega_{m}\times(0,T),
\end{equation}
where
$\Lambda=\Lambda_{0}\leq\Lambda_{1}\leq\cdots\leq\Lambda_{l}\leq\cdots$.

\begin{prop}\label{prop1}
Assume the above. Let $u\in\overset{\circ}{V}(\Omega_{T})$ be a weak
solution of \eqref{IBVP0}. Then for $0<\delta<T$ and for all
$\gamma'$, $D_{x'}^{\gamma'}u\in\,C^{0}(\omega\times(\delta,T))$,
and for each $m$, and any $0<\epsilon<1$,
$u\in\,C^{\infty}\left(\left(\omega_{\epsilon}\cap\overline{\omega}_{m}\right)\times(\delta,T)\right)$.
Moreover for any nonnegative integer $k$, and any $m$,
\begin{align}\label{prop1-1}
&\sum_{r+2s\leq{k}}\|D_{x}^{r}D_{t}^{s}u\|_{L^{\infty}\left(\left(\overline{\omega}_{m}\cap\omega_{\epsilon}\right)\times(\delta,T)\right)}\nonumber\\
&\leq\,C\left(\|\varphi\|_{L^{2}(\omega)}
+\sum_{|\gamma'|\leq\overline{k}+k+1,s\leq{\frac{k}{2}+1}}\left(\|D_{x'}^{\gamma'}D_{t}^{s}g\|_{L^{2}(\Omega_{T})}
+\|D_{x'}^{\gamma'}D_{t}^{s}f\|_{L^{2}(\Omega_{T})}\right)\right),
\end{align}
where $\overline{k}=[\frac{n-1}{2}]+1$ and $C$ depends only on
$\epsilon,\delta,k,n,N,\lambda,T$ and $\Lambda_{\overline{k}+k+2}$.
\end{prop}

Our proof of Proposition \ref{prop1} adapts the alternative proof of
Li-Nirenberg in the elliptic case to the parabolic systems. For
laminar systems, we could establish the estimates in
$\omega\times(\delta,T)$, as well as in Lemmma \ref{lem32}-Lemma
\ref{lem34}, just replacing the domain $D$ by $\omega$. Next, we
will establish the interior estimates for higher derivatives of weak
solutions.

Denote
$$\Omega_{T}=\omega\times(0,T),\quad\mbox{and}\quad\omega_{\epsilon}=\left\{x\in\omega~\big|~\emph{dist}(x,\partial\omega)>\epsilon\right\}, \quad\mbox{for}\quad 0<\epsilon<1.$$
Then

\begin{lemma}\label{lem42}
Under \eqref{coeff1} \eqref{coeff2} and \eqref{coeff3}, if
$u\in{V}(\Omega_{T})$ is a weak solution of problem \eqref{IBVP0},
then for any $0<\epsilon<1$ and integer $k\geq0$, for
$0\leq|\gamma'|\leq{k}$,
$D_{x'}^{\gamma'}u\in{W}^{1,0}_{2}(\Omega_{T})$, and we have
\begin{align}\label{lem42-1}
&\sup_{0\leq\tau\leq{T}}\int_{\omega_{\epsilon}}|D^{\gamma'}_{x'}u(x,\tau)|^{2}dx
+\sum_{|\gamma'|\leq\,k}\int_{0}^{T}\int_{\omega_{\epsilon}}|DD_{x'}^{\gamma'}u(x,t)|^{2}dxdt\nonumber\\
&\hspace{1cm}\leq\,C\bigg\{\sum_{|\gamma'|\leq\,k}\|D_{x'}^{\gamma'}\varphi\|_{L^{2}(\omega)}^{2}
+\sum_{|\gamma'|\leq\,k}\|D_{x'}^{\gamma'}f\|^{2}_{L^{2}(\Omega_{T})}+\sum_{|\gamma'|\leq\,k}\|D_{x'}^{\gamma'}g\|^{2}_{L^{2}(\Omega_{T})}\bigg\},
\end{align}
where $C$ depends only on $n,N,\lambda,\Lambda_{k},T,k$, and
$\epsilon$.
\end{lemma}

In order to estimate higher derivative, it is customary to
differentiate the equation, to multiply by a suitable derivative of
$u$ and by a cutoff function, and to integrate by parts. Clearly, we
are not allowed to apply $D_{n}$ across $\{x_{n}=c_{m}\}$ since the
coefficients are smooth only in $x'=(x_{1},\cdots,x_{n-1})$.
Furthermore, we do not know yet that $u$ has additional derivatives
in the $x'$ directions. So in place of taking derivatives, it is
standard to use difference quotients in these directions.

\begin{proof}[Proof of Lemma \ref{lem42}.]
For $k=0$, the estimate \eqref{lem42-1} is done by Lemma \ref{lem32}
for $D=\omega$. For $|\gamma'|=1$, denote the difference quotient in
$x_{s}$-direction $(s=1,\cdots,n-1)$ as
$$\Delta_{\iota}^{s}u(x,t)=\frac{u(x+\iota{e}_{s},t)-u(x,t)}{\iota}.$$
Taking the difference quotient $\Delta_{\iota}^{s}$ to the equation
in \eqref{IBVP0}, we obtain
\begin{align}\label{lem42.1}
\Delta_{\iota}^{s}u_{t}-D_{\alpha}\left(A_{ij}^{\alpha\beta}(x,t)D_{\beta}\Delta_{\iota}^{s}u^{j}\right)
=-D_{\alpha}\bigg(\Delta_{\iota}^{s}g^{\alpha\,i}-(\Delta_{\iota}^{s}A_{ij}^{\alpha\beta})(x+\iota{e}_{s},t)D_{\beta}u^{j}\bigg)+\Delta_{\iota}^{s}f^{i}.
\end{align}
Multiply by
$\zeta=(\overline{\tilde{\eta}^{2}\Delta_{\iota}^{s}u})_{h}$, where
$\tilde{\eta}(x)\in C_{0}^{\infty}(\omega)$ is a cutoff function
satisfying $0\leq\tilde{\eta}(x)\leq 1$, and for any $\epsilon>0$,
\begin{equation}\label{etatilde}
\tilde{\eta}(x)=
\begin{cases}
1&\mbox{in}\ \omega_{\epsilon},\\
0&\mbox{outside}\ \omega_{\frac{\epsilon}{2}},
\end{cases}
\quad\quad|D\tilde{\eta}|\leq\frac{C(n)}{\epsilon}.
\end{equation}
Then, for $0\leq|\iota|\leq\frac{\epsilon}{4}$, integrating by
parts, we obtain
\begin{align*}
&\int_{\omega}\Delta_{\iota}^{s}u(\overline{\tilde{\eta}^{2}\Delta_{\iota}^{s}u})_{h}(\cdot,\tau)\,dx
-\int_{\omega}\Delta_{\iota}^{s}u(\overline{\tilde{\eta}^{2}\Delta_{\iota}^{s}u})_{h}(\cdot,0)\,dx\\
&-\int_{0}^{\tau}\int_{\omega}(\Delta_{\iota}^{s}u)\bigg((\overline{\tilde{\eta}^{2}\Delta_{\iota}^{s}u})_{h}\bigg)_{t}dxdt
+\int_{0}^{\tau}\int_{\omega}A_{ij}^{\alpha\beta}(x,t)D_{\beta}\Delta_{\iota}^{s}u^{j}
D_{\alpha}(\overline{\tilde{\eta}^{2}\Delta_{\iota}^{s}u^{i}})_{h}dxdt\\
&=\int_{0}^{\tau}\int_{\omega}\Delta_{\iota}^{s}f^{i}(\overline{\tilde{\eta}^{2}\Delta_{\iota}^{s}u^{i}})_{h}dxdt\\
&+\int_{0}^{\tau}\int_{\omega}\bigg(\Delta_{\iota}^{s}g^{\alpha\,i}
-(\Delta_{\iota}^{s}A_{ij}^{\alpha\beta})(x+\iota\,e_{s},t)D_{\beta}u^{j}\bigg)D_{\alpha}(\overline{\tilde{\eta}^{2}\Delta_{\iota}^{s}u^{i}})_{h}dxdt.
\end{align*}
Similarly as in the proof of Lemma \ref{lem31}, sending
$h\rightarrow 0$, and making use of \eqref{coeff1} \eqref{coeff2}
and \eqref{coeff3}, we have
\begin{align*}
&\sup_{0\leq\tau\leq{T}}\int_{\omega}\tilde{\eta}^{2}|\Delta^{s}_{\iota}u(x,\tau)|^{2}dx
+\int_{0}^{T}\int_{\omega}\tilde{\eta}^{2}|D\Delta^{s}_{\iota}u(x,t)|^{2}dxdt\\
&\hspace{2cm}\leq\,C\bigg\{\|Du\|^{2}_{L^{2}(\Omega_{T})}+\|\Delta^{s}_{\iota}\varphi\|^{2}_{L^{2}(\omega)}
+\|\Delta^{s}_{\iota}f\|^{2}_{L^{2}(\Omega_{T})}+\|\Delta^{s}_{\iota}g\|^{2}_{L^{2}(\Omega_{T})}\bigg\},
\end{align*}
where $C$ depends only on $n,N,\lambda,\Lambda_{1},\epsilon$. So
that
\begin{align}\label{lem42.3}
&\sup_{0\leq\tau\leq{T}}\int_{\omega}\tilde{\eta}^{2}|D_{x_{s}}u(x,\tau)|^{2}dx+\int_{0}^{T}\int_{\omega}\tilde{\eta}^{2}|DD_{x_{s}}u(x,t)|^{2}dxdt\nonumber\\
&\hspace{2cm}\leq\,C\bigg\{\|Du\|^{2}_{L^{2}(\Omega_{T})}+\|D_{x_{s}}\varphi\|^{2}_{L^{2}(\omega)}
+\|D_{x_{s}}f\|^{2}_{L^{2}(\Omega_{T})}+\|D_{x_{s}}g\|^{2}_{L^{2}(\Omega_{T})}\bigg\}.
\end{align}
Hence, for any $\gamma'$, $|\gamma'|=1$, combining with Lemma
\ref{lem32}, we obtain \eqref{lem42-1} for $k=1$. For the general
$k$, we can make use of further differentiation in the $x'$
direction, and obtain the estimates \eqref{lem42-1} by induction.
\end{proof}

\begin{lemma}\label{lem43}
Under \eqref{coeff1} \eqref{coeff2} and \eqref{coeff3}, if
$u\in{V}(\Omega_{T})$ is a weak solution of problem \eqref{IBVP0},
then for any $0<\epsilon<1$, $0<\delta<T$, and integer $k\geq0$, for
$0\leq|\gamma'|\leq{k}$,
$D_{x'}^{\gamma'}u\in{W}_{2}^{1,0}(\omega_{\epsilon}\times(\delta,T))$,
and we have
\begin{align}\label{lem43-0}
&\sum_{|\gamma'|\leq{k}}\sup_{\delta\leq\tau\leq{T}}\int_{\omega_{\epsilon}}|D^{\gamma'}_{x'}u(x,\tau)|^{2}dx
+\sum_{|\gamma'|\leq{k}}\int_{\delta}^{T}\int_{\omega_{\epsilon}}|DD_{x'}^{\gamma'}u|^{2}dxdt\nonumber\\
&\hspace{1cm}\leq\,C\bigg(\|\varphi\|_{L^{2}(\omega)}^{2}
+\sum_{|\gamma'|\leq{k}}\|D^{\gamma'}_{x'}f\|_{L^{2}(\Omega_{T})}^{2}+\sum_{|\gamma'|\leq{k}}\|D^{\gamma'}_{x'}g\|_{L^{2}(\Omega_{T})}^{2}\bigg),
\end{align}
where $C$ depends only on $n,N,\lambda,\Lambda_{k},T,k$, $\epsilon$
and $\delta$.
\end{lemma}

\begin{proof}
For $k=0$, \eqref{lem43-0} has been established in Lemma \ref{lem32}
with $D=\omega$. For any $0<\delta<T$, by Lemma \ref{lem32}, we
have, for $|\gamma'|=1$,
$$\int_{\frac{\delta}{10}}^{\frac{\delta}{9}}\int_{\omega}|D_{x'}u(x,t)|^{2}dxdt
\leq\,C\bigg(\|\varphi\|^{2}_{L^{2}(\omega)}+\|f\|^{2}_{L^{2}(\Omega_{T})}+\|g\|^{2}_{L^{2}(\Omega_{T})}\bigg).$$
Then by Fubini theorem, there exist
$\overline{t}_{1}\in(\frac{\delta}{10},\frac{\delta}{9})$ such that
\begin{align}\label{lem43.1}
\int_{\omega}|D_{x'}u(x,\overline{t}_{1})|^{2}dx
\leq\,C\bigg(\|\varphi\|^{2}_{L^{2}(\omega)}+\|f\|^{2}_{L^{2}(\Omega_{T})}+\|g\|^{2}_{L^{2}(\Omega_{T})}\bigg).
\end{align}
Consider \eqref{lem42.1} in $\omega\times(\overline{t}_{1},T)$, by
the same process, we obtain, similar as \eqref{lem42.3},
\begin{align*}
&\sup_{\overline{t}_{1}\leq\tau\leq{T}}\int_{\omega}\tilde{\eta}^{2}|D_{x_{s}}u(x,\tau)|^{2}dx
+\int_{\overline{t}_{1}}^{T}\int_{\omega}\tilde{\eta}^{2}|DD_{x_{s}}u(x,t)|^{2}dxdt\nonumber\\
&\leq\,C\bigg\{\|Du\|^{2}_{L^{2}(\Omega_{T})}+\|D_{x_{s}}u(x,\overline{t}_{1})\|^{2}_{L^{2}(\omega)}
+\|D_{x_{s}}f\|^{2}_{L^{2}(\Omega_{T})}+\|D_{x_{s}}g\|^{2}_{L^{2}(\Omega_{T})}\bigg\}.
\end{align*}
where $C$ depends only on $n,N,\lambda,\Lambda_{1}$, $\epsilon$ and
$\delta$. Then combining with \eqref{lem43.1} and \eqref{lem32.1},
we proved Lemma \ref{lem43} for $k=1$, and
\begin{align}\label{lem43.3}
\sum_{|\gamma'|\leq{1}}\sup_{\overline{t}_{1}\leq\tau\leq{T}}\int_{\omega_{\epsilon}}&|D^{\gamma'}_{x'}u(x,\tau)|^{2}dx
+\sum_{|\gamma'|\leq{1}}\int_{\overline{t}_{1}}^{T}\int_{\omega_{\epsilon}}|DD_{x'}^{\gamma'}u|^{2}dxdt\nonumber\\
&\leq\,C\bigg(\|\varphi\|_{L^{2}(\omega)}^{2}
+\sum_{|\gamma'|\leq{1}}\|D^{\gamma'}_{x'}f\|_{L^{2}(\Omega_{T})}^{2}+\sum_{|\gamma'|\leq{1}}\|D^{\gamma'}_{x'}g\|_{L^{2}(\Omega_{T})}^{2}\bigg),
\end{align}
where $C$ depends only on $n,N,\lambda,\Lambda_{1},T$, $\epsilon$
and $\delta$. Then by \eqref{lem43.3}, we have for $|\gamma'|=2$,
$$\int_{\frac{\delta}{9}}^{\frac{\delta}{8}}\int_{\omega}|D^{\gamma'}_{x'}u(x,t)|^{2}dxdt
\leq\,C\bigg(\|\varphi\|^{2}_{L^{2}(\omega)}+\sum_{|\gamma'|\leq{1}}\|D^{\gamma'}_{x'}f\|_{L^{2}(\Omega_{T})}^{2}
+\sum_{|\gamma'|\leq{1}}\|D^{\gamma'}_{x'}g\|_{L^{2}(\Omega_{T})}^{2}\bigg).$$
By Fubini theorem, there exists
$\overline{t}_{2}\in(\frac{\delta}{9},\frac{\delta}{8})$. Repeating
the above process in $\omega\times(\overline{t}_{2},T)$, we obtain
the estimate \eqref{lem43-0} for $k=2$. For further $k>2$, we can
establish \eqref{lem43-0} by induction.
\end{proof}

To save space and the reader's patience, in the following we shall
simply differentiate the equation in place of taking difference
quotients in $x'$ to obtain higher regularity.

\begin{lemma}\label{lem44}
Under \eqref{coeff1} \eqref{coeff2} and \eqref{coeff3}, if
$u\in{V}(\Omega_{T})$ is a weak solution of problem \eqref{IBVP0},
then for any $0<\delta<T$ and $0<\epsilon<1$, we have, for
$|\gamma'|\leq{k}$, and $l\geq0$,
\begin{align}\label{lem44-1}
&\sum_{|\gamma'|\leq\,k,s\leq{l}-1}\sup_{\delta\leq\tau\leq{T}}\int_{\omega_{\epsilon}}|DD_{x'}^{\gamma'}D_{t}^{s}u(x,\tau)|^{2}dx+
\sum_{|\gamma'|\leq\,k,s\leq{l}}\int_{\delta}^{T}\int_{\omega_{\epsilon}}|D_{x'}^{\gamma'}D_{t}^{s}u|^{2}dxdt\nonumber\\
&\leq\,C\bigg(\|\varphi\|_{L^{2}(\omega)}^{2}
+\sum_{|\gamma'|\leq\,k,s\leq{l-1}}\|D_{x'}^{\gamma'}f\|_{L^{2}(\Omega_{T})}^{2}
+\sum_{|\gamma'|\leq\,k,s\leq{l}}\|D_{x'}^{\gamma'}D_{t}^{s}g\|^{2}_{L^{2}(\Omega_{T})}\bigg),
\end{align}
where $C$ depends only on
$n,N,\lambda,\Lambda_{k+2l+2},T,k,\epsilon$ and $\delta$.
\begin{align}\label{lem44-2}
&\sup_{\delta\leq\tau\leq\,T}\sum_{|\gamma'|\leq\,k,s\leq{l}}\int_{\omega_{\epsilon}}|D_{x'}^{\gamma'}D_{t}^{s}u(x,\tau)|^{2}dx
+\sum_{|\gamma'|\leq\,k,s\leq{l}}\int_{\delta}^{T}\int_{\omega_{\epsilon}}|DD_{x'}^{\gamma'}D_{t}^{s}u|^{2}dxdt\nonumber\\
&\leq\,C\bigg(\|\varphi\|_{L^{2}(\omega)}^{2}
+\sum_{|\gamma'|\leq\,k,s\leq{l}}\|D_{x'}^{\gamma'}D_{t}^{s}f\|_{L^{2}(\Omega_{T})}^{2}
+\sum_{|\gamma'|\leq\,k,s\leq{l}}\|D_{x'}^{\gamma'}D_{t}^{s}g\|^{2}_{L^{2}(\omega_{m}\times(0,T))}\bigg),
\end{align}
where $C$ depends only on
$n,N,\lambda,\Lambda_{k+2l+2},T,k,\epsilon$ and $\delta$.
\end{lemma}

\begin{proof}
For any $\gamma'$, $|\gamma'|=1$, applying $D_{x'}^{\gamma'}$ on
both sides of the equation in \eqref{IBVP0}, we have
\begin{equation*}
D_{x'}^{\gamma}u_{t}-D_{\alpha}(A^{\alpha\beta}_{ij}D_{\beta}D_{x'}^{\gamma'}u^{j})
=D_{\alpha}(D_{x'}^{\gamma'}A^{\alpha\beta}_{ij}D_{\beta}u^{j}-g^{\alpha\,i})+D_{x'}^{\gamma'}f.
\end{equation*}
Similarly as the process to prove Lemma \ref{lem33}, in virtue of a
cutoff function, defined by \eqref{etatilde}, instead of
\eqref{lem4-2}, we obtain
\begin{align*}
&\sum_{|\gamma'|\leq\,k}\sup_{\delta\leq\tau\leq{T}}\int_{\omega_{\epsilon}}|DD_{x'}^{\gamma'}u(x,\tau)|^{2}dx+
\sum_{|\gamma'|\leq\,k}\int_{\delta}^{T}\int_{\omega_{\epsilon}}|D_{x'}^{\gamma'}u_{t}|^{2}dxdt\nonumber\\
&\leq\,C\bigg(\sum_{|\gamma'|\leq\,k}\int_{\frac{\delta}{2}}^{T}\int_{\omega_{\frac{\epsilon}{2}}}|DD_{x'}^{\gamma'}u|dxdt
+\sum_{|\gamma'|\leq\,k}\|D_{x'}^{\gamma'}f\|_{L^{2}(\Omega_{T})}^{2}
+\sum_{|\gamma'|\leq\,k,s\leq1}\|D_{x'}^{\gamma'}D_{t}^{s}g\|^{2}_{L^{2}(\Omega_{T})}\bigg),
\end{align*}
where $C$ depends only on $n,N,\lambda,\Lambda_{k+2},k,T$,
$\epsilon$ and $\delta$. Combining with Lemma \ref{lem43}, we know
$D_{x'}^{\gamma'}u\in\,W^{1,1}_{2}(\omega_{m}\cap\omega_{\epsilon}\times(0,T))$,
and we have \eqref{lem44-1}.

Further applying $D_{x'}^{\gamma'}$ to \eqref{lem34-1}, we have
\begin{align}\label{lem44-6}
D_{x'}^{\gamma'}\Delta_{t}^{h}u_{t}&-D_{\alpha}(A^{\alpha\beta}_{ij})D_{\beta}D_{x'}^{\gamma'}\Delta_{t}^{h}u^{j})
=D_{\alpha}\left(\left(\Delta_{t}^{h}A^{\alpha\beta}_{ij}\right)D_{\beta}D_{x'}^{\gamma'}u^{j}\right)\nonumber\\
&+D_{\alpha}\bigg(\left(\Delta_{t}^{h}D_{x'}^{\gamma'}A^{\alpha\beta}_{ij}\right)D_{\beta}u^{j}
+\left(D_{x'}^{\gamma'}A^{\alpha\beta}_{ij}\right)\left(\Delta_{t}^{h}D_{\beta}u^{j}\right)-\Delta_{t}^{h}g^{\alpha\,i}\bigg)+D_{x'}^{\gamma'}\Delta_{t}^{h}f^{i},
\end{align}
Multiplying $\widetilde{\eta}^{2}D_{x'}^{\gamma'}\Delta_{t}^{h}u$ on
both sides of \eqref{lem44-6}, integrating by parts, we have
\eqref{lem44-2}. For general $k$ and $l$, we can obtain
\eqref{lem44-1} and \eqref{lem44-2} by induction.
\end{proof}

Now denoting
\begin{equation}\label{w31}
w=(w^{i})=\left(A^{n\beta}_{ij}D_{\beta}u^{j}-g^{ni}\right),
\end{equation}
then we have
\begin{lemma}\label{lem45}
For $0<\delta<T$, $0<\epsilon<1$, and for $|\gamma'|\leq\,k$,
$s\leq{l}$,
$D_{x'}^{\gamma'}D_{t}^{s}w,D_{x'}^{\gamma'}D_{t}^{s}\partial_{n}w\in\,L^{2}_{loc}((\omega_{\epsilon}\cap\omega_{m})\times(\delta,T))$,
\begin{align}\label{lem45-1}
&\sum_{|\gamma'|\leq\,k,s\leq{l}}\int_{\delta}^{T}\int_{\omega_{\epsilon}}|D_{x'}^{\gamma'}D_{t}^{s}w|^{2}dxdt
+\sum_{|\gamma'|\leq\,k-1,s\leq{l}}\int_{\delta}^{T}\int_{\omega_{\epsilon}\cap\omega_{m}}|D_{x'}^{\gamma'}D_{t}^{s}\partial_{n}w|^{2}dxdt\nonumber\\
&\leq\,C\bigg(\|\varphi\|_{L^{2}(\omega)}^{2}
+\sum_{|\gamma'|\leq\,k,s\leq{l}}\|D_{x'}^{\gamma'}f\|_{L^{2}(\Omega_{T})}^{2}
+\sum_{|\gamma'|\leq\,k,s\leq\,l+1}\|D_{x'}^{\gamma'}D_{t}^{s}g\|^{2}_{L^{2}(\Omega_{T})}\bigg),
\end{align}
where $C$ depends only on $n,N,\lambda,\Lambda_{k+2l+1},T,k$,
$\epsilon$ and $\delta$.
\end{lemma}

\begin{proof}
Rewrite the equation in \eqref{IBVP0} as
\begin{equation}\label{w2}
\partial_{n}w=u_{t}^{i}+\sum_{\alpha\leq\,n-1}\partial_{\alpha}\left(g^{\alpha\,i}-A^{\alpha\beta}_{ij}D_{\beta}u^{j}\right)-f^{i}.
\end{equation}
For any $\gamma'$, $|\gamma'|=1$, applying $D_{x'}^{\gamma'}$ to
\eqref{w31} and \eqref{w2}, by virtue of \eqref{lem43-0} and
\eqref{lem44-1} with $l=1$, we obtain,
$D_{x'}^{\gamma'}\partial_{n}w\in\,L^{2}_{loc}(\Omega_{T})$ for
$k=2$. Further applying $\partial_{t}$, similarly as the above, by
virtue of \eqref{lem44-1} and \eqref{lem44-2}, we obtain
\eqref{lem45-1} for $k=2$ and $l=1$. For general $k$ and $l$, we
will obtain \eqref{lem45-1} by induction.
\end{proof}

Here we need the following embedding inequality, which is a
variation of well-known Sobolev's inequality. The proof also could
be found in \cite{ln}.

\begin{lemma}\label{lem9}
Let $f$ be a real function in a bounded domain
$\omega\subset\mathbb{R}^{n}$ with
$D_{x'}^{\gamma'}f\in{L}^{2}(\omega)$ and
$D_{x'}^{\gamma'}\partial_{n}f\in{L}^{2}(\omega)$ for all
$0\leq|\gamma'|\leq[\frac{n-1}{2}]+1=:\overline{k}$. Then
$f\in{C}^{0}(\overline{\omega})$, and
$$\|f\|_{L^{\infty}(\omega)}\leq\,C(n)\sum_{|\gamma'|\leq\,\overline{k}}\left(\|D_{x'}^{\gamma'}f\|_{L^{2}(\omega)}
+\|D_{x'}^{\gamma'}\partial_{n}f\|_{L^{2}(\omega)}\right).$$
Further, if $f(x,t)\in{L}^{2}(0,T;C^{0}(\omega))$ and
$f_{t}(x,t)\in{L}^{2}(0,T;C^{0}(\omega))$, then
$f\in{C}^{0}(\overline{\Omega_{T}})$, and
\begin{align*}
&\|f\|_{L^{\infty}(\Omega_{T})}\\
&\leq\,C(n)\sum_{|\gamma'|\leq\,\overline{k}}\left(\|D_{x'}^{\gamma'}f\|_{L^{2}(\Omega_{T})}
+\|D_{x'}^{\gamma'}\partial_{n}f\|_{L^{2}(\Omega_{T})}+\|D_{x'}^{\gamma'}f_{t}\|_{L^{2}(\Omega_{T})}
+\|D_{x'}^{\gamma'}\partial_{n}f_{t}\|_{L^{2}(\Omega_{T})}\right).
\end{align*}
\end{lemma}

\begin{proof}[Proof of Proposition \ref{prop1}]
It is well known that for each $m$,
$u\in\,C^{\infty}((\omega_{m}\cap\omega_{\epsilon})\times(\delta,T))$.
For $k\geq\overline{k}=[\frac{n-1}{2}]+1$ and
$|\gamma'|\leq\,k-\overline{k}$, by Lemma \ref{lem43} and Lemma
\ref{lem44} and application of Lemma \ref{lem9} with
$f=D_{x'}^{\gamma'}u$, we have
$D_{x'}^{\gamma'}u\in\,C^{0}(\omega_{\epsilon}\times(\delta,T))$,
and
\begin{align}\label{prop1-2}
&\sum_{|\gamma'|\leq\,k-\overline{k}}\|D_{x'}^{\gamma'}u\|_{L^{\infty}(\omega_{\epsilon}\times(\delta,T))}\nonumber\\
&\leq\,C\bigg(\|\varphi\|_{L^{2}(\omega)}
+\sum_{|\gamma'|\leq\,k,s\leq1}\|D_{x'}^{\gamma'}D_{t}^{s}g\|_{L^{2}(\Omega_{T})}
+\sum_{|\gamma'|\leq\,k,s\leq1}\|D_{x'}^{\gamma'}D_{t}^{s}f\|_{L^{2}(\Omega_{T})}\bigg),
\end{align}
where $C$ depends only on $\epsilon,\delta,k,n,N,\lambda,T$ and
$\Lambda_{k+2}$. Similarly, for $k\geq\overline{k}+1$ and
$|\gamma'|\leq\,k-\overline{k}-1$, by Lemma \ref{lem45} and Lemma
\ref{lem9} with $f=D_{x'}^{\gamma'}w$, we have
$D_{x'}^{\gamma'}w\in\,C^{0}(\omega_{\epsilon}\times(\delta,T))$,
and
\begin{align}\label{prop1-3}
&\sum_{|\gamma'|\leq\,k-\overline{k}-1}\|D_{x'}^{\gamma'}w\|_{L^{\infty}(\omega_{\epsilon}\times(\delta,T))}\nonumber\\
&\leq\,C\bigg(\|\varphi\|_{L^{2}(\omega)}
+\sum_{|\gamma'|\leq\,k,s\leq2}\|D_{x'}^{\gamma'}D_{t}^{s}g\|_{L^{2}(\Omega_{T})}
+\sum_{|\gamma'|\leq\,k,s\leq1}\|D_{x'}^{\gamma'}D_{t}^{s}f\|_{L^{2}(\Omega_{T})}\bigg),
\end{align}
where $C$ depends only on $\epsilon,\delta,k,n,N,\lambda,T$ and
$\Lambda_{k+3}$. Consequently,
$DD_{x'}^{\gamma'}u\in{L}^{\infty}_{loc}(\omega\times(\delta,T))$,
and
\begin{align}\label{prop1-4}
&\sum_{|\gamma'|\leq\,k-\overline{k}-1}\|DD_{x'}^{\gamma'}u\|_{L^{\infty}((\omega_{m}\cap\omega_{\epsilon})\times(\delta,T))}\nonumber\\
&\leq\,C\bigg(\|\varphi\|_{L^{2}(\omega)}
+\sum_{|\gamma'|\leq\,k,s\leq2}\|D_{x'}^{\gamma'}D_{t}^{s}g\|_{L^{2}(\Omega_{T})}
+\sum_{|\gamma'|\leq\,k,s\leq1}\|D_{x'}^{\gamma'}D_{t}^{s}f\|_{L^{2}(\Omega_{T})}\bigg),
\end{align}
where $C$ has the same dependence as in \eqref{prop1-3}. Indeed, by
\eqref{prop1-2}, we only need to show that
$\partial_{n}D_{x'}^{\gamma'}u\in\,L^{\infty}_{loc}((\omega_{m}\cap\omega_{\epsilon})\times(\delta,T))$,
for $|\gamma'|\leq\,k-\bar{k}-1$. Since
$$A^{nn}_{ij}\partial_{n}D_{x'}^{\gamma'}u^{j}=D_{x'}^{\gamma'}w-\left(\left(D_{x'}^{\gamma'}A_{ij}^{n\beta}\right)\partial_{\beta}u^{j}
+\sum_{\beta\leq\,n-1}A_{ij}^{n\beta}\left(\partial_{\beta}D_{x'}^{\gamma'}u^{j}\right)
-D_{x'}^{\gamma'}g^{ni}\right),$$ it follows, by \eqref{prop1-2} and
\eqref{prop1-3}, that
$A_{ij}^{nn}\partial_{n}D_{x'}^{\gamma'}u^{j}\in\,L^{\infty}_{loc}(\omega\times(\delta,T))$
and
\begin{align*}
&\sum_{|\gamma'|\leq\,k-\overline{k}-1}\|A_{ij}^{nn}\partial_{n}D_{x'}^{\gamma'}u\|_{L^{\infty}((\omega_{m}\cap\omega_{\epsilon})\times(\delta,T))}\nonumber\\
& \leq\,C\bigg(\|\varphi\|_{L^{2}(\omega)}
+\sum_{|\gamma'|\leq\,k,s\leq2}\|D_{x'}^{\gamma'}D_{t}^{s}g\|_{L^{2}(\Omega_{T})}
+\sum_{|\gamma'|\leq\,k,s\leq1}\|D_{x'}^{\gamma'}D_{t}^{s}f\|_{L^{2}(\Omega_{T})}\bigg),
\end{align*}
where $C$ depends only on $\epsilon,\delta,k,n,N,\lambda,T$ and
$\Lambda_{k+3}$. Because of \eqref{coeff1} and \eqref{coeff2},
$(A^{nn}_{ij})$ is a positive definite $N\times\,N$ matrix with
eigenvalues in $[\lambda,\Lambda_{0}]$. Consequently,
$\partial_{n}D_{x'}^{\gamma'}u\in{L}^{\infty}_{loc}(\omega\times(\delta,T))$
and
\begin{align*}
&\sum_{|\gamma'|\leq\,k-\overline{k}-1}\|\partial_{n}D_{x'}^{\gamma'}u\|_{L^{\infty}((\omega_{m}\cap\omega_{\epsilon})\times(\delta,T))}\nonumber\\
& \leq\,C\bigg(\|\varphi\|_{L^{2}(\omega)}
+\sum_{|\gamma'|\leq\,k,s\leq2}\|D_{x'}^{\gamma'}D_{t}^{s}g\|_{L^{2}(\Omega_{T})}
+\sum_{|\gamma'|\leq\,k,s\leq1}\|D_{x'}^{\gamma'}D_{t}^{s}f\|_{L^{2}(\Omega_{T})}\bigg),
\end{align*}
where $C$ depends only on $\epsilon,\delta,k,n,N,\lambda,T$ and
$\Lambda_{k+3}$.

On the other hand, by Lemma \ref{lem44} and Lemma \ref{lem9}, we
have
\begin{align*}
&\sum_{|\gamma'|\leq\,k-\overline{k}-1}\|D_{x'}^{\gamma'}u_{t}\|_{L^{\infty}((\omega_{m}\cap\omega_{\epsilon})\times(\delta,T))}\nonumber\\
& \leq\,C\bigg(\|\varphi\|_{L^{2}(\omega)}
+\sum_{|\gamma'|\leq\,k-1,s\leq2}\|D_{x'}^{\gamma'}D_{t}^{s}g\|_{L^{2}(\Omega_{T})}
+\sum_{|\gamma'|\leq\,k-1,s\leq2}\|D_{x'}^{\gamma'}D_{t}^{s}f\|_{L^{2}(\Omega_{T})}\bigg),
\end{align*}
where $C$ depends only on $\epsilon,\delta,k,n,N,\lambda,T$ and
$\Lambda_{k+3}$.

Inequality \eqref{prop1-4} gives us the desired bound for tangential
derivatives in spatial space of $u$ and $\partial_{n}u$. To estimate
derivatives involving $\partial_{n}^{k}u$ for $k>1$, we simply
observe that these can be derived recursively from those already
established. Indeed, according to the equation in \eqref{IBVP0}, and
by the definition of weak solution,
$$u^{i}_{t}=D_{\alpha}(A^{\alpha\beta}_{ij}D_{\beta}u^{j}-g^{\alpha\,i})+f^{i}\quad\mbox{in}\ \omega\times(0,T),\quad\mbox{in the sense of distribution}.$$
So in every $\omega_{m}\times(0,T)$, the equation could be rewritten
piecewise as
\begin{equation}\label{energy7}
A^{nn}_{ij}D_{nn}u^{j}=u^{i}_{t}+D_{\alpha}g^{\alpha\,i}-\sum_{\alpha+\beta\leq\,2n-1}D_{\alpha}(A^{\alpha\beta}_{ij}D_{\beta}u^{j})
-(D_{n}A_{ij}^{nn})D_{n}u^{j}-f^{i}.
\end{equation}

Since the matrix $A^{nn}_{ij}$ has a bounded inverse, we can
estimate $D_{x'}^{\gamma'}\partial_{n}^{2}u$ pointwise for each open
strip. Applying $\partial_{n}$ and $D_{t}$ to \eqref{energy7}, we
can then estimate higher derivatives. We thus obtain
\begin{align*}
&\sum_{r+2s\leq\,k-\overline{k}+1}\|D^{r}_{x}D_{t}^{s}u\|_{L^{\infty}((\omega_{m}\cap\omega_{\epsilon})\times(\delta,T))}\nonumber\\
&\leq\,C\bigg(\|\varphi\|_{L^{2}(\omega)}
+\sum_{|\gamma'|\leq\,k,s\leq\frac{k-\overline{k}-1}{2}+1}\|D_{x'}^{\gamma'}D_{t}^{s}g\|_{L^{2}(\Omega_{T})}
+\sum_{|\gamma'|\leq{k},s\leq\frac{k-\overline{k}-1}{2}+1}\|D_{x'}^{\gamma'}D_{t}^{s}f\|_{L^{2}(\Omega_{T})}\bigg),
\end{align*}
where $C$ depends only on $\epsilon,\delta,k,n,N,\lambda,T$ and
$\Lambda_{k+3}$. Hence,
$u\in\,C^{\infty}((\overline{\omega}_{m}\cap\omega_{\epsilon})\times[\delta,T])$.
Proposition \ref{prop1} is established.
\end{proof}

\subsection{Interior Estimates for the equations in Domain
$D\times(0,T)$}\label{subsD}
\bigskip

In this subsection we consider the equation \eqref{IBVP} in
${D}\times(0,T)$, defined as in section \ref{sec2}. Here we still
assume that $f,g\in C^{\infty}(\overline{{D}}_{m}\times[0,T])$, and
$A_{ij}^{\alpha\beta}$ satisfies \eqref{coeff1} \eqref{coeff2} and
\eqref{coeff31}. we will establish the interior estimates for a weak
solution of the equation \eqref{IBVP}.

\begin{lemma}\label{lemin1}
Under \eqref{coeff1} \eqref{coeff2} and \eqref{coeff31}, suppose
$u(x,t)\in{V}(Q_{T})$ is a weak solution of \eqref{IBVP}, then, for
$0<\epsilon<1$, $0<\delta<T$, we have
\begin{align}\label{lemin1.0}
&\|u\|^{2}_{V({D}_{\epsilon}\times(\delta,T))}\leq\,C\bigg(\|u\|^{2}_{L^{2}(Q_{T})}+\|f\|^{2}_{L^{2}(Q_{T})}+\|g\|^{2}_{L^{2}(Q_{T})}\bigg),
\end{align}
where $C$ depends only on $\lambda$, $\epsilon$ and $\delta$.
\end{lemma}

\begin{proof}
By Remark \ref{rem1}, for $0<\epsilon<1$, $0<\delta<T$, we have
$u\in{V}^{1,0}({D}_{\frac{\epsilon}{2}}\times(\frac{\delta}{2},T))$.
Take $\zeta=\overline{u}_{h}\eta^{2}(x,t)$ in \eqref{weaksoln1},
where $\eta(x,t)$ is a suitable cutoff function,
 satisfying $0\leq\eta\leq1$, and
\begin{equation}\label{eta}
\eta(x,t)=
\begin{cases}1,&\mbox{in}~{D}_{\epsilon}\times(\delta,T),\\
0,&\mbox{in}~Q_{T}\setminus{D}_{\frac{\epsilon}{2}}\times(\frac{\delta}{2},T),
\end{cases}\quad|\eta_{t}|\leq\frac{C}{\delta},\quad|D\eta|\leq\frac{C}{\epsilon}.
\end{equation}
Then, instead of \eqref{lem32.3}, we have
\begin{align*}
&\lim_{h\rightarrow0}\int_{0}^{\tau}\int_{{D}}u(\overline{u}_{h}\eta^{2})_{t}dxdt\nonumber\\
&=\frac{1}{2}\int_{{D}}(\eta^{2}u^{2})(x,\tau)dx-\frac{1}{2}\int_{{D}}(\eta^{2}u^{2})(x,0)dx
+2\int_{0}^{\tau}\int_{{D}}\eta\eta_{t}u^{2}dxdt\\
&=\frac{1}{2}\int_{{D}}(\eta^{2}u^{2})(x,\tau)dx+2\int_{0}^{\tau}\int_{{D}}\eta\eta_{t}u^{2}dxdt.
\end{align*}
And passing to the limit as $h\rightarrow 0$ in other terms, it
leads to
\begin{align*}
&\frac{1}{2}\int_{{D}}(\eta^{2}u^{2})(x,\tau)dx+\int_{0}^{\tau}\int_{{D}}A^{\alpha\beta}_{ij}\eta\,D_{\beta}u\eta\,D_{\alpha}udxdt
=-2\int_{0}^{\tau}\int_{{D}}\eta{u}A^{\alpha\beta}_{ij}D_{\beta}uD_{\alpha}\eta\,dxdt\\
&\hspace{3cm}+\int_{0}^{\tau}\int_{{D}}\bigg(fu\eta^{2}+\eta^{2}g^{\alpha}D_{\alpha}u+2\eta\eta_{t}u^{2}+2u\eta{}g^{\alpha}D_{\alpha}\eta\bigg)dxdt.
\end{align*}
From it follows the estimate \eqref{lemin1.0}.
\end{proof}

For the derivatives in $t$, we have

\begin{lemma}\label{lemin2}
Under \eqref{coeff1} \eqref{coeff2} and \eqref{coeff3}, if
$u\in{V}(Q_{T})$ is a weak solution of \eqref{IBVP}, then for any
$0<\epsilon<1$, $0<\delta<T$, we have
\begin{align}
&\sup_{\delta\leq\tau\leq\,T}\int_{{D}_{\epsilon}}|Du(x,\tau)|^{2}dx+\int_{\delta}^{T}\int_{{D}_{\epsilon}}|u_{t}|^{2}dxdt\nonumber\\
&\hspace{3cm}\leq\,C\bigg(\|u\|_{L^{2}(Q_{T})}^{2}
+\|f\|^{2}_{L^{2}(Q_{T})}+\sum_{s\leq1}\|D_{t}^{s}g\|^{2}_{L^{2}(Q_{T})}\bigg),
\end{align}
where $C$ depends only on $n,N,\lambda,\Lambda_{2},T,\epsilon$ and
$\delta$.
\end{lemma}
\begin{lemma}\label{lemin3}
Under \eqref{coeff1} \eqref{coeff2} and \eqref{coeff3}, if
$u\in{V}(Q_{T})$ is a weak solution of \eqref{IBVP}, then for any
$0<\epsilon<1$ and $0<\delta<T$, we have
\begin{align}\label{lemin4.2}
&\sup_{\delta\leq\tau\leq\,T}\sum_{|\gamma'|\leq\,k}\int_{{D}_{\epsilon}}|u_{t}(x,\tau)|^{2}dx
+\sum_{|\gamma'|\leq\,k}\int_{\delta}^{T}\int_{{D}_{\epsilon}}|Du_{t}|^{2}dxdt\nonumber\\
&\leq\,C\bigg(\|u\|_{L^{2}(Q_{T})}^{2}
+\sum_{s\leq1}\|D_{t}^{s}f\|_{L^{2}(Q_{T})}^{2}
+\sum_{s\leq1}\|D_{t}^{s}g\|^{2}_{L^{2}({D}_{m}\times(0,T))}\bigg),
\end{align}
where $C$ depends only on $n,N,\lambda,\Lambda_{2},T$, $\epsilon$
and $\delta$.
\end{lemma}

\subsection{Interior Estimates for Laminar Systems in Domain
$\omega\times(0,T)$}\label{subslaminar}

In this subsection we consider the laminar  systems \eqref{IBVP} in
$\omega\times(0,T)$, defined as in subsection \ref{subibvp}. Here we
still assume that $f,g\in
C^{\infty}(\overline{\omega}_{m}\times[0,T])$, and
$A_{ij}^{\alpha\beta}$ satisfies \eqref{coeff1} \eqref{coeff2} and
\eqref{coeff3}. we have

\begin{prop}\label{prop41}
Under \eqref{coeff1} \eqref{coeff2} and \eqref{coeff3}, let
$u\in{V}(\Omega_{T})$ be a weak solution of \eqref{IBVP}. Then, for
$0<\delta<T$ and for all $\gamma'$,
$D_{x'}^{\gamma'}u\in\,C^{0}(\omega\times(\delta,T))$, and for each
$m$,
$u\in\,C^{\infty}((\omega\cap\overline{\omega}_{m})\times(\delta,T))$.
Moreover for any small $\epsilon>0$, $k\geq0$, and any $m$
\begin{align}\label{prop41.1}
&\sum_{r+2s\leq{k}}\|D_{x}^{r}D_{t}^{s}u\|_{L^{\infty}(\overline{\omega}_{m}\cap\omega_{\epsilon}\times(\delta,T))}\nonumber\\
&\leq\,C\left(\|u\|_{L^{2}(\Omega_{T})}
+\sum_{|\gamma'|\leq\overline{k}+k+1,s\leq{\frac{k}{2}+1}}\left(\|D_{x'}^{\gamma'}D_{t}^{s}g\|_{L^{2}(\Omega_{T})}
+\|D_{x'}^{\gamma'}D_{t}^{s}f\|_{L^{2}(\Omega_{T})}\right)\right),
\end{align}
where $\overline{k}=[\frac{n-1}{2}]+1$ and $C$ depends only on
$\epsilon,\delta,k,n,N,\lambda,T$ and $\Lambda_{\overline{k}+k+2}$.
\end{prop}

\begin{corollary}\label{cor1}
Suppose $A$ is constant in $x$ and smooth in $t$ in
$\omega_{m}\times(0,T)$, satisfying \eqref{coeff1} \eqref{coeff2},
and $f,g$ are constant in $\omega_{m}\times(0,T)$. If
$u\in{V}(\Omega_{T})$ is a weak solution of \eqref{IBVP}, then, for
$0<\delta<T$ and for all $\gamma'$,
$D_{x'}^{\gamma'}u\in\,C^{0}(\omega\times(\delta,T))$, and for any
$0<\epsilon<1$, $k\geq 0$, we have
$$\|u\|_{C^{k,k/2}((\omega_{m}\cap\omega_{\epsilon})\times(\delta,T))}
\leq\,C\left(\|u\|_{L^{2}(\Omega_{T})}+\|f\|_{L^{\infty}(\Omega_{T})}+\|g\|_{L^{\infty}(\Omega_{T})}\right),$$
where $C=C(\epsilon,k,n,N,\lambda,\Lambda,T, A)$.
\end{corollary}

The proof of Proposition \ref{prop41} will needs the following
Lemmas, analogically as in Subsection \ref{subibvp}. From arguments
similar to that in the derivation of Lemma \ref{lem43}-\ref{lem45},
just replacing \eqref{lem32.1} by \eqref{lemin1.0} with $D=\omega$,
we have the following higher interior estimates for the solution of
\eqref{IBVP}.

\begin{lemma}\label{lemin4}
Under \eqref{coeff1} \eqref{coeff2} and \eqref{coeff3}, if
$u\in{V}(\Omega_{T})$ is a weak solution of \eqref{IBVP}, then for
any $0<\epsilon<1$, $0<\delta<T$ and positive integer $k$, for
$1\leq|\gamma'|\leq{k}$, $D_{x'}^{\gamma'}u\in
W_{2}^{1,0}(\omega_{\epsilon}\times(\delta,T))$, and we have
\begin{align*}
&\sum_{|\gamma'|\leq{k}}\sup_{\delta\leq\tau\leq{T}}\int_{\omega_{\epsilon}(\tau)}|D^{\gamma'}_{x'}u(x,\tau)|^{2}dx
+\sum_{|\gamma'|\leq{k}}\int_{\delta}^{T}\int_{\omega_{\epsilon}}|DD_{x'}^{\gamma'}u|^{2}dxdt\nonumber\\
&\hspace{1cm}\leq\,C\bigg(\|u\|_{L^{2}(\Omega_{T})}^{2}
+\sum_{|\gamma'|\leq{k}}\|D^{\gamma'}_{x'}f\|_{L^{2}(\Omega_{T})}^{2}+\sum_{|\gamma'|\leq{k}}\|D^{\gamma'}_{x'}g\|_{L^{2}(\Omega_{T})}^{2}\bigg),
\end{align*}
where $C$ depends only on $n,N,\lambda,\Lambda_{k},k$, $\epsilon$
and $\delta$.
\end{lemma}

\begin{lemma}\label{lemin5}

Under \eqref{coeff1} \eqref{coeff2} and \eqref{coeff3}, if
$u\in{V}(\Omega_{T})$ is a weak solution of problem \eqref{IBVP0},
then for any $0<\delta<T$ and $0<\epsilon<1$, we have, for
$|\gamma'|\leq{k}$, and $l\geq0$,
\begin{align*}
&\sum_{|\gamma'|\leq\,k,s\leq{l}}\sup_{\delta\leq\tau\leq{T}}\int_{\omega_{\epsilon}}|DD_{x'}^{\gamma'}D_{t}^{s}u(x,\tau)|^{2}dx+
\sum_{|\gamma'|\leq\,k,s\leq{l}+1}\int_{\delta}^{T}\int_{\omega_{\epsilon}}|D_{x'}^{\gamma'}D_{t}^{s}u|^{2}dxdt\nonumber\\
&\leq\,C\bigg(\|u\|_{L^{2}(\Omega_{T})}^{2}
+\sum_{|\gamma'|\leq\,k,s\leq{l}}\|D_{x'}^{\gamma'}f\|_{L^{2}(\Omega_{T})}^{2}
+\sum_{|\gamma'|\leq\,k,s\leq{l}+1}\|D_{x'}^{\gamma'}D_{t}^{s}g\|^{2}_{L^{2}(\Omega_{T})}\bigg),
\end{align*}
where $C$ depends only on
$n,N,\lambda,\Lambda_{k+2l+2},T,k,\epsilon$ and $\delta$.
\begin{align*}
&\sup_{\delta\leq\tau\leq\,T}\sum_{|\gamma'|\leq\,k,s\leq{l}}\int_{\omega_{\epsilon}}|D_{x'}^{\gamma'}D_{t}^{s}u(x,\tau)|^{2}dx
+\sum_{|\gamma'|\leq\,k,s\leq{l}}\int_{\delta}^{T}\int_{\omega_{\epsilon}}|DD_{x'}^{\gamma'}D_{t}^{s}u|^{2}dxdt\nonumber\\
&\leq\,C\bigg(\|u\|_{L^{2}(\Omega_{T})}^{2}
+\sum_{|\gamma'|\leq\,k,s\leq{l}}\|D_{x'}^{\gamma'}D_{t}^{s}f\|_{L^{2}(\Omega_{T})}^{2}
+\sum_{|\gamma'|\leq\,k,s\leq{l}}\|D_{x'}^{\gamma'}D_{t}^{s}g\|^{2}_{L^{2}(\omega_{m}\times(0,T))}\bigg),
\end{align*}
where $C$ depends only on
$n,N,\lambda,\Lambda_{k+2l+2},T,k,\epsilon$ and $\delta$.
\end{lemma}

\begin{lemma}\label{lemin6}
For $0<\delta<T$, $0<\epsilon<1$, and for $|\gamma'|\leq\,k$,
$s\leq{l}$,
$D_{x'}^{\gamma'}D_{t}^{s}w,D_{x'}^{\gamma'}D_{t}^{s}\partial_{n}w\in\,L^{2}_{loc}((\omega_{\epsilon}\cap\omega_{m})\times(\delta,T))$,
\begin{align*}
&\sum_{|\gamma'|\leq\,k,s\leq{l}}\int_{\delta}^{T}\int_{\omega_{\epsilon}}|D_{x'}^{\gamma'}D_{t}^{s}w|^{2}dxdt
+\sum_{|\gamma'|\leq\,k-1,s\leq{l}}\int_{\delta}^{T}\int_{\omega_{\epsilon}\cap\omega_{m}}|D_{x'}^{\gamma'}D_{t}^{s}\partial_{n}w|^{2}dxdt\nonumber\\
&\leq\,C\bigg(\|u\|_{L^{2}(\Omega_{T})}^{2}
+\sum_{|\gamma'|\leq\,k,s\leq{l}}\|D_{x'}^{\gamma'}f\|_{L^{2}(\Omega_{T})}^{2}
+\sum_{|\gamma'|\leq\,k,s\leq\,l+1}\|D_{x'}^{\gamma'}D_{t}^{s}g\|^{2}_{L^{2}(\Omega_{T})}\bigg),
\end{align*}
where $C$ depends only on $n,N,\lambda,\Lambda_{k+2l+1},T,k$,
$\epsilon$ and $\delta$.
\end{lemma}
The proof  Proposition \ref{prop41} can be established by Lemma
\ref{lem9} and Lemma \ref{lemin4}-\ref{lemin6}.

\bigskip
\section{A Perturbation Lemma}\label{sec5}
\bigskip

In this section, we will give a perturbation lemma for parabolic
systems. For simplicity, here we still suppose that $\omega$ is the
unit cube. For $0<\lambda<\Lambda_{0}<\infty$,
$\mathscr{A}(\lambda,\Lambda_{0})$ denotes the class of measurable
vector-valued functions $(A_{ij}^{\alpha\beta}(x,t))$ satisfying
\eqref{coeff1} and \eqref{coeff2}. Denote
$$\Omega_{T}^{(\sigma)}=(1-\sigma)\omega\times(\sigma{T},T).$$
In this section we use, unless otherwise stated, $C$ to denote
various positive constants whose values may change from line to line
and which depend only on $n,N,\lambda,\Lambda_{2},T$.

\begin{lemma}\label{pertubationlemma}
For $0<\epsilon<1$, suppose $A,B\in\mathscr{A}(\lambda,\Lambda_{0})$
satisfies
\begin{equation}\label{AB}
\int_{0}^{T}\int_{\omega}|A-B|<\epsilon.
\end{equation}
If $A$ is $C^{1}$ in $t$, then for any
$f\in{W}_{2}^{0,1}(\omega\times(0,T))$,
$g\in{W}_{2}^{0,1}(\omega\times(0,T))$ and the solution
$u\in{V}(\omega\times(0,T))$ of
$$(u^{i})_{t}-D_{\alpha}(A^{\alpha\beta}_{ij}(x,t)D_{\beta}u^{j})=-D_{\alpha}g^{\alpha i}+f^{i},\quad\quad\mbox{in}\ \Omega_{T}=\omega\times(0,T),$$
there exists some solution $v\in\,V(\Omega_{T}^{(\frac{1}{4})})$ of
$$(v^{i})_{t}-D_{\alpha}(B^{\alpha\beta}_{ij}(x,t)D_{\beta}v^{j})=0,\quad\quad\mbox{in}\ \Omega_{T}^{(\frac{1}{4})},$$
and we have
\begin{align*}
\|u-v\|_{V(\Omega_{T}^{(\frac{1}{2})})}
\leq\,C\left(\|g\|_{L^{2}(\Omega_{T})}+\|f\|_{L^{2}(\Omega_{T})}+\epsilon^{\gamma}\left(\|u\|_{L^{2}(\Omega_{T})}
+\|g_{t}\|_{L^{2}(\Omega_{T})}+\|f_{t}\|^{2}_{L^{2}(\Omega_{T})}\right)\right),
\end{align*}
where $C$ and $\gamma$ depend on $n,N,\lambda,\Lambda_{2},T$.
\end{lemma}

In our proof of Lemma \ref{pertubationlemma}, we make use of Theorem
\ref{reverseholder}, an extension of a classical result of Campanato
in \cite{camp} for strongly parabolic systems to parabolic systems
with coefficients satisfying only \eqref{coeff1} and \eqref{coeff2}.
See Appendix.

\begin{proof}[Proof of Lemma \ref{pertubationlemma}.]

{\bf STEP 1.} By Lemma \ref{lemin1} and Lemma \ref{lemin3} with
$D=\omega$, we have
$$\|u\|^{2}_{W^{1,1}_{2}\big(\Omega_{T}^{(\frac{1}{5})}\big)}\leq\,C\left(\|u\|^{2}_{L^{2}(\Omega_{T})}+\sum_{s\leq1}\|D_{t}^{s}g\|^{2}_{L^{2}(\omega\times(0,T))}
+\|f\|^{2}_{L^{2}(\Omega_{T})}\right).$$ By Fubini theorem, there
exists $\frac{1}{8}<\sigma<\frac{1}{4}$ such that
\begin{equation}\label{pert5-2}
\int_{\partial((1-\sigma)\omega)}\bigg(\int_{\frac{1}{5}T}^{T}(|u|^{2}+|u_{t}|^{2}+|\nabla{u}|^{2})dt\bigg)dS
\leq\,C\|u\|^{2}_{W^{1,1}_{2}\big(\Omega_{T}^{(\frac{1}{5})}\big)}.
\end{equation}
By Lemma \ref{lemin1}, Lemma \ref{lemin3} and Lemma \ref{lemin4}
with $D=\omega$, for any $\frac{T}{10}\leq\tau\leq{T}$, we have
$$\int_{(1-\sigma)\omega}|u(x,\tau)|^{2}dx\leq\,C\left(\|u\|^{2}_{L^{2}(\Omega_{T})}+\|g\|^{2}_{L^{2}(\Omega_{T})}+\|f\|^{2}_{L^{2}(\Omega_{T})}\right),$$
$$\int_{(1-\sigma)\omega}|Du(x,\tau)|^{2}dx\leq\,C\left(\|u\|^{2}_{L^{2}(\Omega_{T})}+\sum_{s\leq1}\|D_{t}^{s}g\|^{2}_{L^{2}(\Omega_{T})}+\|f\|^{2}_{L^{2}(\Omega_{T})}\right),$$
$$\int_{(1-\sigma)\omega}|u_{t}(x,\tau)|^{2}dx\leq\,C\left(\|u\|^{2}_{L^{2}(\Omega_{T})}+\sum_{s\leq1}\|D_{t}^{s}g\|^{2}_{L^{2}(\Omega_{T})}
+\sum_{s\leq1}\|D_{t}^{s}f\|^{2}_{L^{2}(\Omega_{T})}\right).$$
Therefore,
\begin{align*}
&\int_{(1-\sigma)\omega}\bigg(|u|^{2}+|u_{t}|^{2}+|\nabla{u}|^{2}\bigg)\left(x,\sigma{T}\right)dx
+\int_{(1-\sigma)\omega}\bigg(|u|^{2}+|u_{t}|^{2}+|\nabla{u}|^{2}\bigg)(x,T)dx\\
&\hspace{3cm}\leq\,C\left(\|u\|^{2}_{L^{2}(\Omega_{T})}+\sum_{s\leq1}\|D_{t}^{s}g\|^{2}_{L^{2}(\Omega_{T})}
+\sum_{s\leq1}\|D_{t}^{s}f\|^{2}_{L^{2}(\Omega_{T})}\right).
\end{align*}
Then combining with \eqref{pert5-2}, we have
$$\int_{\partial\big(\Omega_{T}^{(\sigma)}\big)}\bigg(|u|^{2}+|u_{t}|^{2}+|\nabla\,u|^{2}\bigg)
\leq\,C\left(\|u\|^{2}_{L^{2}(\Omega_{T})}+\sum_{s\leq1}\|D_{t}^{s}g\|^{2}_{L^{2}(\Omega_{T})}
+\sum_{s\leq1}\|D_{t}^{s}f\|^{2}_{L^{2}(\Omega_{T})}\right).$$ Fix
some $0<\delta<1$, then take
$U\in{H}^{\frac{3}{2}-\delta}(\Omega_{T}^{(\sigma)})$ (the usual
fractional Sobolev space in $(n+1)$ dimensions) as an extension of
$u$ on $\partial\Omega_{T}^{(\sigma)}$ such that
$$\|U\|^{2}_{H^{\frac{3}{2}-\delta}\big(\Omega_{T}^{(\sigma)}\big)}
\leq\,C\int_{\partial\big(\Omega_{T}^{(\sigma)}\big)}
\bigg(|u|^{2}+|u_{t}|^{2}+|\nabla{u}|^{2}\bigg).$$ By the Sobolev
embedding theorem,
$$\|U_{t}\|_{L^{\overline{p}}(\Omega_{T}^{(\sigma)})}+\|\nabla{U}\|_{L^{\overline{p}}(\Omega_{T}^{(\sigma)})}\leq\,C\|U\|_{H^{3/2-\delta}(\Omega_{T}^{(\sigma)})},$$
where $\overline{p}=\frac{2n+2}{n+2\delta}\in(2,\frac{2n+2}{n})$. It
implies that
\begin{equation}\label{pert5-3}
\|U_{t}\|^{2}_{L^{\overline{p}}(\Omega_{T}^{(\sigma)})}+\|\nabla{U}\|^{2}_{L^{\overline{p}}(\Omega_{T}^{(\sigma)})}\leq\,C\left(\|u\|^{2}_{L^{2}(\Omega_{T})}+\sum_{s\leq1}\|D_{t}^{s}g\|^{2}_{L^{2}(\Omega_{T})}
+\sum_{s\leq1}\|D_{t}^{s}f\|^{2}_{L^{2}(\Omega_{T})}\right).
\end{equation}

{\bf STEP 2.} There exists $v\in\,V(\Omega_{T}^{(\sigma)})$
satisfying
$$
\begin{cases}
-(v^{i})_{t}+D_{\alpha}(B^{\alpha\beta}_{ij}(x,t)D_{\beta}v^{j})=0\
&\mbox{in}\
\Omega_{T}^{(\sigma)},\\
v(x,t)=u(x,t) &\mbox{on}\ \partial((1-\sigma)\omega)\times(\sigma{T},T),\\
v(x,t)=u(x,t)&\mbox{on}\ (1-\sigma)\omega\times\{\sigma{T}\}.
\end{cases}
$$
Since
$U\in{W}^{1,1}_{2}(\Omega_{T}^{(\sigma)})\subset{V}(\Omega_{T}^{(\sigma)})$,
then $v-U\in\overset{\circ}{V}(\Omega_{T}^{(\sigma)})$ and satisfies
$$-(v^{i}-U^{i})_{t}+D_{\alpha}(B^{\alpha\beta}_{ij}(x,t)D_{\beta}(v^{j}-U^{j}))
=U^{i}_{t}-D_{\alpha}(B^{\alpha\beta}_{ij}(x,t)D_{\beta}U^{j}),\quad\mbox{in}\
\Omega_{T}^{(\sigma)}.$$ Applying Theorem \ref{reverseholder}, we
have for all $2\leq{p}<p_{0}$, where $p_{0}$ is the one in Theorem
\ref{reverseholder}, (to determine the above constant $\delta$),
depending only on $n,\lambda$ and $\Lambda_{0}$, such that
\begin{align*}
\|\nabla(v-U)\|_{L^{p}(\Omega_{T}^{(\sigma)})}
&\leq\,C\left(\|U_{t}\|_{L^{p}(\Omega_{T}^{(\sigma)})}+
\|\nabla{U}\|_{L^{p}((1-\sigma)\omega\times(\sigma{T},T))}\right).
\end{align*}
Now we choose $\delta$ such that $2<\overline{p}<p_{0}$ with
$\overline{p}=p$. Then,
\begin{align*}
\|\nabla(v-U)\|^{2}_{L^{p}(\Omega_{T}^{(\sigma)})}
&\leq\,C\left(\|u\|^{2}_{L^{2}(\Omega_{T})}+\sum_{s\leq1}\|D_{t}^{s}g\|^{2}_{L^{2}(\Omega_{T})}
+\sum_{s\leq1}\|D_{t}^{s}f\|^{2}_{L^{2}(\Omega_{T})}\right).
\end{align*}
Recalling \eqref{pert5-3}, we have
$$\|\nabla{v}\|^{2}_{L^{p}(\Omega_{T}^{(\sigma)})}\leq\,C\left(\|u\|^{2}_{L^{2}(\Omega_{T})}+\sum_{s\leq1}\|D_{t}^{s}g\|^{2}_{L^{2}(\Omega_{T})}
+\sum_{s\leq1}\|D_{t}^{s}f\|^{2}_{L^{2}(\Omega_{T})}\right).$$

{\bf Step 3.} Combining the equations of $u'$s and $v'$s, we obtain
that
$$-(u^{i}-v^{i})_{t}+D_{\alpha}\left(A^{\alpha\beta}_{ij}(x,t)D_{\beta}(u^{j}-v^{j})\right)
=D_{\alpha}g^{\alpha\,i}+f^{i}+D_{\alpha}\left((B^{\alpha\beta}_{ij}-A^{\alpha\beta}_{ij})D_{\beta}v^{j}\right),\quad\mbox{in}\
\Omega_{T}^{(\sigma)}.$$ Since $u-v=0$ on
$\partial_{p}(\Omega_{T}^{(\sigma)})$, the parabolic boundary of
$\Omega_{T}^{(\sigma)}$, it follows from Lemma \ref{lem32} that
\begin{align*}
\sup_{\sigma{T}\leq\tau\leq{T}}\int_{(1-\sigma)\omega}&|(u-v)(x,\tau)|^{2}dx+\|\nabla(u-v)\|^{2}_{L^{2}(\Omega_{T}^{(\sigma)})}\\
&\leq\,C\left(\|g\|^{2}_{L^{2}(\Omega_{T}^{(\sigma)})}+\|f\|^{2}_{L^{2}(\Omega_{T}^{(\sigma)})}
+\int_{\sigma{T}}^{T}\int_{(1-\sigma)\omega}\big|(B-A)\nabla{v}\big|^{2}dxdt\right).
\end{align*}
By H\"{o}lder inequality and \eqref{AB}, we have
\begin{align*}
&\hspace{.5cm}\sup_{\frac{1}{2}{T}\leq\tau\leq{T}}\int_{\frac{1}{2}\omega}|(u-v)(x,\tau)|^{2}dx
+\|\nabla(u-v)\|^{2}_{L^{2}(\Omega_{T}^{(\frac{1}{2})})}\\
&\leq\,C\left(\|g\|^{2}_{L^{2}(\Omega_{T})}+\|f\|^{2}_{L^{2}(\Omega_{T})}
+\left(\int_{\sigma{T}}^{T}\int_{(1-\sigma)\omega}|B-A|^{\frac{2p}{p-2}}dxdt\right)^{\frac{p-2}{p}}
\left(\int_{\sigma{T}}^{T}\int_{(1-\sigma)\omega}|\nabla\,v|^{p}dxdt\right)^{\frac{2}{p}}\right)\\
&\leq\,C\left(\|g\|^{2}_{L^{2}(\Omega_{T})}+\|f\|^{2}_{L^{2}(\Omega_{T})}+\epsilon^{\frac{p-2}{p}}\|\nabla{v}\|^{2}_{L^{p}(\Omega_{T}^{(\sigma)})}\right)\\
&\leq\,C\left(\|g\|^{2}_{L^{2}(\Omega_{T})}+\|f\|^{2}_{L^{2}(\Omega_{T})}
+\epsilon^{\frac{p-2}{p}}\left(\|u\|^{2}_{L^{2}(\Omega_{T})}+\sum_{s\leq1}\|D_{t}^{s}g\|^{2}_{L^{2}(\Omega_{T})}
+\sum_{s\leq1}\|D_{t}^{s}f\|^{2}_{L^{2}(\Omega_{T})}\right)\right).
\end{align*}
Taking $\gamma=\frac{p-2}{2p}$, Lemma \ref{pertubationlemma} is
proved.
\end{proof}

\bigskip
\section{Estimates of $|u|$ and Preliminaries for Estimates of $|\nabla{u}|$}\label{gradient}
\bigskip

Before proving Theorem \ref{thm1}, we derive uniform $L^{\infty}$
and gradient estimates in this section. In order to estimate $|u|$
and $|\nabla{u}|$ at a point $(x,t)$ in some ${D}_{m}\times(0,T)$,
we need only consider the case that for some $m_{0}$,
$X^{0}=(x^{0},t^{0})$ is in ${D}_{m_{0}}\times(0,T)$ and close to
the lateral boundary $\partial{D}_{m_{0}}\times(0,T)$, and estimate
$|u|$ and $|\nabla{u}|$ at $X^{0}$. In that case we approximate the
problem by laminar systems with a finite number of strips.

We shall focus on a neighborhood of $X^{0}$ in this section. Without
loss of generality, we take $x^{0}$ as the origin in
$\mathbb{R}^{n}$ and $t^{0}=1<T$, that is, $X^{0}=(0,1)$. By
suitable rotation and scaling, we suppose there lie a finite number
of the $\partial{D}_{m}$ in the cube
$\omega=(\frac{-1}{2},\frac{1}{2})^{n}$, and these hypersurfaces
take the form
$$x_{n}=f_{m}(x'),\quad\quad\,x'\in\bigg(\frac{-1}{2},\frac{1}{2}\bigg)^{n-1},\quad m=1,\cdots,l,$$
and
$$-\frac{1}{2}=f_{0}(x')<f_{1}(x')<\cdots<f_{l}(x')<f_{l+1}=\frac{1}{2},$$
where $f_{m}\in\,C^{1,\alpha}([-\frac{1}{2},\frac{1}{2}]^{n-1})$,
thus we have $l+1$ regions
$${D}_{m}=\left\{x=(x',x_{n})\in\omega~\bigg|~f_{m}(x')<x_{n}<f_{m+1}(x'),x'\in(-\frac{1}{2},\frac{1}{2})^{n-1}\right\},\quad0\leq\,m\leq\,l.$$
We suppose that $f_{m_{0}}(0')<0<f_{m_{0}+1}(0')$, and
$(0',f_{m_{0}+1}(0'))$ is the closest point on $\partial{D}_{m_{0}}$
to the origin. So that
$$\nabla'f_{m_{0}+1}(0')=0.$$

After rotation and scaling, \eqref{IBVP} still have the same form,
and the coefficient conditions \eqref{coeff1} \eqref{coeff2} still
hold. We now consider the equation in
$\Omega_{1}=\omega\times(0,1)$. Denote the cylinder
$$Q(0,r)=r\omega\times(1-r^{2},1),$$
where
$$r\omega=\big\{x\in\mathbb{R}^{n}~\big|~|x_{i}|<\frac{r}{2},i=1,\cdots,n\big\},$$
then $Q(0,1)=\Omega_{1}=\omega\times(0,1)$. Our desired estimate for
$\nabla{u}(0,1)$ is given by the following:

\begin{prop}\label{prop61}
Suppose the coefficients
$A(x,t)\in{C}^{\mu,1}(\overline{{D}}_{m}\times(0,1))$ $(0<\mu<1)$
satisfy \eqref{coeff1} and \eqref{coeff2}, with
$\{{D}_{m}\}_{m=0}^{l}$ as above. If $u\in{V}(\Omega_{1})$ is a weak
solution of \eqref{IBVP}, then for
$0<\alpha'<\min\{\mu,\frac{\alpha}{2(1+\alpha)}\}$,
$$|u(0,1)|+|\nabla{u}(0,1)|\leq\,C\left(\|u\|_{L^{2}(\Omega_{1})}+\|f\|_{L^{\infty}(\Omega_{1})}
+\max_{1\leq\,m\leq\,l+1}\|g\|_{{C}^{\alpha',0}(\overline{{D}}_{m}\times(0,1))}\right),$$
where $C$ depends only on $n$, $N$, $l$, $\alpha$, $\mu$, $\lambda$,
$\Lambda_{0}$, $T$,
$\max\limits_{0\leq\,m\leq\,l}\|f_{m}\|_{C^{1,\alpha}([-\frac{1}{2},\frac{1}{2}]^{n-1})}$,
and
$\max\limits_{1\leq\,m\leq\,l}\|A\|_{{C}^{\alpha',1}(\overline{{D}}_{m}\times[0,1])}$.
\end{prop}

The $L^{\infty}$ estimate of $|u|$ in \eqref{thm1.1} and estimate
\eqref{thm1.2} in Theorem \ref{thm1} follows from Proposition
\ref{prop61}. The proof of Proposition \ref{prop61} will use the
perturbation Lemma \ref{pertubationlemma} in $Q(0,1)$. We
approximate the system ``$A$" by a laminar system with coefficients
$\overline{A}$ that are piecewise smooth functions. Precisely, we
introduce strips in $\omega$,
$$\omega_{m}=\bigg\{x\in\omega:~f_{m}(0')<x_{n}<f_{m+1}(0')\bigg\},$$
and define the coefficients $\overline{A}$ as
$$
\overline{A}(x,t)=
\begin{cases}
\mathop{\lim}\limits_{\scriptstyle \quad\,
y\in{D}_{m}\hfill\atop\scriptstyle
y\rightarrow(0',f_{m}(0'))\hfill}A^{\alpha\beta}_{ij}(y,t),&x\in\omega_{m}\times(0,1),\quad
m>m_{0},\\
\quad\quad A^{\alpha\beta}_{ij}(0,t),&x\in\omega_{m_{0}}\times(0,1),\\
\mathop{\lim}\limits_{\scriptstyle \quad\,
y\in{D}_{m}\hfill\atop\scriptstyle
y\rightarrow(0',f_{m+1}(0'))\hfill}A^{\alpha\beta}_{ij}(y,t),&x\in\omega_{m}\times(0,1),\quad
m<m_{0},
\end{cases}
$$
Using $f$ and $g$, we similarly define $\widetilde{F}$ and
$\widetilde{G}$, respectively. We measure the difference
$A-\overline{A}$ in terms of a norm $\|\cdot\|_{Y^{s,p}}$ defined
below on $Q(0,1)$.

\begin{defn}
For $s>0$, $1\leq\,p<\infty$, and any vector- or matrix-valued
function $F$, we introduce the norm
$$\|F\|_{Y^{s,p}}=\sup_{0<r<1}r^{1-s}\bigg(\fint_{Q(0,r)}|F|^{p}dxdt\bigg)^{1/p}.$$
\end{defn}

\begin{lemma}\label{lem61}
Take
$$0<\alpha'<\min\left\{\mu,\frac{\alpha}{2(\alpha+1)}\right\},$$
and $A,\overline{A}$ defined as above. Then there exists a positive
constant $E$, depending only on
$n,l,\alpha,\alpha',\lambda,\Lambda_{0}$,
$\max\limits_{0\leq\,m\leq\,l+1}\|A\|_{C^{\alpha',0}(\overline{{D}}_{m}\times(0,1))}$
and
$\max\limits_{0\leq\,m\leq\,l}\|f_{m}\|_{C^{1,\alpha}([-\frac{1}{2},\frac{1}{2}]^{n-1})}$,
such that
$$\|A-\overline{A}\|_{Y^{1+\alpha',2}}\leq\,E.$$
\end{lemma}

It can be proved in the same way as Lemma 5.2 in \cite{lv}. For
reader's convenience, we present the proof here.

\begin{proof}
Due to the definition of ${D}_{m}$ and $\omega_{m}$, and Lemma 5.1
in \cite{lv}, we have
$$r^{-n}\left|(\omega_{m}\cap{r\omega})\setminus{D}_{m}\right|\leq{C}r^{\frac{\alpha}{1+\alpha}}.$$
Then by the definition of $A$, $\overline{A}$,
\begin{align}\label{lem62.1}
&\hspace{.5cm}\left(\fint_{Q(0,r)}|A-\overline{A}|^{2}dxdt\right)^{1/2}\nonumber\\
&=\left(\frac{1}{r^{n+2}}\sum_{m}\int_{1-r^{2}}^{1}\int_{\omega_{m}\cap{r\omega}\cap{D}_{m}}|A-\overline{A}|^{2}dxdt
+\frac{1}{r^{n+2}}\sum_{m}\int_{1-r^{2}}^{1}\int_{(\omega_{m}\cap{r\omega})\setminus{D}_{m}}|A-\overline{A}|^{2}dxdt\right)^{1/2}\nonumber\\
&\leq\left(\frac{1}{r^{n+2}}\sum_{m}\int_{1-r^{2}}^{1}\int_{\omega_{m}\cap{r\omega}\cap{D}_{m}}|A-\overline{A}|^{2}dxdt\right)^{1/2}
+Cr^{\frac{\alpha}{2(1+\alpha)}},
\end{align}
where $C$ depending only on $\Lambda_{0},l$ and the $C^{1,\alpha}$
norm of $f_{m}$, $1\leq{m}\leq{l}$. The first term in the right-side
of \eqref{lem62.1} requires a slightly different estimate, depending
on whether $m<m_{0}$, $m=m_{0}$ or $m>m_{0}$. For $m<m_{0}$,
\begin{align*}
&\hspace{.5cm}\left(\frac{1}{r^{n+2}}\int_{1-r^{2}}^{1}\int_{\omega_{m}\cap{r\omega}\cap{D}_{m}}|A-\overline{A}|^{2}dxdt\right)^{1/2}\\
&=\left(\frac{1}{r^{n+2}}\int_{1-r^{2}}^{1}\int_{\omega_{m}\cap{r\omega}\cap{D}_{m}}|A(x,t)-A(0',f_{m}(0'),t)|^{2}dxdt\right)^{1/2}\\
&\leq\,C\left(\frac{1}{r^{n+2}}\int_{1-r^{2}}^{1}\int_{\omega_{m}\cap{r\omega}\cap{D}_{m}}|x-(0',f_{m}(0'))|^{2\alpha'}dxdt\right)^{1/2}
\leq\,Cr^{\alpha'},
\end{align*}
where $C$ depends only on
$\|A\|_{C^{\alpha',0}(\overline{{D}}_{m}\times(0,1))}$
$(\alpha'<\mu)$. For $m=m_{0}$,
\begin{align*}
&\hspace{.5cm}\left(\frac{1}{r^{n+2}}\int_{1-r^{2}}^{1}\int_{\omega_{m_{0}}\cap{r\omega}\cap{D}_{m_{0}}}|A-\overline{A}|^{2}dxdt\right)^{1/2}\\
&=\left(\frac{1}{r^{n+2}}\int_{1-r^{2}}^{1}\int_{\omega_{m_{0}}\cap{r\omega}\cap{D}_{m_{0}}}|A(x,t)-A(0,t)|^{2}dxdt\right)^{1/2}
\leq\,Cr^{\alpha'},
\end{align*}
and for $m>m_{0}$,
\begin{align*}
&\hspace{.5cm}\left(\frac{1}{r^{n+2}}\int_{1-r^{2}}^{1}\int_{\omega_{m}\cap{r\omega}\cap{D}_{m}}|A-\overline{A}|^{2}dxdt\right)^{1/2}\\
&=\left(\frac{1}{r^{n+2}}\int_{1-r^{2}}^{1}\int_{\omega_{m}\cap{r\omega}\cap{D}_{m}}|A(x,t)-A(0',f_{m-1}(0'),t)|^{2}dxdt\right)^{1/2}\\
&\leq\,C\left(\frac{1}{r^{n+2}}\int_{1-r^{2}}^{1}\int_{\omega_{m}\cap{r\omega}\cap{D}_{m}}|x-(0',f_{m-1}(0'))|^{2\alpha'}dxdt\right)^{1/2}
\leq\,Cr^{\alpha'}.
\end{align*}
In either case we therefore conclude from \eqref{lem62.1} that
$$\left(\fint_{Q(0,r)}|A-\overline{A}|^{2}dxdt\right)^{1/2}\leq\,C\left(r^{\alpha'}+r^{\frac{\alpha}{2(1+\alpha)}}\right)\leq\,Cr^{\alpha'}.$$
We now choose $E=C$, then the lemma follows.
\end{proof}

We now prove Proposition \ref{prop61}. The method we use here is an
adaption of that of Li-Nirenberg in \cite{ln} for the elliptic case.
See also related papers of L. Caffarelli and Caffarelli-Cabr\'{e}
\cite{caf,cc}.

\begin{proof}[Proof of Proposition \ref{prop61}]

For simplicity, we treat the case $f^{i}=g^{\alpha{i}}\equiv0$. We
will show that
\begin{equation}\label{prop61:result}
|u(0,1)|+|\nabla{u}(0,1)|\leq\,C\|u\|_{L^{2}(\Omega_{1})}.
\end{equation}
By Lemma \ref{lem61},
$$\|A-\overline{A}\|_{Y^{1+\alpha',2}}\leq\,E.$$
In fact, we can further assume that
\begin{equation}\label{prop61:epsilon}
\|A-\overline{A}\|_{Y^{1+\alpha',2}}\leq\,\epsilon_{0}.
\end{equation}
for some small enough $\epsilon_{0}>0$ (depending only on
$n,N,\lambda,\Lambda,\alpha'$, and $E$). Indeed, we pick $r_{0}$
satisfying $r_{0}^{\alpha'}(1+E)=\epsilon_{0}$ and let
$$\widehat{A}(x,t)=A(r_{0}x,r_{0}^{2}(t-1)+1),\quad\overline{\widehat{A}}(x,t)=\overline{A}(r_{0}x,r_{0}^{2}(t-1)+1),$$
and$$\widehat{u}(x,t)=u(r_{0}x,r_{0}^{2}(t-1)+1).$$ A simple
calculation yields
\begin{align*}
&\hspace{.5cm}r^{-\alpha'}\left(\fint_{Q(0,r)}\big|\widehat{A}-\overline{\widehat{A}}\big|^{2}dxdt\right)^{1/2}\\
&=r^{-\alpha'}\left(\frac{1}{r^{n+2}}\int\limits_{\scriptstyle~\ \ \
\,|x|<r\hfill\atop\scriptstyle
1-r^{2}<t<1\hfill}\left|A(r_{0}x,r_{0}^{2}(t-1)+1)-\overline{A}(r_{0}x,r_{0}^{2}(t-1)+1)\right|^{2}dxdt\right)^{1/2}\\
&=r^{-\alpha'}\left(\frac{1}{(r_{0}r)^{n+2}}\int_{Q(0,r_{0}r)}\big|A-\overline{A}\big|^{2}(y,s)dyds\right)^{1/2},
\end{align*}
so
$$\|\widehat{A}-\overline{\widehat{A}}\|_{Y^{1+\alpha',2}}\leq\,r_{0}^{\alpha'}\|A-\overline{A}\|_{Y^{1+\alpha',2}}\leq\epsilon_{0},$$
and, since $f^{i}=g^{\alpha{i}}\equiv0$,
$$\left(\widehat{u}\right)_{t}-\partial\left(\widehat{A}\partial\widehat{u}\right)=0\quad\mbox{in}\ \Omega_{1}.$$

In the following we will always assume that for sufficiently small
$\epsilon_{0}$, \eqref{prop61:epsilon} holds and $u$ is normalized
to satisfy $\|u\|_{L^{2}(\Omega_{1})}\leq 1$. We will find
$w_{k}\in{V}(Q(0,\frac{3}{4^{k+1}}))$, $k\geq0$, such that for all
$k\geq0$,
\begin{align}
(w_{k})_{t}-\partial(\overline{A}\partial{w}_{k})&=0,\quad\quad\quad Q(0,\frac{3}{4^{k+1}}),\label{6induction1}\\
\|w_{k}\|_{L^{2}(Q(0,\frac{2}{4^{k+1}}))}\leq\,C'4^{-\frac{k(n+4+2\alpha')}{2}},
&\quad\|Dw_{k}\|_{L^{\infty}(Q(0,\frac{1}{4^{k+1}}))}\leq\,C'4^{-k\alpha'},\label{6induction2}\\
\|D_{t}w_{k}\|_{L^{\infty}(Q(0,\frac{1}{4^{k+1}}))}\leq\,C'4^{-k(\alpha'-1)},&\quad
\|DD_{t}w_{k}\|_{L^{\infty}(Q(0,\frac{1}{4^{k+1}}))}\leq\,C'4^{-k(\alpha'-2)},\label{6induction3}\\
\bigg\|u-\sum_{j=0}^{k}w_{j}\bigg\|_{L^{2}(Q(0,\frac{1}{4^{k+1}}))}&\leq\,4^{-\frac{k(n+4+2\alpha')}{2}}.\label{6induction4}
\end{align}

In the proof of \eqref{6induction1}-\eqref{6induction3}, $C$, $C'$
and $\epsilon_{0}$ denote various constants that depend only on
parameters specified in the proposition. In particular, they are
independent of $k$. $C$ will be chosen first and will be large, then
$C'$ (much larger that $C$), and finally $\epsilon_{0}\in(0,1)$
(much smaller than $1/CC'$).

By Lemma \ref{pertubationlemma}, we can find
$w_{0}\in{V}(Q(0,\frac{3}{4}))$, such that
$$\left(w_{0}\right)_{t}-\partial\left(\overline{A}\partial{w}_{0}\right)=0,\quad\quad(x,t)\in Q(0,\frac{3}{4}),$$
with
$$\|u-w_{0}\|_{V(Q(0,\frac{1}{2}))}\leq\,C\epsilon_{0}^{\gamma},$$
so
$$\|u-w_{0}\|_{L^{2}(Q(0,\frac{1}{2}))}\leq\frac{1}{2}C\epsilon_{0}^{\gamma},$$
and
$$\|w_{0}\|_{L^{2}(Q(0,\frac{1}{2}))}\leq\,\|u\|_{L^{2}(Q(0,\frac{1}{2}))}+\|u-w_{0}\|_{L^{2}(Q(0,\frac{1}{2}))}\leq\,C(\leq\,C').$$
Recalling the definition of $\overline{A}$, it follows from
Corollary \ref{cor1} that
$$\|Dw_{0}\|_{L^{\infty}(Q(0,\frac{1}{4}))}\leq\,C\leq\,C',\quad\|D_{t}w_{0}\|_{L^{\infty}(Q(0,\frac{1}{4}))}\leq\,C', \quad\|DD_{t}w_{0}\|_{L^{\infty}(Q(0,\frac{1}{4}))}\leq\,C'.$$ So
far, we have verified \eqref{6induction1}-\eqref{6induction4} for
$k=0$.  Suppose that \eqref{6induction1}-\eqref{6induction4} hold up
to $k$ ($k\geq\,0$), we will prove them for $k+1$. Let
\begin{align*}
&W_{k+1}(x,t)=\left(u-\sum_{j=0}^{k}w_{j}\right)\left(\frac{x}{4^{k+1}},\frac{t-1}{4^{2(k+1)}}+1\right),\\
&A_{k+1}(x,t)=A\left(\frac{x}{4^{k+1}},\frac{t-1}{4^{2(k+1)}}+1\right),
\quad\overline{A}_{k+1}(x,t)=\overline{A}\left(\frac{x}{4^{k+1}},\frac{t-1}{4^{2(k+1)}}+1\right),\\
&g_{k+1}(x,t)=-\frac{1}{4^{k+1}}\left(A_{k+1}-\overline{A}_{k+1}\right)(x,t)\sum_{j=0}^{k}\partial{w}_{j}\left(\frac{x}{4^{k+1}},\frac{t-1}{4^{2(k+1)}}+1\right).
\end{align*}
then $W_{k+1}$ satisfies
\begin{equation}\label{equwk+1}
\big(W_{k+1}\big)_{t}-\partial\big(A_{k+1}\partial{W_{k+1}}\big)=-\partial{g}_{k+1},\quad\quad
(x,t)\in{Q}(0,1).
\end{equation}
By simple calculation, using the fact that $|Q(0,1)|=1$ and
\eqref{prop61:epsilon}, we obtain
\begin{equation}\label{prop61Ak+1}
\|A_{k+1}-\overline{A}_{k+1}\|_{L^{2}(Q(0,1))}=\bigg(\fint_{Q(0,\frac{1}{4^{k+1}})}|A-\overline{A}|^{2}\bigg)^{1/2}
\leq\frac{1}{4^{(k+1)\alpha'}}\|A-\overline{A}\|_{Y^{1+\alpha',2}}\leq\frac{\epsilon_{0}}{4^{(k+1)\alpha'}}.
\end{equation}
By Lemma \ref{pertubationlemma}, there exists
$v_{k+1}\in{V}(Q(0,\frac{3}{4}))$ such that
$$-(v_{k+1})_{t}+\partial\left(\overline{A}_{k+1}\partial{v}_{k+1}\right)=0,\quad\mbox{in}\quad Q(0,\frac{3}{4}),$$
and
\begin{align}\label{prop61.W-v}
&\hspace{.5cm}\|W_{k+1}-v_{k+1}\|_{V(Q(0,\frac{1}{2}))}\nonumber\\
&\leq\,C\left(\|g_{k+1}\|_{L^{2}(Q(0,1))}
+\left(\frac{\epsilon_{0}}{4^{(k+1)\alpha'}}\right)^{\gamma}\left(\|W_{k+1}\|_{L^{2}(Q(0,1))}+\|D_{t}g_{k+1}\|_{L^{2}(Q(0,1))}\right)\right).
\end{align}
In the following, we will estimate these three terms on the right
hand side of \eqref{prop61.W-v}. Making a change of variable and
using \eqref{6induction4}, we first have
\begin{align}\label{prop61.wk+1l2}
\|W_{k+1}\|_{L^{2}(Q(0,1))}&=\left(4^{(k+1)(n+2)}\iint_{Q(0,\frac{1}{4^{k+1}})}\big|u-\sum_{j=0}^{k}w_{j}\big|^{2}(y,s)dyds\right)^{1/2}\nonumber\\
&\leq4^{\frac{(k+1)(n+2)}{2}}4^{-\frac{k(n+4+2\alpha')}{2}}\nonumber\\
&\leq\frac{C}{4^{(k+1)(1+\alpha')}}
\end{align}
Further, we need to estimate the $L^{2}$ norm of $g_{k+1}$ and
$D_{t}g_{k+1}$. In fact, by the induction hypothesis
\eqref{6induction2} and \eqref{6induction3}, we have
\begin{align}\label{prop61.gk+1}
&\hspace{.5cm}\|g_{k+1}\|_{L^{2}(Q(0,1))}\leq\frac{1}{4^{k+1}}\|A_{k+1}-\overline{A}_{k+1}\|_{L^{2}(Q(0,1))}
\sum_{j=0}^{k}C'4^{-j\alpha'},
\end{align}
and
\begin{align}\label{prop61.gtk+1}
\|D_{t}g_{k+1}\|_{L^{2}(Q(0,1))}
&\leq\frac{1}{4^{k+1}}\|D_{t}A_{k+1}-D_{t}\overline{A}_{k+1}\|_{L^{2}(Q(0,1))}
\sum_{j=0}^{k}C'4^{-j\alpha'}\nonumber\\
&+\frac{1}{4^{k+1}}\|A_{k+1}-\overline{A}_{k+1}\|_{L^{2}(Q(0,1))}
\sum_{j=0}^{k}\frac{1}{4^{2(k+1)}}\cdot{C'}4^{-j(\alpha'-2)}.
\end{align}
Recalling the definition of $A_{k+1}$ and $\overline{A}_{k+1}$, and
using the smoothness of $A$, we have
\begin{align*}
\left(\int_{Q(0,1)}|D_{t}(A_{k+1}-\overline{A}_{k+1})|^{2}dxdt\right)^{1/2}
&\leq\frac{1}{4^{2(k+1)}}\left(\fint_{Q(0,\frac{1}{4^{k+1}})}|D_{t}(A-\overline{A}|^{2}dxdt\right)^{1/2}\\
&\leq\frac{2}{4^{2(k+1)}}\|D_{t}A\|_{L^{\infty}(Q(0,\frac{1}{4^{k+1}}))}\\
&\leq\frac{C}{4^{2(k+1)}}.
\end{align*}
Therefore, combining these estimates with \eqref{prop61Ak+1},
\eqref{prop61.gk+1} and \eqref{prop61.gtk+1}, we have
\begin{align}\label{prop61.gk+10}
&\|g_{k+1}\|_{L^{2}(Q(0,1))}\leq\frac{1}{4^{k+1}}\cdot\frac{\epsilon_{0}}{4^{(k+1)\alpha'}}\cdot\frac{4^{\alpha'}}{4^{\alpha'}-1}C'
\leq\frac{CC'\epsilon_{0}}{4^{(k+1)(1+\alpha')}},
\end{align}
and
\begin{align}\label{prop61.gtk+10}
&\|D_{t}g_{k+1}\|_{L^{2}(Q(0,1))}\leq\,\frac{CC'}{4^{3(k+1)}}.
\end{align}
So substituting \eqref{prop61.wk+1l2}, \eqref{prop61.gk+10} and
\eqref{prop61.gtk+10} into \eqref{prop61.W-v}, we have
\begin{equation}\label{prop61.Wk+1-v}
\|W_{k+1}-v_{k+1}\|_{V(Q(0,\frac{1}{2}))}\leq\,CC'\max\{\epsilon_{0}^{\gamma},\epsilon_{0}\}\cdot\frac{1}{4^{(k+1)(1+\alpha')}},
\end{equation}
and
\begin{align*}
\|v_{k+1}\|_{L^{2}(Q(0,\frac{1}{2}))}&\leq\|W_{k+1}-v_{k+1}\|_{L^{2}(Q(0,\frac{1}{2}))}+\|W_{k+1}\|_{L^{2}(Q(0,\frac{1}{2}))}\\
&\leq\,\max\{\epsilon_{0}^{\gamma},\epsilon_{0}\}\cdot\frac{CC'}{4^{(k+1)(1+\alpha')}}+\frac{C}{4^{(k+1)(1+\alpha')}}.
\end{align*}

Let
$$w_{k+1}(x,t)=v_{k+1}\left(4^{k+1}x,4^{2(k+1)}(t-1)+1\right),\quad\,(x,t)\in\,Q(0,\frac{3}{4^{k+2}}).$$
A change of variables yields \eqref{6induction1}, and
\begin{align*}
&\hspace{0.5cm}\|w_{k+1}\|_{L^{2}(Q(0,\frac{1}{4^{k+2}}))}\\
&=\left(\int_{Q(0,\frac{1}{4^{k+2}})}|v_{k+1}|^{2}\left(4^{k+1}x,4^{2(k+1)}(t-1)+1\right)dxdt\right)^{1/2}\\
&=\left(4^{-(k+1)(n+2)}\int_{Q(0,\frac{1}{4})}|v_{k+1}|^{2}\left(y,s\right)dyds\right)^{1/2}\\
&=\frac{1}{4^{\frac{(k+1)(n+2)}{2}}}\|v_{k+1}\|_{L^{2}(Q(0,\frac{1}{4}))}\\
&\leq\,C'4^{-\frac{(k+1)(n+4+2\alpha')}{2}},
\end{align*}
that is, \eqref{6induction2} is obtained for $k+1$;
\begin{align*}
&\hspace{0.5cm}\bigg\|u-\sum_{j=0}^{k+1}w_{j}\bigg\|_{L^{2}(Q(0,\frac{1}{4^{k+2}}))}\\
&=\left(4^{-(k+1)(n+2)}\int_{Q(0,\frac{1}{4})}|W_{k+1}-v_{k+1}|^{2}\left(x,t\right)dxdt\right)^{1/2}\\
&=\frac{1}{4^{\frac{(k+1)(n+2)}{2}}}\|W_{k+1}-v_{k+1}\|_{L^{2}(Q(0,\frac{1}{4}))}\\
&\leq\,4^{-\frac{(k+1)(n+4+2\alpha')}{2}},
\end{align*}
that is, \eqref{6induction4} holds for $k+1$.

Combining the above and Corollary \ref{cor1}, we have
$$\|Dv_{k+1}\|_{L^{\infty}(Q(0,\frac{1}{4}))}\leq\,C\|v_{k+1}\|_{L^{2}(Q(0,\frac{1}{2}))}\leq\frac{2C}{4^{(k+1)(1+\alpha')}},$$
$$\|D_{t}v_{k+1}\|_{L^{\infty}(Q(0,\frac{1}{4}))}\leq\,C\|v_{k+1}\|_{L^{2}(Q(0,\frac{1}{2}))}\leq\frac{2C}{4^{(k+1)(1+\alpha')}},$$
and
$$\|DD_{t}v_{k+1}\|_{L^{\infty}(Q(0,\frac{1}{4}))}\leq\,C\|v_{k+1}\|_{L^{2}(Q(0,\frac{1}{2}))}\leq\frac{2C}{4^{(k+1)(1+\alpha')}}.$$
By a change of variables, we have
$$\|Dw_{k+1}\|_{L^{\infty}(Q(0,\frac{1}{4^{k+2}}))}\leq\frac{C'}{4^{(k+1)\alpha'}},\quad\|D_{t}w_{k+1}\|_{L^{\infty}(Q(0,\frac{1}{4^{k+2}}))}\leq\frac{C'}{4^{(k+1)(\alpha'-1)}},$$
and
$$\|DD_{t}w_{k+1}\|_{L^{\infty}(Q(0,\frac{1}{4^{k+2}}))}\leq\frac{C'}{4^{(k+1)(\alpha'-2)}}.$$
Estimates \eqref{6induction2} and \eqref{6induction3} for $k+1$
follow from the above estimates. Thus we have established
\eqref{6induction1}-\eqref{6induction4} for all $k$.

An easy consequence of \eqref{6induction2} and \eqref{6induction3},
we have
\begin{equation}\label{wkLinfty}
\|w_{k}(\cdot,t)\|_{L^{\infty}(Q(0,\frac{1}{4^{k+1}}))}\leq\frac{C}{4^{k(1+\alpha')}}.
\end{equation}
For $|x|<\frac{1}{4^{k+1}}$ and $1-\frac{1}{4^{2(k+1)}}<t\leq1$,
using \eqref{6induction3} and \eqref{wkLinfty}, we have
\begin{align*}
\bigg|\sum_{j=0}^{k}w_{j}(x,t)-\sum_{j=0}^{\infty}w_{j}(0,1)\bigg|
&\leq\bigg|\sum_{j=0}^{k}w_{j}(x,t)-\sum_{j=0}^{k}w_{j}(0,1)\bigg|
+\sum_{j=k+1}^{\infty}\big|w_{j}(0,1)\big|\\
&\leq\,C\left(\sum_{j=0}^{k}\frac{|x|}{4^{j\alpha'}}+\sum_{j=0}^{k}\frac{|1-t|}{4^{j(\alpha'-1)}}\right)
+C\sum_{j=k+1}^{\infty}\frac{1}{4^{j(1+\alpha')}}\\
&\leq\,C\left(\sum_{j=0}^{k}\frac{|x|}{4^{j\alpha'}}+\sum_{j=0}^{k}\frac{\sqrt{|1-t|}}{4^{j\alpha'}}\right)
+C4^{-k(1+\alpha')}\\
&\leq\,C\left(|x|+\sqrt{|1-t|}\right)+C4^{-k}.
\end{align*}
Then
$$\bigg\|\sum_{j=0}^{k}w_{j}-\sum_{j=0}^{\infty}w_{j}(0,1)\bigg\|_{L^{2}(Q(0,\frac{1}{4^{k+1}}))}\leq\frac{C}{4^{\frac{k(n+4)}{2}}}.$$
So, in view of \eqref{6induction4}, we have
\begin{align}\label{prop61.u-w}
&\bigg\|u-\sum_{j=0}^{\infty}w_{j}(0,1)\bigg\|_{L^{2}(Q(0,\frac{1}{4^{k+1}}))}\nonumber\\
&\leq\bigg\|u-\sum_{j=0}^{k}w_{j}\bigg\|_{L^{2}(Q(0,\frac{1}{4^{k+1}}))}
+\bigg\|\sum_{j=0}^{k}w_{j}-\sum_{j=0}^{\infty}w_{j}(0,1)\bigg\|_{L^{2}(Q(0,\frac{1}{4^{k+1}}))}\nonumber\\
&\leq\frac{C}{4^{\frac{k(n+4)}{2}}}.
\end{align}
Thus, after sending $k\rightarrow\infty,$
\begin{equation}\label{prop61.u}
u(0,1)=\sum_{j=0}^{\infty}w_{j}(0,1).
\end{equation}
The estimate of $|u(0,1)|$ in \eqref{prop61:result} is established.
By Taylor expansion,
$$u(x,t)-u(0,1)=\nabla_{x}u(0,1)\cdot{x}+O(|x|^{2}+|1-t|).$$
Using \eqref{prop61.u-w} and \eqref{prop61.u}, we have
\begin{align}\label{prop61.nablau1}
\|\nabla_{x}u(0,1)\cdot{x}\|_{L^{2}(Q(0,\frac{1}{4^{k+1}}))}
&\leq\|u-u(0,1)\|_{L^{2}(Q(0,\frac{1}{4^{k+1}}))}+C\|x^{2}+|1-t|~\|_{L^{2}(Q(0,\frac{1}{4^{k+1}}))}\nonumber\\
&\leq\frac{C}{4^{\frac{k(n+4)}{2}}}.
\end{align}
Let $\mathbf{e}=\frac{\nabla_{x}u(0,1)}{|\nabla_{x}u(0,1)|}$, then
\begin{align}\label{prop61.nablau2}
\left(\mathop{\iint}\limits_{\scriptstyle \quad\,
x\cdot{\mathbf{e}}>\frac{1}{2}|x|\hfill\atop\scriptstyle
(x,t)\in{Q}(0,\frac{1}{4^{k+1}})\hfill}
\big|\nabla_{x}u(0,1)\cdot{x}\big|^{2}dxdt\right)^{1/2}
&\geq\left(\mathop{\iint}\limits_{\scriptstyle \quad\,
x\cdot{\mathbf{e}}>\frac{1}{2}|x|\hfill\atop\scriptstyle
(x,t)\in{Q}(0,\frac{1}{4^{k+1}})\hfill}\left(\frac{1}{2}\big|\nabla_{x}u(0,1)\big|\cdot\big|{x}\big|\right)^{2}dxdt\right)^{1/2}\nonumber\\
&\geq\frac{1}{C}|\nabla_{x}u(0,1)|\left(\iint_{{Q}(0,\frac{1}{4^{k+1}})}|{x}|^{2}dxdt\right)^{1/2}\nonumber\\
&=\frac{|\nabla_{x}u(0,1)|}{C}\frac{1}{4^{\frac{k(n+2)}{2}}}
\end{align}
Combining \eqref{prop61.nablau1} and \eqref{prop61.nablau2}, it
implies that
\begin{equation*}
|\nabla{u}(0,1)|\leq\,C.
\end{equation*}
Estimate \eqref{prop61:result} is established. We have completed the
proof of Proposition \ref{prop61} for $f^{i}=g^{\alpha{i}}\equiv0$.
For the general case, by the method used in Proposition 5.3 in
\cite{lv}, suppose
$$\|u\|_{L^{2}(\Omega_{1})}+\|f\|_{L^{\infty}(\Omega_{1})}
+\max_{1\leq\,m\leq\,l+1}\|g\|_{{C}^{\alpha,0}(\overline{{D}}_{m}\times(0,1))}\leq
1,$$ we can obtain the same assertion. We leave the details to the
interested readers.
\end{proof}

\bigskip
\section{H\"{o}lder Estimates of the Gradient}\label{sec7}

Theorem \ref{thm1} can be deduced from the following Proposition. We
use the notation of Section \ref{gradient}.

\begin{prop}\label{prop71}
Let $A$ be as in Section \ref{gradient}, and let
$u\in{V}(\Omega_{1})$ be a solution of
\begin{equation}\label{g=0}
(u^{i})_{t}-D_{\alpha}\bigg(A_{ij}^{\alpha\beta}(x,t)D_{\beta}u^{j}\bigg)=0\quad\quad\hbox{in}\quad
\Omega_{1}.
\end{equation}
Then for all $x\in{D}_{m_{0}}\cap\frac{1}{2}\omega$,
$$|\nabla{u}(x,1)-\nabla{u}(0,1)|\leq\,C\|u\|_{L^{2}(\Omega_{1})}|x|^{\alpha'},$$
where $\alpha'<\min\{\mu,\frac{\alpha}{2(1+\alpha)}\}$, and $C$
depends only on $n,N,l,\alpha,\epsilon$, $\lambda,\Lambda_{2},\mu$,
the $C^{1,\alpha}$ norm of $\omega_{m}$ and
$\|A\|_{C^{\alpha',1}(\overline{D}_{m}\times[0,T])}$.
\end{prop}

\subsection{Beginning of the proof of Proposition \ref{prop71}}

As explained in Section \ref{gradient} we may assume without loss of
generality that
$$\|u\|_{L^{2}(\Omega_{1})}\leq1\quad\mbox{and}\quad\|A-\overline{A}\|_{Y^{1+\alpha',2}}\leq\epsilon_{0},$$
where $\epsilon_{0}$ is the small constant in Section
\ref{gradient}. To prove the $C^{1,\alpha}$ estimate, we slightly
strengthen \eqref{6induction1}-\eqref{6induction4}. Namely, we show
that we can find $\{w_{k}\}_{k=0}^{\infty}$ in
$V(Q(0,\frac{3}{4^{k+1}}))$ such that for $k\geq0$, $w_{k}$ satisfy,
in addition to \eqref{6induction1}-\eqref{6induction4}, for any
$1-\frac{1}{4^{2(k+1)}}\leq\,t\leq1$,
\begin{equation}\label{remin61w-v7}
\left\|u(\cdot,t)-\sum_{j=0}^{k+1}w_{j}(\cdot,t)\right\|_{L^{2}(\frac{1}{4^{k+1}}\omega)}\leq\,4^{-\frac{(k+1)(n+2+2\alpha')}{2}}.
\end{equation}
and
\begin{equation}\label{d2wk7}
\|D^{2}w_{k}\|_{L^{\infty}((\omega_{m}\cap\frac{1}{4^{k+1}}\omega)\times(1-\frac{1}{4^{2(k+1)}},1))}\leq\,C4^{k(1-\alpha')}.
\end{equation}
These estimates will be used in the proof of Proposition
\ref{prop71}.

\begin{proof}[Proof of \eqref{remin61w-v7} and \eqref{d2wk7}] We will
prove those $\{w_{k}\}_{k=0}^{\infty}$, found in Proposition
\ref{prop61}, also satisfy \eqref{remin61w-v7} and \eqref{d2wk7}.
First, for $k=0$, by Lemma \ref{pertubationlemma}, we can find
$w_{0}\in{V}(Q(0,\frac{3}{4}))$, such that
$$\left(w_{0}\right)_{t}-\partial\left(\overline{A}\partial{w}_{0}\right)=0,\quad\quad(x,t)\in Q(0,\frac{3}{4}),$$
with
$$\|u-w_{0}\|_{V(Q(0,\frac{1}{2}))}\leq\,C\epsilon_{0}^{\gamma},$$
and by Corollary \ref{cor1}, we have
$$\|D^{2}w_{0}\|_{L^{\infty}(Q(0,\frac{1}{4}))}\leq\,C\|w_{0}\|_{L^{2}(Q(0,\frac{1}{2}))}\leq\,C\leq\,C'.$$
Suppose that \eqref{6induction1}-\eqref{6induction4},
\eqref{remin61w-v7} and \eqref{d2wk7} hold up to $k$ ($k\geq\,0$),
we will prove them for $k+1$. Let
\begin{align*}
&W_{k+1}(x,t)=\left(u-\sum_{j=0}^{k}w_{j}\right)\left(\frac{x}{4^{k+1}},\frac{t-1}{4^{2(k+1)}}+1\right),\\
&A_{k+1}(x,t)=A\left(\frac{x}{4^{k+1}},\frac{t-1}{4^{2(k+1)}}+1\right),
\quad\overline{A}_{k+1}(x,t)=\overline{A}\left(\frac{x}{4^{k+1}},\frac{t-1}{4^{2(k+1)}}+1\right),\\
&g_{k+1}(x,t)=-\frac{1}{4^{k+1}}\left(A_{k+1}-\overline{A}_{k+1}\right)(x,t)\sum_{j=0}^{k}\partial{w}_{j}\left(\frac{x}{4^{k+1}},\frac{t-1}{4^{2(k+1)}}+1\right).
\end{align*}
then $W_{k+1}$ satisfies
$$\label{7equwk+1}
\big(W_{k+1}\big)_{t}-\partial\big(A_{k+1}\partial{W_{k+1}}\big)=-\partial{g}_{k+1},\quad\quad
(x,t)\in{Q}(0,1).
$$
There exists $v_{k+1}\in{V}(Q(0,\frac{3}{4}))$ such that
$$-(v_{k+1})_{t}+\partial\left(\overline{A}_{k+1}\partial{v}_{k+1}\right)=0,\quad\mbox{in}\quad Q(0,\frac{3}{4}),$$
with
\begin{align*}\label{7prop61.W-v}
&\hspace{.5cm}\|W_{k+1}-v_{k+1}\|_{V(Q(0,\frac{1}{2}))}\nonumber\\
&\leq\,C\left(\|g_{k+1}\|_{L^{2}(Q(0,1))}
+\left(\frac{\epsilon_{0}}{4^{(k+1)\alpha'}}\right)^{\gamma}\left(\|W_{k+1}\|_{L^{2}(Q(0,1))}+\|D_{t}g_{k+1}\|_{L^{2}(Q(0,1))}\right)\right)\\
&\leq\,CC'\max\{\epsilon_{0}^{\gamma},\epsilon_{0}\}\cdot\frac{1}{4^{(k+1)(1+\alpha')}},
\end{align*}
and
\begin{align*}
\|v_{k+1}\|_{L^{2}(Q(0,\frac{1}{2}))}\leq\,\max\{\epsilon_{0}^{\gamma},\epsilon_{0}\}\cdot\frac{CC'}{4^{(k+1)(1+\alpha')}}+\frac{C}{4^{(k+1)(1+\alpha')}}.
\end{align*}

Let
$$w_{k+1}(x,t)=v_{k+1}\left(4^{k+1}x,4^{2(k+1)}(t-1)+1\right),\quad\,(x,t)\in\,Q(0,\frac{3}{4^{k+2}}).$$
Then we have, for $1-\frac{1}{4^{2(k+2)}}\leq\tau\leq1$
\begin{align*}
&\hspace{0.5cm}\left(\int_{\frac{1}{4^{k+2}}\omega}\left|u(x,\tau)-\sum_{j=0}^{k+1}w_{j}(x,\tau)\right|^{2}dx\right)^{1/2}\\
&=\left(4^{-(k+1)n}\int_{\frac{1}{4}\omega}|W_{k+1}-v_{k+1}|^{2}(x,\tau)dx\right)^{1/2}\\
&=\frac{1}{4^{\frac{(k+1)n}{2}}}\|W_{k+1}-v_{k+1}\|_{V(Q(0,\frac{1}{4}))}\\
&\leq\,4^{-\frac{(k+1)(n+2+2\alpha')}{2}}.
\end{align*}
Similar as the proof of \eqref{6induction2} and \eqref{6induction3},
in view of Corollary \ref{cor1}, we have
$$\|D^{2}v_{k+1}\|_{L^{\infty}(Q(0,\frac{1}{4}))}\leq\,C\|v_{k+1}\|_{L^{2}(Q(0,\frac{1}{2}))}\leq\frac{2C}{4^{(k+1)(1+\alpha')}}.$$
By a change of variables, we have \eqref{d2wk7}.
\end{proof}

Similarly as in \cite{ln}, associated with
$\overline{A}^{(m)}:=\overline{A}|_{\omega_{m}\times(0,T)}$, we
introduce a linear transformation
$N^{(m)}:\mathbb{R}^{nN}\rightarrow\mathbb{R}^{nN}$ as follow: For
$b=(b_{\alpha}^{i})\in\mathbb{R}^{nN}$
$(1\leq\alpha\leq{n},1\leq{i}\leq{N})$,
$$
\begin{array}{ccc}
  (N^{(m)}b)^{i}_{\alpha}=b^{i}_{\alpha}, & 1\leq\alpha\leq{n-1},1\leq{i}\leq{N} \\
  (N^{(m)}b)^{i}_{n}=\overline{A}^{(m)n\beta}_{ij}b^{i}_{\beta}, & 1\leq{i}\leq{N}.
\end{array}
$$
Since $(\overline{A}_{ij}^{(m)nn})$ is a positive definite
$N\times{N}$ matrix with eigenvalues in $[\lambda,\Lambda_{0}]$, it
is clear that $N^{(m)}$ is invertible and
\begin{equation}\label{lem72.-1}
\|N^{m}\|,\|(N^{(m)})^{-1}\|\leq\,C(n,N,\lambda,\Lambda_{0}).
\end{equation}
We also define linear transformations
$T^{(m)}:\mathbb{R}^{nN}\rightarrow\mathbb{R}^{nN}$ by setting
$$T^{(m)}=(N^{(m)})^{-1}N^{(m_{0})}.$$

\begin{lemma}\label{lem71}
\begin{equation}\label{lem71.1}
\nabla{u}(0,1)=\sum_{j=0}^{\infty}\nabla{w}_{j}(0,1),
\end{equation}
and for
$x\in(\frac{1}{4^{k+1}}\omega\cap\omega_{m})\setminus\frac{1}{4^{k+2}}\omega$,
\begin{equation}\label{lem71.2}
\left|\sum_{j=0}^{k}\nabla{w}_{j}(x,1)-\sum_{j=0}^{k}T^{(m)}\nabla{w}_{j}(0,1)\right|\leq\,C|x|^{\alpha'}.
\end{equation}
\end{lemma}

\begin{proof}
We first prove \eqref{lem71.1}. For
$\frac{1}{4^{k+1}}\omega\subset\omega_{m_{0}}$, it follows from
\eqref{d2wk7} that
$$\big|w_{j}(x,1)-[w_{j}(0,1)+\nabla{w}_{j}(0,1)x]\big|\leq\,4^{j(1-\alpha')}|x|^{2},\quad j\leq k,x\in\frac{1}{4^{k+1}}\omega.$$
This, and \eqref{remin61w-v7}, yield
\begin{align}\label{lem72.3}
&\left\|u(x,1)-\left[\sum_{j=0}^{k}w_{j}(0,1)+\nabla{w}_{j}(0,1)x\right]~\right\|_{L^{2}(\frac{1}{4^{k+1}}\omega)}\nonumber\\
&\leq\left\|u(x,1)-\sum_{j=0}^{k}w_{j}(x,1)\right\|_{L^{2}(\frac{1}{4^{k+1}}\omega)}+
\left\|\sum_{j=0}^{k}\left[w_{j}(x,1)+w_{j}(0,1)+\nabla{w}_{j}(0,1)x\right]~\right\|_{L^{2}(\frac{1}{4^{k+1}}\omega)}\nonumber\\
&\leq\,4^{-\frac{k(n+2+2\alpha')}{2}}+C\sum_{j=0}^{k}4^{j(1-\alpha')}\big\|~|x|^{2}~\big\|_{L^{2}(\frac{1}{4^{k+1}}\omega)}\nonumber\\
&\leq\,C4^{-\frac{k(n+2)}{2}}.
\end{align}
From \eqref{wkLinfty} and \eqref{6induction2}, we know that
$\sum_{j=0}^{\infty}w_{j}(0,1)$ and
$\sum_{j=0}^{\infty}\nabla{w}_{j}(0,1)$ are convergent and
\begin{align}\label{lem72.4}
\left|\sum_{j=0}^{\infty}w_{j}(0,1)-\sum_{j=0}^{k}w_{j}(0,1)\right|&\leq\,C\frac{1}{4^{k(1+\alpha')}},
\end{align}
\begin{align}\label{lem72.5}
\left|\sum_{j=0}^{\infty}\nabla{w}_{j}(0,1)-\sum_{j=0}^{k}\nabla{w}_{j}(0,1)\right|&\leq\,C\frac{1}{4^{k\alpha'}}.
\end{align}
Combining \eqref{lem72.3}, \eqref{lem72.4} and \eqref{lem72.5}, we
have
$$\left\|u(x,1)-\left[\sum_{j=0}^{\infty}w_{j}(0,1)+\sum_{j=0}^{\infty}\nabla{w}_{j}(0,1)x\right]~\right\|_{L^{2}(\frac{1}{4^{k+1}}\omega)}
\leq\frac{C}{4^{k(n+2)/2}}.$$ Equation \eqref{lem71.1} follows from
the above.

Next we prove \eqref{lem71.2}. The matching condition of $w_{j}$ at
$x_{n}=c_{m-1}$ is, for all $x'\in(-\frac{1}{2},\frac{1}{2})^{n-1}$,
\begin{equation}\label{lem71.5}
N^{(m)}\nabla{w}_{j}^{(m)}(x',c_{m-1},1)=N^{(m-1)}\nabla{w}_{j}^{(m-1)}(x',c_{m-1},1),
\end{equation}
where $w_{j}^{(m)}(\cdot,1)=w_{j}(\cdot,1)|_{\omega_{m}}$.

For $m=m_{0}$, \eqref{lem71.2} follows from \eqref{d2wk7}. We will
only show \eqref{lem71.2} for $m\geq\,m_{0}+1$ since the proof is
the same for $m\leq\,m_{0}-1$. For
$x=(x',x_{n})\in\frac{1}{4^{k+1}}\omega\cap\omega_{m}\setminus\frac{1}{4^{k+2}}\omega$,
$m\geq\,m_{0}+1$, we have
\begin{align*}
&\sum_{j=0}^{k}\left|\nabla{w}_{j}^{(m)}(x,1)-T^{(m)}\nabla{w}_{j}(0,1)\right|\\
&\leq\sum_{j=0}^{k}\left|\nabla{w}_{j}^{(m)}(x,1)-\nabla{w}_{j}^{(m)}(0',c_{m-1},1)\right|
+\sum_{j=0}^{k}\left|\nabla{w}_{j}^{(m)}(0',c_{m-1},1)-T^{(m)}\nabla{w}_{j}(0,1)\right|.
\end{align*}
By \eqref{d2wk7},
$$\left|\nabla{w}_{j}^{(m)}(x,1)-\nabla{w}_{j}^{(m)}(0',c_{m-1},1)\right|\leq\,C4^{j(1-\alpha')}(|x'|+x_{n}-c_{m-1})\leq\,C4^{j(1-\alpha')}|x|.$$
By \eqref{lem72.-1}, \eqref{lem71.5}, and \eqref{d2wk7},
\begin{align*}
&\big|\nabla{w}_{j}^{(m)}(0',c_{m-1},1)-T^{(m)}\nabla{w}_{j}(0,1)\big|\\
&\leq\,C\,\big|N^{(m)}\nabla{w}_{j}^{(m)}(0',c_{m-1},1)-N^{(m_{0})}\nabla{w}^{(m_{0})}_{j}(0,1)\big|\\
&\leq\,C\sum_{i=m_{0}+2}^{m}\big|N^{(i)}\nabla{w}_{j}^{(i)}(0',c_{i-1},1)-N^{(i-1)}\nabla{w}_{j}^{(i-1)}(0',c_{i-2},1)\big|\\
&\hspace{.5cm}+C\,\big|N^{(m_{0}+1)}\nabla{w}_{j}^{(m_{0}+1)}(0',c_{m_{0}},1)-N^{(m_{0})}\nabla{w}^{(m)}_{j}(0,1)\big|\\
&\leq\,C\sum_{i=m_{0}+2}^{m}\big|N^{(i-1)}\nabla{w}_{j}^{(i-1)}(0',c_{i-1},1)-N^{(i-1)}\nabla{w}_{j}^{(i-1)}(0',c_{i-2},1)\big|\\
&\hspace{.5cm}+C\,\big|N^{(m_{0})}\nabla{w}_{j}^{(m_{0})}(0',c_{m_{0}},1)-N^{(m_{0})}\nabla{w}^{(m)}_{j}(0,1)\big|\\
&\leq\,C\left(\sum_{i=m_{0}+2}^{m}4^{j(1-\alpha')}(c_{i-1}-c_{i-2})+4^{j(1-\alpha')}(c_{m_{0}}-0)\right)\\
&=C4^{j(1-\alpha')}c_{m-1}\\
&\leq\,C4^{j(1-\alpha')}|x|.
\end{align*}
It follows that
$$\sum_{j=0}^{k}|\nabla{w}_{j}^{(m)}(x,1)-T^{(m)}\nabla{w}_{j}(0,1)|\leq\,C4^{k(1-\alpha')}|x|\leq\,C4|x|^{\alpha'}.$$
Estimate \eqref{lem71.2} is established; so is Lemma \ref{lem71}.
\end{proof}

\begin{lemma}
Let $\bar{x}$ be on the $x_{n}$-axis and
$\bar{x}+a|\bar{x}|\omega\subset(D_{m+1}\cap\omega_{m+1})$ for some
$a>0$. Then
\begin{equation}\label{lem72.1}
\left|\nabla{u}(y,1)-\sum_{j=0}^{k}\nabla{w}_{j}(y,1)\right|\leq\,C(a)|\bar{x}|^{\alpha'},\quad\,y\in\bar{x}+\frac{a}{2}|\bar{x}|\omega,
\end{equation}
where $k$ satisfies $4^{-(k+2)}\leq|\bar{x}|<4^{-(k+1)}$;
consequently,
\begin{equation}\label{lem72.2}
\left|\nabla{u}(y,1)-\nabla{u}(z,1)\right|\leq\,C(a)|\bar{x}|^{\alpha'},\quad\,y,z\in\bar{x}+\frac{a}{2}|\bar{x}|\omega,
\end{equation}
\end{lemma}

\begin{proof}
Let
$$\widehat{w}(y,t)=u\left(\bar{x}+a|\bar{x}|y,(a|\bar{x}|)^{2}(t-1)+1\right)-\sum_{j=0}^{k}w_{j}\left(\bar{x}+a|\bar{x}|y,(a|\bar{x}|)^{2}(t-1)+1\right),\quad y\in\omega.$$
Then $\widehat{w}$ satisfies the following systems
$$\widehat{w}_{t}-\partial\left(A\left(\bar{x}+a|\bar{x}|\cdot,(a|\bar{x}|)^{2}(t-1)+1\right)\partial\widehat{w}\right)=\partial\widehat{g}\quad\mbox{in}\ \omega\times(0,1),$$
where
\begin{align*}
\widehat{g}=-a|\bar{x}|\sum_{j=0}^{k}\bigg(&A^{(m+1)}\left(\bar{x}+a|\bar{x}|y,(a|\bar{x}|)^{2}(t-1)+1\right)\\
&-A^{(m+1)}\left(0',c_{m},(a|\bar{x}|)^{2}(t-1)+1\right)\bigg)
\cdot\partial{w}_{j}\left(\bar{x}+a|\bar{x}|y,(a|\bar{x}|)^{2}(t-1)+1\right),
\end{align*}
with
$A^{(m+1)}(\cdot,(a|\bar{x}|)^{2}(t-1)+1):=A(\cdot,(a|\bar{x}|)^{2}(t-1)+1)|_{D_{m+1}\times(0,1)}$.
Since $\bar{x}+a|\bar{x}|\omega\in(D_{m+1}\cap\omega_{m+1})$, the
$C^{\mu}(\omega)$-seminorm of
$A^{(m+1)}(\bar{x}+a|\bar{x}|\cdot,(a|\bar{x}|)^{2}(t-1)+1)$ is
bounded by $C(a)|\bar{x}|^{\mu}$. Thus, by \eqref{6induction3} and
\eqref{d2wk7},
$$\|\widehat{g}\|_{C^{\mu}(\omega)}\leq\,C(a)|\bar{x}|^{1+\mu}.$$
We also deduced from \eqref{remin61w-v7} that, for
$1-\frac{1}{4^{2(k+1)}}<t\leq1$,
$$\|\widehat{w}(\cdot,t)\|_{L^{2}(\omega)}\leq\,C(a)|\bar{x}|^{1+\mu}$$
By the Schauder theory,
$$\|\nabla\widehat{w}(\cdot,1)\|_{L^{\infty}(\frac{1}{2}\omega)}\leq\,C(a)|\bar{x}|^{1+\alpha'}$$
Estimate \eqref{lem72.1} follows from the above. Estimate
\eqref{lem72.2} follows from \eqref{lem72.1} and \eqref{d2wk7}.
\end{proof}

\subsection{Completion of the Proof of Proposition \ref{prop71}}

For some small $r_{1}$, depending only on the parameters specified
in Proposition \ref{prop71}, if $x$ satisfies $|x|\geq{r}_{1}$, the
desired estimate in Proposition \ref{prop71} follows from the
gradient estimates in Proposition \ref{prop61}. So we always assume
that $x\in{D}_{m_{0}}\setminus\{0\}$ and $|x|<r_{1}$. In the
following we repeatedly use the smallness of $|x|$. We select an
$\bar{x}$ as follows. If $c_{m_{0}}>80|x|$, set $\bar{x}=(0',10|x|)$
(and $m=m_{0}$), otherwise let $m\geq{m}_{0}+1$ be the smallest
index for which $c_{m+1}-c_{m}>80|x|$, and set
$\bar{x}=(0',c_{m}+10|x|)$. Clearly,
$10|x|\leq|\bar{x}|\leq\,100(l+1)|x|$ and
$\bar{x}+a|x|\omega\subset{D}_{m+1}\cap\omega_{m+1}$, with $a=8$.
With this choice of $\bar{x}$, let $k$ satisfy
$\frac{1}{4^{k+2}}\leq|\bar{x}|\leq\frac{1}{4^{k+1}}$. Then by
\eqref{lem71.1} \eqref{lem71.2} and \eqref{lem72.1}, we have
\begin{align}\label{prop7.1}
&\left|\nabla{u}(\bar{x},1)-T^{(m)}\nabla{u}(0,1)\right|\nonumber\\
&\leq\left|\nabla{u}(\bar{x},1)-\sum_{j=0}^{k}\nabla{w}_{j}(\bar{x},1)\right|
+\left|\sum_{j=0}^{k}\nabla{w}_{j}(\bar{x},1)-\sum_{j=0}^{\infty}T^{(m)}\nabla{w}_{j}(0,1)\right|\nonumber\\
&\leq\,C|\bar{x}|^{\alpha'}\leq\,C|x|^{\alpha'}.
\end{align}
Let $z$ be on either the graph of $f_{m_{0}}$ or $f_{m_{0}+1}$, so
that the distance of $x$ to $z$ is the least distance of $x$ to the
union of graphs of $\{f_{i}\}$. Let $L$ be the line passing through
$z$ that is normal to this graph. Clearly $x\in{L}$. Let $z^{(j)}$
denote the intersection of $L$ with the graph of $f_{i}$ for
$m_{0}\leq{j}\leq\,m+1$. Using the smallness of $|x|$ and the
$C^{\alpha'}$ property of $\{f_{i}\}$, it is not difficult to see
that
\begin{equation}\label{prop7.2}
|z^{(j)}-(0',f_{j}(0'))|\leq\,4|x|,\quad\,m_{0}\leq{j}\leq\,m,
\end{equation}
and
$$|z^{(m+1)}-z^{(m)}|\geq\,40|x|.$$
Here $m$ is as defined before, and we have used the fact that the
point $(0',f_{m_{0}}(0'))$ is the projection of the origin onto the
graph of the function $f_{m_{0}}$. The same argument shows that we
can find $\bar{z}$ on the segment determined by $z^{(m)}$ and
$z^{(m+1)}$ with $|\bar{z}-z^{(m)}|=10|x|$ such that
$$|\nabla{u}(\bar{z},1)-\widetilde{T}^{(m)}\nabla{u}(x,1)|\leq\,C|x|^{\alpha'}$$
where the $\{\widetilde{T}^{(m)}\}$ are defined in the natural way.
Due to \eqref{prop7.2} and the H\"{o}lder continuity of $A^{(j)}$,
we have
$$|T^{(m)}-\widetilde{T}^{(m)}|\leq\,C|x|^{\mu},$$
so
\begin{equation}\label{prop7.3}
|\nabla{u}(\bar{z},1)-T^{(m)}\nabla{u}(x,1)|\leq\,C|x|^{\alpha'}.
\end{equation}
It is easy to see, by the smallness of $r_{1}$ and H\"{o}lder
continuity of $\{\nabla{f}_{j}\}$, that
$$|\bar{x}-\bar{z}|\leq2|x|.$$
By \eqref{lem72.2},
\begin{equation}\label{prop7.4}
|\nabla{u}(\bar{x},1)-\nabla{u}(\bar{z},1)|\leq\,C|\bar{x}|^{\alpha'}\leq\,C|x|^{\alpha'}.
\end{equation}
A combination of \eqref{prop7.1}, \eqref{prop7.3}, \eqref{prop7.4}
and \eqref{lem72.-1} yields
$$|\nabla{u}(x,1)-\nabla{u}(0,1)|\leq\,C|T^{(m)}[\nabla{u}(x,1)-\nabla{u}(0,1)]|\leq\,C|x|^{\alpha'}.$$
Proposition \ref{prop71} is established.

Similarly, we can prove the following more general proposition; we
leave the details to the interested reader.

\begin{prop}\label{prop72}
Let $A$ be as in Section \ref{gradient}, and let
$u\in{V}(\Omega_{1})$ be a solution of
\begin{equation*}
(u^{i})_{t}-D_{\alpha}\bigg(A_{ij}^{\alpha\beta}(x,t)D_{\beta}u^{j}\bigg)=-D_{\alpha}g^{\alpha\,i}+f^{i}\quad\quad\hbox{in}\quad
\Omega_{1}.
\end{equation*}
Then for all $x\in{D}_{m_{0}}\cap\frac{1}{2}\omega$,
$$|\nabla{u}(x,1)-u(0,1)|\leq\,C\left(\|u\|_{L^{2}(\Omega_{T})}
+\|f\|_{L^{\infty}(\Omega_{T})}+\max_{1\leq\,m\leq\,L}\|g\|_{C^{\alpha',0}(\overline{\omega_{m}}\times[0,T])}\right)|x|^{\alpha'},$$
where $\alpha'<\min\{\mu,\frac{\alpha}{2(1+\alpha)}\}$, and $C$
depends only on $n,N,l,\alpha,\epsilon$, $\lambda,\Lambda_{0},\mu$,
$\|A\|_{C^{\alpha',1}(\overline{\omega_{m}}\times[0,T])}$ and the
$C^{1,\alpha}$ norm of $\omega_{m}$.
\end{prop}

\subsection{Proof of Theorem \ref{thm2}}

Here, for simplicity, we still only treat the case
$f^{i}=g^{\alpha{i}}\equiv0$.

\begin{proof}[Proof of Theorem \ref{thm2}]
Since
$A_{ij}^{\alpha\beta}\in{C}^{\mu,k}(\overline{{D}}_{m}\times[0,T])$,
applying $D_{t}$ to \eqref{g=0} and denoting $v=u_{t}$, then we have
\begin{equation}\label{g=0t}
(v^{i})_{t}-D_{\alpha}\bigg(A_{ij}^{\alpha\beta}(x,t)D_{\beta}v^{j}\bigg)
=D_{\alpha}\bigg(D_{t}A_{ij}^{\alpha\beta}(x,t)D_{\beta}u^{j}\bigg)\quad\quad\hbox{in}\quad
Q_{T}.
\end{equation}
Since $A_{t}\in{C^{\mu,k-1}}(\overline{{D}}_{m}\times[0,T])$ and
$\partial{u}\in{C}^{\alpha',0}(({D}_{\epsilon}\cap\overline{{D}}_{m})\times(\epsilon{T},T))$,
it follows that
$A_{t}\partial{u}\in{C}^{\alpha',0}(({D}_{\epsilon}\cap\overline{{D}}_{m})\times(\epsilon{T},T))$.
Then for $l=1$, apply Theorem \ref{thm1} to \eqref{g=0t}, and in
view of Lemma \ref{lemin2}, we have, for any
$0<\epsilon<\frac{1}{2}$ and
$\alpha'<\min\{\mu,\frac{\alpha}{2(1+\alpha)}\}$,
\begin{align}\label{thm2.v}
&\|v\|_{L^{\infty}({D}_{\epsilon}\times(\epsilon{T},T))}+
\|D_{x}v\|_{L^{\infty}({D}_{\epsilon}\times(\epsilon{T},T))}+
\|D_{x}v\|_{C^{\alpha',0}(({D}_{\epsilon}\cap\overline{{D}}_{m})\times(\epsilon{T},T))}\nonumber\\
&\leq\,C\left(\|v\|_{L^{2}({D}_{\frac{\epsilon}{2}}\times(\frac{\epsilon}{2}{T},T))}
+\max_{1\leq\,m\leq\,L}\|A_{t}D{u}\|_{C^{\alpha',0}(({D}_{\frac{\epsilon}{2}}\cap\overline{{D}}_{m})\times[\frac{\epsilon}{2}{T},T])}\right)\nonumber\\
&\leq\,C\left(\|u_{t}\|_{L^{2}({D}_{\frac{\epsilon}{2}}\times(\frac{\epsilon}{2}{T},T))}
+\max_{1\leq\,m\leq\,L}\|D{u}\|_{C^{\alpha',0}({D}_{\frac{\epsilon}{2}}\cap\overline{{D}}_{m})\times[\frac{\epsilon}{2}{T},T])}\right)\nonumber\\
&\leq\,C\|u\|_{L^{2}(Q_{T})},
\end{align}
where $C$ depends only on
$n,N,L,\alpha,\epsilon,\lambda,\Lambda_{2},\mu,T$,$\|A\|_{C^{\alpha',2}(\overline{{D}}_{m}\times[0,T])}$
and the $C^{1,\alpha}$ norm of ${D}_{m}$. Estimate \eqref{thm2.1}
for $l=1$ is proved.   To prove \eqref{thm2.1} for $l=2$, we apply
Lemma \ref{lemin2} to \eqref{g=0t} to obtain
\begin{align*}
&\|v_{t}\|_{L^{2}({D}_{\frac{\epsilon}{2}}\times(\frac{\epsilon}{2}{T},T))}\\
&\leq\,C\left(\|v\|_{L^{2}({D}_{\frac{\epsilon}{4}}\times(\frac{\epsilon}{4}{T},T))}
+\|(D_{t}A)\nabla{u}\|_{L^{2}({D}_{\frac{\epsilon}{4}}\times(\frac{\epsilon}{4}{T},T))}
+\|D_{t}\left[(\partial_{t}A)\nabla{u}\right]\|_{L^{2}({D}_{\frac{\epsilon}{4}}\times(\frac{\epsilon}{4}{T},T))}\right),
\end{align*}
where $C$ depends only on $n,N,\epsilon,\lambda,\Lambda_{2},T$.
Since
$$D_{t}\left[\left(D_{t}A\right)\nabla{u}\right]=(D_{tt}A)\nabla{u}+(D_{t}A)D_{t}(\nabla{u})$$
it follows, from Lemma \ref{lemin1} and Lemma \ref{lemin2}, that
\begin{align*}
\|(D_{t}A)D_{t}(\nabla{u})\|_{L^{2}({D}_{\frac{\epsilon}{4}}\times(\frac{\epsilon}{4}{T},T))}
&=\|(D_{t}A)\nabla{v}\|_{L^{2}({D}_{\frac{\epsilon}{4}}\times(\frac{\epsilon}{4}{T},T))}\\
&\leq\,C\,\|\nabla{v}\|_{L^{2}({D}_{\frac{\epsilon}{4}}\times(\frac{\epsilon}{4}{T},T))}\\
&\leq\,C\left(\|v\|_{L^{2}({D}_{\frac{\epsilon}{8}}\times(\frac{\epsilon}{8}{T},T))}
+\|(D_{t}A)\nabla{u}\|_{L^{2}({D}_{\frac{\epsilon}{8}}\times(\frac{\epsilon}{8}{T},T))}\right)\\
&\leq\,C\left(\|u_{t}\|_{L^{2}({D}_{\frac{\epsilon}{8}}\times(\frac{\epsilon}{8}{T},T))}
+\|\nabla{u}\|_{L^{2}({D}_{\frac{\epsilon}{8}}\times(\frac{\epsilon}{8}{T},T))}\right)\\
&\leq\,C\,\|u\|_{L^{2}(Q_{T})},
\end{align*}
where $C$ depends only on $n,N,\epsilon,\lambda,\Lambda_{2},T$. Thus
we have shown
\begin{align}\label{thm2.vt}
\|v_{t}\|_{L^{2}({D}_{\frac{\epsilon}{2}}\times(\frac{\epsilon}{2}{T},T))}
&\leq\,C\left(\|u_{t}\|_{L^{2}({D}_{\frac{\epsilon}{4}}\times(\frac{\epsilon}{4}{T},T))}
+\|\nabla{u}\|_{L^{2}({D}_{\frac{\epsilon}{4}}\times(\frac{\epsilon}{4}{T},T))}
+\|u\|_{L^{2}(Q_{T})}\right)\nonumber\\
&\leq\,C\|u\|_{L^{2}(Q_{T})},
\end{align}
where $C$ depends only on $n,N,\epsilon,\lambda,\Lambda_{4},T$.

Apply $D_{t}$ to \eqref{g=0t} and write $w=v_{t}(=u_{tt})$, we have
\begin{equation}\label{g=0tt}
(w^{i})_{t}-D_{\alpha}\bigg(A_{ij}^{\alpha\beta}D_{\beta}w^{j}\bigg)
=D_{\alpha}\bigg(D_{tt}A_{ij}^{\alpha\beta}D_{\beta}u^{j}+2D_{t}A_{ij}^{\alpha\beta}D_{\beta}v^{j}\bigg)
\quad\hbox{in}~ Q_{T}.
\end{equation}
Apply Theorem \ref{thm1} to \eqref{g=0tt}, and in combination with
\eqref{thm2.v} and \eqref{thm2.vt}, we have, for any
$0<\epsilon<\frac{1}{2}$ and
$\alpha'<\min\{\mu,\frac{\alpha}{2(1+\alpha)}\}$,
\begin{align*}
&\|w\|_{L^{\infty}({D}_{\epsilon}\times(\epsilon{T},T))}+
\|D_{x}w\|_{L^{\infty}({D}_{\epsilon}\times(\epsilon{T},T))}+
\|D_{x}w\|_{C^{\alpha',0}(({D}_{\epsilon}\cap\overline{{D}}_{m})\times(\epsilon{T},T))}\nonumber\\
&\leq\,C\left(\|w\|_{L^{2}({D}_{\frac{\epsilon}{2}}\times(\frac{\epsilon}{2}{T},T))}
+\max_{1\leq\,m\leq\,L}\left(\|A_{tt}D{u}\|_{C^{\alpha',0}(({D}_{\frac{\epsilon}{2}}\cap\overline{{D}}_{m})\times[\frac{\epsilon}{2}{T},T])}
+\|A_{t}D{v}\|_{C^{\alpha',0}(({D}_{\frac{\epsilon}{2}}\cap\overline{{D}}_{m})\times[\frac{\epsilon}{2}{T},T])}\right)\right)\nonumber\\
&\leq\,C\left(\|v_{t}\|_{L^{2}({D}_{\frac{\epsilon}{2}}\times(\frac{\epsilon}{2}{T},T))}
+\max_{1\leq\,m\leq\,L}\left(\|D{u}\|_{C^{\alpha',0}({D}_{\frac{\epsilon}{2}}\cap\overline{{D}}_{m})\times[\frac{\epsilon}{2}{T},T])}
+\|D{v}\|_{C^{\alpha',0}(({D}_{\frac{\epsilon}{2}}\cap\overline{{D}}_{m})\times[\frac{\epsilon}{2}{T},T])}\right)\right)\nonumber\\
&\leq\,C\|u\|_{L^{2}(Q_{T})},
\end{align*}
where $C$ depends only on
$n,N,L,\alpha,\epsilon,\lambda,\Lambda_{4},\mu,T$,$\|A\|_{C^{\alpha',3}(\overline{{D}}_{m}\times[0,T])}$
and the $C^{1,\alpha}$ norm of ${D}_{m}$. Estimate \eqref{thm2.1}
for $l=2$ is proved. For general $k$ and $l>2$, repeating this
process by induction, we proved Theorem \ref{thm2}.
\end{proof}

\section{Appendix}

Let $\omega\subset\mathbb{R}^{n}$ be a bounded open set with Lipschitz
boundary $\partial\omega$. For $0<\lambda<\Lambda<\infty$, $T>0$,
$\mathscr{A}(\lambda,\Lambda)$ denotes the class of measurable
vector-valued functions $(A^{\alpha\beta}_{ij}(x,t))$ satisfying
\eqref{coeff1} and \eqref{coeff2}, with $\Lambda_{0}=\Lambda$ and $Q_{T}=\omega\times(0,T)$. Consider, for $T>0$, $g\in
L^{2}(\omega\times(0,T),\mathbb{R}^{N})$ and $f\in
L^{2}(0,T,L^{2}(\omega,\mathbb{R}^{N}))$,
\begin{equation}\label{appendix1}
\begin{cases}
u_{t}^{i}-\partial_{x_{\alpha}}(A_{ij}^{\alpha\beta}(x,t)\partial_{x_{\beta}}u^{j})=\partial_{x_{\alpha}}g_{\alpha}^{i}(z)+f^{i}(z)
&\mbox{in}~\omega\times(0,T),\quad\forall\,i,\\
u=0&\mbox{on}~(\partial\omega\times(0,T))\cup(\omega\times\{0\}).
\end{cases}
\end{equation}
In the following we use notation $2^{*}=\frac{2n}{n-2}$ if $n\geq3$, $2^{*}=\infty$ if $n=1,2$.

\begin{theorem}\label{reverseholder}
For $n\geq1$, $N\geq1$, $0<\lambda\leq\Lambda<\infty$, $T>0$, and $A\in\mathscr{A}(\lambda,\Lambda)$, let
$u\in\overset{\circ}{V}(\omega\times(0,T),\mathbb{R}^{N}))$ be a
weak solution of \eqref{appendix1}. Then there exists a $2<p_{0}<2^{*}$,
depending only on $n,N,\lambda$ and $\Lambda$, such
that if $p\in[2,p_{0})$, $g\in
L^{p}(\omega\times(0,T),\mathbb{R}^{N})$ and $f\in
L^{p}(0,T,L^{2}(\omega,\mathbb{R}^{N}))$, then $u\in
L^{p}(0,T,W_{0}^{1,p}(\omega,\mathbb{R}^{N}))$. Moreover,
$$\int_{0}^{T}\int_{\omega}\bigg(|u|^{p}+|\nabla{u}|^{p}\bigg)\,dxdt
\leq\,C\left(\int_{0}^{T}\bigg(\int_{\omega}|f|^{2}dx\bigg)^{\frac{p}{2}}dt+\int_{0}^{T}\int_{\omega}|g|^{p}dxdt\right),$$
where $C$ depends only on $\lambda,\Lambda,n,N,\omega$ and $T$.
\end{theorem}

\begin{remark}
The above theorem was proved by Campanato in \cite{camp} for strongly parabolic systems, i.e.,
under
\begin{equation}\label{coeff5}
\lambda|\xi|^{2}\leq\,A_{ij}^{\alpha\beta}(x,t)\xi_{\alpha}^{i}\xi_{\beta}^{j}\leq\Lambda|\xi|^{2},\quad\forall\xi\in\mathbb{R}^{nN}.
\end{equation}
The method in \cite{camp} does not
apply under the weaker hypotheses \eqref{coeff1} and \eqref{coeff2}.
\end{remark}

We introduce some standard notations:
$$B_{R}(x_{0})=\{x\in\mathbb{R}^{n}:~|x-x_{0}|<R\},$$
$$z_{0}=(x_{0},t_{0}),\quad\,Q_{R}(z_{0})=B_{R}(x_{0})\times(t_{0}-R^{2},t_{0}),$$
$$\fint_{Q_{R}(z_{0})}fdz=\frac{1}{|Q_{R}(z_{0})|}\int_{Q_{R}(z_{0})}fdz.$$
When no confusion may arise, we shall omit $x_{0}$ and $z_{0}$ in
the notations. In the following we use $Q$ to denote some parabolic cube, i.e. for some
$(\bar{x}_{0},\bar{t}_{0})\in\mathbb{R}^{n+1}$ and $\bar{R}>0$,
$$Q=\left\{(x,t)\in\mathbb{R}^{n+1}~:~\bar{t}_{0}-\bar{R}^{2}<{t}<\bar{t}_{0},~|x_{i}-\bar{x}_{0i}|<\bar{R},~i=1,2,\cdots,n \right\}.$$

\subsection{Theorem \ref{reverseholder} when $f\equiv0$}
In this subsection we  prove Theorem \ref{reverseholder} when
$f\equiv0$. The proof relies on the following result that can be
proved exactly the same way as for the analogous elliptic one in \cite{gi} [see
Proposition 1.1 in Chapter V there] by simply changing Euclidean cubes to
parabolic cubes. For $\theta>0$ small, the elliptic one was proved
in \cite{gm}.

\begin{prop}\label{lq}
Let $Q$ be a parabolic cube, $q>1$, $0<\theta<1$, and let $h,H$ be
two nonnegative functions in ${L}^{q}(Q)$. Suppose
$$\fint_{Q_{R}(z_{0})}h^{q}dz\leq{b}\left(\fint_{Q_{4R}(z_{0})}hdz\right)^{q}+\fint_{Q_{4R}(z_{0})}H^{q}dz+\theta\fint_{Q_{4R}(z_{0})}h^{q}dz,$$
for every $Q_{4R}(z_{0})\subset{Q}$. Then there
exist constants $\epsilon>0$ and $C>0$, depending only on
$b,q,\theta,n$ and $|Q|$, such that for all $p\in[q,q+\epsilon)$ and
all $Q_{4R}(z_{0})\subset{Q}$,
$$\left(\fint_{Q_{R}(z_{0})}h^{p}dz\right)^{\frac{1}{p}}\leq\,C
\left(\left(\fint_{Q_{4R}(z_{0})}h^{q}dz\right)^{\frac{1}{q}}+\left(\fint_{Q_{4R}(z_{0})}H^{p}dz\right)^{\frac{1}{p}}\right).$$
\end{prop}

This result was used in \cite{gs} to derive partial regularity of solutions of
some nonlinear parabolic systems satisfying strongly elliptic condition
\eqref{coeff5}. We first establish interior estimates in Theorem
\ref{appendix1} when $f\equiv0$.

\begin{prop}\label{prop73}
Under the hypotheses of Theorem \ref{appendix1} with $f\equiv0$, there exists $2<p_{0}<2^{*}$, depending only on $n,N,\lambda$ and $\Lambda$, such that for
any ${Q}_{5R}\subset\omega\times(0,T)$, and
$2\leq{p}<{p}_{0}$, we have
\begin{equation}\label{prop73.0}
\fint_{Q_{R}}|\nabla{u}|^{p}dz\leq\,C\left(\fint_{Q_{4R}}|\nabla{u}|^{2}\right)^{\frac{p}{2}}+\fint_{Q_{4R}}|g|^{p}dz,
\end{equation}
where $C$ depends only on $n,N,\lambda$, and $\Lambda$.
\end{prop}

Let $\chi(x)$ be a function in $C^{\infty}_{c}(B_{2}(x_{0}))$ such
that $0\leq\chi\leq1$, $\chi\equiv1$ in $B_{1}(x_{0})$ and
$|\nabla\chi|\leq2$. We denote
$$\chi_{2R}(x)=\chi\left(x_{0}+\frac{x-x_{0}}{R}\right),$$
and let $\tau_{2R}\in{C}^{\infty}([t_{0}-(2R)^{2},t_{0}])$,
$$0\leq\tau_{2R}\leq1,\quad~\tau_{2R}=1~\mbox{on}~[t_{0}-R^{2},t_{0}],\quad\tau_{2R}=0~\mbox{on}~[t_{0}-(2R)^{2},t_{0}-(\frac{3}{2}R)^{2}],$$
and they satisfy
$$|\nabla\chi_{2R}|\leq\frac{C(n)}{R},\quad|\nabla\tau_{2R}|\leq\frac{C(n)}{R^{2}}.$$
We note that
\begin{equation}\label{chi2R}
\int_{B_{2R}}\chi_{2R}^{4}=\frac{1}{R^{n}}\int_{B_{2}}\chi^{4},\quad\quad\quad\int_{B_{2R}}\chi_{2R}^{2}=\frac{1}{R^{n}}\int_{B_{2}}\chi^{2}.
\end{equation}

Define the weighted means of $u(x,t)$ in $B_{2R}(x_{0})$ as
$$\bar{u}(t):=\bar{u}_{x_{0},2R}(t)=\frac{\int_{B_{2R}(x_{0})}u(x,t)\chi_{2R}^{2}dx}{\int_{B_{2R}(x_{0})}\chi_{2R}^{2}(x)dx}.$$

\begin{lemma}\label{caccioppoli2}
Let $A\in\mathscr{A}(\lambda,\Lambda)$ and let
$u\in\overset{\circ}{V}(\omega\times(0,T),\mathbb{R}^{N}))$ be a
weak solution of \eqref{appendix1} with $f\equiv0$. Then for all
${Q}_{2R}\subset{Q}$, we have
\begin{equation}\label{lem75.0}
\sup_{t_{0}-R^{2}\leq\,t\leq\,t_{0}}\int_{B_{R}}|u-\bar{u}_{2R}|^{2}dx+\int_{Q_{R}}|\nabla{u}|^{2}dz\leq\frac{C}{R^{2}}\int_{Q_{2R}}|u-\bar{u}_{2R}|^{2}dz+C\int_{Q_{2R}}|g|^{2}dz,
\end{equation}
where $C$ depends only on $n,N,\lambda$, and $\Lambda$.
\end{lemma}

\begin{proof}
Let $z_{0}=(x_{0},t_{0})\in{Q}$, and $Q_{2R}(z_{0})\subset{Q}$. By
the definition of weak solutions, we test \eqref{appendix1} with
$$\varphi=\Big(u_{h}-\bar{u}_{h}\Big)\eta^{2},$$
where $\eta=\eta_{2R}=\chi_{2R}\tau_{2R}$,
\begin{equation}\label{steklov}
u_{h}(x,t):=\frac{1}{2h}\int_{t-h}^{t+h}\widetilde{u}(x,s)ds,\quad\mbox{for}~~0<h<R^{2},
\end{equation}
and
$$\widetilde{u}(x,t)=\begin{cases}
u(x,\tau),&t\geq\tau,\\
u(x,t),&t_{0}-(2R)^{2}< t<\tau,\\
u(x,t_{0}-(2R)^{2}),&t\leq{t}_{0}-(2R)^{2}.
\end{cases}
$$
Then we have, for $\tau\in[t_{0}-R^{2},t_{0}]$
\begin{align}\label{lem75.1}
\int_{B_{2R}}(u\varphi)(\cdot,t_{0})dx-\int_{t_{0}-(2R)^{2}}^{\tau}\int_{B_{2R}}u\varphi_{t}dxdt
&+\int_{t_{0}-(2R)^{2}}^{\tau}\int_{B_{2R}}A_{ij}^{\alpha\beta}D_{\beta}u^{j}D_{\alpha}\varphi^{i}dxdt\nonumber\\
&=\int_{t_{0}-(2R)^{2}}^{\tau}\int_{B_{2R}}g^{\alpha}D_{\alpha}\varphi\,dxdt.
\end{align}
By our choice of the test function, we know that the term
$$\int_{t_{0}-(2R)^{2}}^{\tau}\left(\int_{B_{2R}}\left(u_{h}-\bar{u}_{h}\right)\chi^{2}dx\right)\Big(\partial_{t}\bar{u}\Big)\tau_{2R}^{2}dt=0.$$
It follows that
\begin{align*}
\int_{t_{0}-(2R)^{2}}^{\tau}\int_{B_{2R}}\bar{u}\partial_{t}\left((u_{h}-\bar{u}_{h})\eta^{2}\right)dxdt
=\int_{B_{2R}}\left(\bar{u}\Big(u_{h}-\bar{u}_{h}\Big)\eta^{2}\right)(x,\tau)dx.
\end{align*}
Then we have
\begin{align}\label{lem75.2}
&\int_{t_{0}-(2R)^{2}}^{\tau}\int_{B_{2R}}u\varphi_{t}dxdt=\int_{t_{0}-(2R)^{2}}^{\tau}\int_{B_{2R}}u\partial_{t}\left(\Big(u_{h}-\bar{u}_{h}\Big)\eta^{2}\right)dxdt\nonumber\\
&=\int_{t_{0}-(2R)^{2}}^{\tau}\int_{B_{2R}}\left(u-\bar{u}\right)\partial_{t}\left(\left(u_{h}-\bar{u}_{h}\right)\eta^{2}\right)dxdt
+\int_{B_{2R}}\left(\bar{u}\Big(u_{h}-\bar{u}_{h}\Big)\eta^{2}\right)(x,\tau)dx\nonumber\\
&=\int_{t_{0}-(2R)^{2}}^{\tau}\int_{B_{2R}}\left(u-\bar{u}\right)\left(u_{h}-\bar{u}_{h}\right)_{t}\eta^{2}dxdt\nonumber\\
&\hspace{1cm}+\int_{t_{0}-(2R)^{2}}^{\tau}\int_{B_{2R}}\left(u-\bar{u}\right)\left(u_{h}-\bar{u}_{h}\right)\left(\partial_{t}\left(\eta^{2}\right)\right)dxdt
+\int_{B_{2R}}\left(\bar{u}\Big(u_{h}-\bar{u}_{h}\Big)\eta^{2}\right)(x,\tau)dx.
\end{align}
First, we will show that
\begin{align}\label{lem75.3}
&\lim_{h\rightarrow0}\int_{t_{0}-(2R)^{2}}^{\tau}\int_{B_{2R}}\left(u-\bar{u}\right)\left(u_{h}-\bar{u}_{h}\right)_{t}\eta^{2}dxdt\nonumber\\
&=\frac{1}{2}\int_{B_{2R}}\left(\left(u-\bar{u}\right)^{2}\eta^{2}\right)(x,\tau)dx
-\frac{1}{2}\int_{t_{0}-(2R)^{2}}^{\tau}\int_{B_{2R}}\left(\left(u-\bar{u}\right)^{2}\partial_{t}\left(\eta^{2}\right)\right)(x,\tau)dx.
\end{align}
By the definition of $u_{h}$, \eqref{steklov}, we have, for
$0<h<\frac{R^{2}}{2}$,
\begin{align*}
&\int_{t_{0}-(2R)^{2}}^{\tau}\int_{B_{2R}}\left(u-\bar{u}\right)\left(u_{h}-\bar{u}_{h}\right)_{t}\eta^{2}dxdt\\
&=\frac{1}{2h}\int_{t_{0}-(2R)^{2}}^{\tau}\int_{B_{2R}}\Big(u-\bar{u}\Big)(x,t)
\bigg\{\Big(\widetilde{u}-\overline{\widetilde{u}}\Big)(x,t+h)-\Big(\widetilde{u}-\overline{\widetilde{u}}\Big)(x,t-h)\bigg\}\eta^{2}(x,t)dxdt\\
&=\frac{1}{2h}\int_{t_{0}-(2R)^{2}}^{\tau}\int_{B_{2R}}\bigg\{\Big(u-\bar{u}\Big)(x,t)
\Big(\widetilde{u}-\overline{\widetilde{u}}\Big)(x,t+h)\eta^{2}(x,t+h)\\
&\hspace{3.5cm}-\Big(u-\bar{u}\Big)(x,t)\Big(\widetilde{u}-\overline{\widetilde{u}}\Big)(x,t-h)\eta^{2}(x,t)\bigg\}dxdt\\
&\hspace{.5cm}+\frac{1}{2h}\int_{t_{0}-(2R)^{2}}^{\tau}\int_{B_{2R}}\Big(u-\bar{u}\Big)(x,t)
\Big(u-\bar{u}\Big)(x,t+h)\left\{\eta^{2}(x,t)-\eta^{2}(x,t+h))\right\}dxdt\\
&=\frac{1}{2h}\int_{\tau}^{\tau+h}\int_{B_{2R}}\Big(u-\bar{u}\Big)(x,t-h)\Big(u-\bar{u}\Big)(x,\tau)\eta^{2}(x,t)dxdt\\
&\hspace{.5cm}+\frac{1}{2h}\int_{t_{0}-(2R)^{2}}^{\tau}\int_{B_{2R}}\Big(u-\bar{u}\Big)(x,t)
\Big({u}-\bar{u}\Big)(x,t+h)\left\{\eta^{2}(x,t)-\eta^{2}(x,t+h))\right\}dxdt\\
&=\mathrm{I}_{h}+\mathrm{II}_{h}.
\end{align*}
Clearly,
$$\lim_{h\rightarrow0}\mathrm{II}_{h}=-\frac{1}{2}\int_{t_{0}-(2R)^{2}}^{\tau}\int_{B_{2R}}\left(u-\bar{u}\right)^{2}\partial_{t}\left(\eta^{2}\right)dxdt,$$
and
\begin{align*}
&\left|2I_{h}-\int_{B_{2R}}\left((u-\bar{u})^{2}\eta^{2}\right)(x,\tau)dx\right|\\
&\leq\frac{1}{h}\int_{\tau}^{\tau+h}\int_{B_{2R}}\left|(u-\bar{u})(x,t-h)-(u-\bar{u})(x,\tau)\right|
\left|u-\bar{u}\right|(x,\tau)\eta^{2}dxdt\\
&\leq\frac{C(R)}{h}\int_{\tau}^{\tau+h}\|u(\cdot,t-h)-u(\cdot,\tau)\|_{L^{2}(\omega)}\cdot
\|\left(u-\bar{u}\right)(x,\tau)\eta^{2}(x,\tau)\|_{L^{2}(\omega)}dt\\
&\rightarrow0\quad\quad\mbox{as}\quad\,h\rightarrow0.
\end{align*}
For the last step above, we have used the fact that $u\in{V}^{1,0}(\omega\times(0,T))$, see Lemma \ref{lem31} (the conclusion and its proof are valid under our hypothese).
Thus we obtain \eqref{lem75.3}. Again, using $u\in{V}^{1,0}(\omega\times(0,T))$, we have
$$\lim_{h\rightarrow0}\int_{B_{2R}}|u-u_{h}|^{2}(x,\tau)=0,\quad\mbox{and}
\quad\lim\limits_{h\rightarrow0}\int_{t_{0}-(2R)^{2}}^{\tau}\int_{B_{2R}}\left|u-u_{h}\right|^{2}dxdt=0.$$
It follows that
$$\lim_{h\rightarrow0}\int_{B_{2R}}|(u-\bar{u})-(u_{h}-\bar{u}_{h})|^{2}(x,\tau)=0,$$
and
$$\lim_{h\rightarrow0}\int_{t_{0}-(2R)^{2}}^{\tau}\int_{B_{2R}}\left|\left(u-\bar{u}\right)-\left(u_{h}-\bar{u}_{h}\right)\right|^{2}dxdt=0.$$
Then we have
$$\lim_{h\rightarrow0}\int_{B_{2R}}\bar{u}\left((u-\bar{u})-(u_{h}-\bar{u}_{h})\right)\eta^{2}(x,\tau)=0,$$

$$\lim_{h\rightarrow0}\int_{t_{0}-(2R)^{2}}^{\tau}\int_{B_{2R}}\left(u-\bar{u}\right)
\Big\{\left(u-\bar{u}\right)-\left(u_{h}-\bar{u}_{h}\right)\Big\}\partial_{t}\left(\eta^{2}\right)dxdt=0.$$
That is,
\begin{align}\label{lem75.40}
\lim_{h\rightarrow0}\int_{B_{2R}}\bar{u}(u_{h}-\bar{u}_{h})\eta^{2}(x,\tau)dx=\int_{B_{2R}}\bar{u}(u-\bar{u})\eta^{2}(x,\tau)dx,
\end{align}
and
\begin{align}\label{lem75.4}
\lim_{h\rightarrow0}\int_{t_{0}-(2R)^{2}}^{\tau}\int_{B_{2R}}\left(u-\bar{u}\right)\left(u_{h}-\bar{u}_{h}\right)\partial_{t}\left(\eta^{2}\right)dxdt
=\int_{t_{0}-(2R)^{2}}^{\tau}\int_{B_{2R}}\left(u-\bar{u}\right)^{2}\partial_{t}\left(\eta^{2}\right)dxdt.
\end{align}
By \eqref{lem75.3}, \eqref{lem75.40} and \eqref{lem75.4}, we have
from \eqref{lem75.2},
\begin{align}\label{lem75.5}
\lim_{h\rightarrow0}\int_{t_{0}-(2R)^{2}}^{\tau}\int_{B_{2R}}u\varphi_{t}dxdt&=\int_{B_{2R}}\bar{u}\left(u-\bar{u}\right)(x,\tau)\eta^{2}(x,\tau)dx
+\frac{1}{2}\int_{B_{2R}}\left(u-\bar{u}\right)^{2}(x,\tau)\eta^{2}(x,\tau)dx\nonumber\\
&\hspace{0.5cm}+\frac{1}{2}\int_{t_{0}-(2R)^{2}}^{\tau}\int_{B_{2R}}\left(u-\bar{u}\right)^{2}\partial_{t}\left(\eta^{2}\right)dxdt.
\end{align}
Since $\|u_{h}-u\|_{V(Q_{2R})}\rightarrow 0$ as $h\rightarrow0$, it
follows that
\begin{align}\label{lem75.6}
 \lim_{h\rightarrow
0}\int_{t_{0}-(2R)^{2}}^{\tau}\int_{B_{2R}}A^{\alpha\beta}_{ij}D_{\beta}u^{j}D_{\alpha}\varphi^{i}dxdt
=\int_{t_{0}-(2R)^{2}}^{\tau}\int_{B_{2R}}A^{\alpha\beta}_{ij}D_{\beta}\left(u-\bar{u}\right)^{j}D_{\alpha}
\Big[\left(u-\bar{u}\right)\eta^{2}\Big]^{i}dxdt,
\end{align}
and
\begin{align}\label{lem75.7}
\lim_{h\rightarrow 0}\int_{t_{0}-(2R)^{2}}^{\tau}\int_{B_{2R}}
g^{\alpha}D_{\alpha}\varphi
=\int_{t_{0}-(2R)^{2}}^{\tau}\int_{B_{2R}}g^{\alpha}D_{\alpha}\Big[\left(u-\bar{u}\right)\eta^{2}\Big]dxdt.
\end{align}
Then sending $h\rightarrow0$ in \eqref{lem75.1}, from
\eqref{lem75.5}, \eqref{lem75.6} and \eqref{lem75.7}, we have
\begin{align*}
&\frac{1}{2}\int_{B_{2R}}\left(u-\bar{u}\right)^{2}(x,\tau)\eta^{2}(x,\tau)dx+\int_{t_{0}-(2R)^{2}}^{\tau}\int_{B_{2R}}A_{ij}^{\alpha\beta}D_{\beta}\left(u-\bar{u}\right)^{j}D_{\alpha}\Big[\left(u-\bar{u}\right)\eta^{2}\Big]^{i}dxdt\\
&=\int_{t_{0}-(2R)^{2}}^{\tau}\int_{B_{2R}}\left(g^{\alpha}D_{\alpha}\Big[\left(u-\bar{u}\right)\eta^{2}\Big]\right)dxdt
+\int_{t_{0}-(2R)^{2}}^{\tau}\int_{B_{2R}}\left(u-\bar{u}\right)^{2}\partial_{t}\left(\eta^{2}\right)dxdt.
\end{align*}
By the Cauchy inequality, for any $\epsilon>0$,
\begin{align*}
&\int_{t_{0}-(2R)^{2}}^{\tau}\int_{B_{2R}}\left(g^{\alpha}D_{\alpha}\Big[\left(u-\bar{u}\right)\eta^{2}\Big]\right)dxdt\\
&\leq\frac{\epsilon}{4}\int_{t_{0}-(2R)^{2}}^{\tau}\int_{B_{2R}}\left|\nabla\Big[\left(u-\bar{u}\right)\eta\Big]\right|^{2}dxdt
+C(\epsilon)\int_{Q_{2R}(z_{0})}|g|^{2}dz.
\end{align*}
By a simply computation, we have
\begin{align*}
\int_{t_{0}-(2R)^{2}}^{\tau}\int_{B_{2R}}\left(u-\bar{u}\right)^{2}\partial_{t}\left(\eta^{2}\right)dxdt
\leq\frac{C}{R^{2}}\int_{t_{0}-(2R)^{2}}^{\tau}\int_{B_{2R}}|u-\bar{u}|^{2}dxdt,
\end{align*}
and
\begin{align*}
&\int_{t_{0}-(2R)^{2}}^{\tau}\int_{B_{2R}}A_{ij}^{\alpha\beta}D_{\beta}\left(u-\bar{u}\right)^{j}D_{\alpha}
\Big[\left(u-\bar{u}\right)\eta^{2}\Big]^{i}dxdt\\
&=\int_{t_{0}-(2R)^{2}}^{\tau}\int_{B_{2R}}A_{ij}^{\alpha\beta}D_{\beta}\Big[\left(u-\bar{u}\right)^{j}\eta\Big]
D_{\alpha}\Big[\left(u-\bar{u}\right)^{i}\eta\Big]dxdt\\
&\hspace{1cm}-\int_{t_{0}-(2R)^{2}}^{\tau}\int_{B_{2R}}A_{ij}^{\alpha\beta}\Big[(D_{\beta}\eta)\left(u-\bar{u}\right)^{j}\Big]
D_{\alpha}\Big[\left(u-\bar{u}\right)^{i}\eta\Big]dxdt\\
&\hspace{2cm}+\int_{t_{0}-(2R)^{2}}^{\tau}\int_{B_{2R}}A_{ij}^{\alpha\beta}D_{\beta}\Big[\left(u-\bar{u}\right)^{j}\eta\Big]\Big[(D_{\alpha}\eta)
\left(u-\bar{u}\right)^{i}\Big]dxdt\\
&\hspace{3cm}-\int_{t_{0}-(2R)^{2}}^{\tau}\int_{B_{2R}}A_{ij}^{\alpha\beta}\Big[(D_{\beta}\eta)\left(u-\bar{u}\right)^{j}\Big]\Big[(D_{\alpha}\eta)
\left(u-\bar{u}\right)^{i}\Big]dxdt.
\end{align*}
Then by the Cauchy inequality again and \eqref{coeff1}, we have
\begin{align*}
&\frac{1}{2}\int_{B_{2R}}\left(u-\bar{u}\right)^{2}(x,\tau)\eta^{2}(x,\tau)dx
+\int_{t_{0}-(2R)^{2}}^{\tau}\int_{B_{2R}}A_{ij}^{\alpha\beta}D_{\beta}\Big[\left(u-\bar{u}\right)^{j}\eta\Big]
D_{\alpha}\Big[\left(u-\bar{u}\right)\eta\Big]^{i}dxdt\\
&\leq\epsilon\int_{t_{0}-(2R)^{2}}^{\tau}\int_{B_{2R}}\left|\nabla\Big[\left(u-\bar{u}\right)\eta\Big]\right|^{2}dxdt
+\frac{C(\epsilon)}{R^{2}}\int_{t_{0}-(2R)^{2}}^{\tau}\int_{B_{2R}}|u-\bar{u}|^{2}dxdt+C(\epsilon)\int_{Q_{2R}(z_{0})}|g|^{2}dz.
\end{align*}
Therefore, by the weak parabolic condition \eqref{coeff2}, taking
$\epsilon=\frac{\lambda}{2}$, we have
$$\int_{B_{R}}\left(u-\bar{u}\right)^{2}(x,\tau)dx+
\int_{Q_{R}}\left|\nabla(u-\bar{u})\right|^{2}dz
\leq\frac{C}{R^{2}}\int_{Q_{2R}(z_{0})}\left|u-\bar{u}\right|^{2}dz+C\int_{Q_{2R}(z_{0})}|g|^{2}dz,$$
where $C$ depends only on $n,N,\lambda$ and $\Lambda$. The proof of Lemma \ref{caccioppoli2} is
completed.
\end{proof}

\begin{proof}[Proof of Proposition \ref{prop73}]

Using property \eqref{chi2R}, we have
$$\int_{B_{2R}}|u-\bar{u}_{2R}|^{2}\leq\,C\int_{B_{2R}}|u-\bar{u}_{4R}|^{2},$$
where $C$ depends only on $n$. From Lemma \ref{caccioppoli2} with
$R$ replaced by $2R$, for $t_{0}-(2R)^{2}\leq{t}\leq{t}_{0}$,
\begin{align}\label{prop73.1}
\int_{B_{2R}}|u-\bar{u}_{4R}|^{2}dx\leq\frac{C}{R^{2}}\int_{Q_{4R}}|u-\bar{u}_{4R}|^{2}dz+\int_{Q_{4R}}|g|^{2}dz.
\end{align}
By H\"{o}lder inequality, Poincar\'{e} inequality and
\eqref{prop73.1}, we have
\begin{align}\label{prop73.2}
\int_{B_{2R}}|u-\bar{u}_{2R}|^{2}dx
&=\left(\int_{B_{2R}}|u-\bar{u}_{2R}|^{2}dx\right)^{\frac{n}{n+2}}\left(\int_{B_{2R}}|u-\bar{u}_{2R}|^{2}dx\right)^{\frac{2}{n+2}}\nonumber\\
&\leq\,C\left(\int_{B_{2R}}|\nabla{u}|^{\frac{2n}{n+2}}dx\right)
\left(\frac{1}{R^{2}}\int_{Q_{4R}}|u-\bar{u}_{4R}|^{2}dz+\int_{Q_{4R}}|g|^{2}dz\right)^{\frac{2}{n+2}}\nonumber\\
&\leq\,C\left(\int_{B_{2R}}|\nabla{u}|^{\frac{2n}{n+2}}dx\right)
\left(\int_{Q_{4R}}|\nabla{u}|^{2}dz+\int_{Q_{4R}}|g|^{2}dz\right)^{\frac{2}{n+2}}
\end{align}
Integrating over $t$ leads to, for every $\epsilon>0$,
\begin{align*}
\iint_{Q_{2R}}|u-\bar{u}_{2R}|^{2}dz
&\leq\,C\left(\int_{Q_{2R}}|\nabla{u}|^{\frac{2n}{n+2}}dz\right)
\left(\int_{Q_{4R}}|\nabla{u}|^{2}dz+\int_{Q_{4R}}|g|^{2}dz\right)^{\frac{2}{n+2}}\\
&\leq\epsilon\left(\int_{Q_{4R}}\left(|\nabla{u}|^{2}+|g|^{2}\right)dz\right)
+\frac{C}{\epsilon^{\frac{2}{n}}}\left(\int_{Q_{2R}}|\nabla{u}|^{\frac{2n}{n+2}}dz\right)^{\frac{n+2}{n}}.
\end{align*}
Using \eqref{lem75.0}, we have
\begin{align*}
\fint_{Q_{2R}}|\nabla{u}|^{2}dz
&\leq\frac{C}{R^{2}}\fint_{Q_{2R}}|u-\bar{u}_{2R}|^{2}dz+C\fint_{Q_{2R}}|g|^{2}dz\\
&\leq\frac{C\epsilon}{R^{2}}\left(\fint_{Q_{4R}}\left(|\nabla{u}|^{2}+|g|^{2}\right)dz\right)
+\frac{C}{\epsilon^{\frac{2}{n}}R^{n+4}}\left(\int_{Q_{4R}}|\nabla{u}|^{\frac{2n}{n+2}}dz\right)^{\frac{n+2}{n}}+C\fint_{Q_{2R}}|g|^{2}dz.
\end{align*}
Taking $\epsilon>0$ such that $\frac{C\epsilon}{R^{2}}=\frac{1}{2}$,
we have
$$\epsilon^{\frac{2}{n}}R^{n+4}=CR^{\frac{4}{n}+n+4}=CR^{\frac{4+(n+4)n}{n}}=CR^{\frac{(n+2)^{2}}{n}}.$$
So
$$\fint_{Q_{R}}|\nabla{u}|^{2}dz\leq\frac{1}{2}\fint_{Q_{4R}}|\nabla{u}|^{2}dz+C\left(\fint_{Q_{4R}}|\nabla{u}|^{\frac{2n}{n+2}}dz\right)^{\frac{n+2}{n}}
+C\fint_{Q_{4R}}|g|^{2}dz.$$ Then taking
$h=|\nabla{u}|^{\frac{2n}{n+2}}$, $q=\frac{n+2}{n}$ and
$H=|g|^{\frac{2n}{n+2}}$, we obtain, in view of Proposition \ref{lq},
\eqref{prop73.0} and have proved Proposition \ref{prop73}.
\end{proof}

Given the interior estimates Proposition \ref{prop73}, we now only need to establish boundary estimates analogous to \eqref{prop73.0}.

\noindent{\bf The completion of the proof of Theorem \ref{reverseholder} when $f\equiv0$.}
Since $\partial\omega$ is Lipschitz continuous, there exists $\bar{R}>0$ such that for all $\bar{x}\in\partial\omega$, $\partial\omega\cap{B}_{\bar{R}}(\bar{x})$ is the graph of a Lipschitz function with controlled Lipschitz constant. In view of Proposition \ref{prop73}, we only need to establish \eqref{prop73.0} for all $Q_{4R}(z_{0})$ with $t_{0}\leq{T}$, $0<R<\frac{1}{8}\bar{R}$.
Note that we allow $\Omega_{T}\setminus{Q}_{4R}\neq\emptyset$, and here $u$ and $g$ have been extended as zero outside $\Omega_{T}:=\omega\times(0,T)$.

There are three cases: Case 1, where $B_{\frac{3}{2}R}(x_{0})\cap\omega^{c}=\emptyset$, can be seen as the interior case, and has been settled; Case 2, where $B_{\frac{3}{2}R}(x_{0})\subset\omega^{c}$, is trivial; We only need to consider Case 3, where $B_{\frac{3}{2}R}(x_{0})\cap\partial\omega^{c}\neq\emptyset$.

For Case 3, in order to prove \eqref{prop73.0} in $Q_{4R}^{+}:=Q_{4R}(z_{0})\cap\Omega_{T}$, we
need only to replace $\bar{u}_{x_{0},2R}(t)$ by
$$\bar{u}^{+}_{x_{0},2R}(t)=\frac{\int_{B_{2R}(x_{0})\cap\omega}u(x,t)\chi_{2R}^{2}dx}{\int_{B_{2R}(x_{0})\cap\omega}\chi_{2R}^{2}(x)dx},$$
and let $\chi(x)$ be a function in $C^{\infty}_{c}(B_{2}(x_{0}))$ such
that $0\leq\chi\leq1$, $\chi\equiv1$ in $B_{\frac{3}{2}}(x_{0})$ and
$|\nabla\chi|\leq4$, take
$$\chi_{2R}(x)=\chi\left(x_{0}+\frac{x-x_{0}}{R}\right),$$
satisfying
$|\nabla\chi_{2R}|\leq\frac{C(n)}{R}.$
Then by the same way, we could obtain the estimate \eqref{prop73.0}. The choice of $R$ and ball $B_{2R}(x_{0})$ guarantees the validity of the Sobolev inequality used in \eqref{prop73.0}.

It follows that for some $p>2$, the $L^{p}$ norm of $|\nabla{u}|$ is
controlled by the $L^{2}$ norm of $|\nabla{u}|$ and the $L^{p}$ norm
of $g$. On the other hand, we know that the $L^{2}$ norm of
$\nabla{u}$ is controlled by the $L^{2}$ norm of $g$. Therefore we
have shown that, for some $p>2$,
\begin{equation}\label{lem81.1}
\int_{\Omega_{T}}|\nabla{u}|^{p}dz\leq\,C\int_{\Omega_{T}}|g|^{p}dz.
\end{equation}

\subsection{Completion of the Proof of Theorem \ref{reverseholder}}
In order to complete the proof of Theorem \ref{reverseholder}, we need the
following Lemma.

\begin{lemma}\label{f}
Under the condition of Theorem \ref{reverseholder} with $g\equiv0$, and let $p_{0}$ be as in Proposition \ref{prop73}. Then for all
$2\leq{p}<p_{0}$,
$f\in{L}^{p}(0,T,L^{2}(\omega,\mathbb{R}^{N}))$, we have
$u\in{L}^{p}(\Omega_{T})$ and
$$\int_{\Omega_{T}}|\nabla{u}|^{p}dxdt\leq\,C\int_{0}^{T}\left(\int_{\omega}|f|^{2}dx\right)^{p/2},$$
where $C$ depends only on $n,N,\lambda,\Lambda$ and $\omega$.
\end{lemma}

\begin{proof}
Let $U$ be the solution of
$$U\in{H}^{1}_{0}(\omega,\mathbb{R}^{N})\cap{H}^{2}(\omega,\mathbb{R}^{N}),\quad-\Delta{U}^{i}=f^{i}.$$
It is known that
$$\|U\|_{W^{2,2}(\omega)}\leq\,C\|f\|_{L^{2}(\omega,\mathbb{R}^{N})},$$
where $C$ depends only on $n, N$ and $\omega$,
and by the imbedding theorems with respect to $x$,
\begin{equation}\label{U}
\int_{\omega}|DU|^{p}dx\leq\,C\|U\|^{p}_{W^{2,2}(\omega)}\leq\,C\|f\|^{p}_{L^{2}(\omega,\mathbb{R}^{N})}
\end{equation}
Then the weak solution of \eqref{appendix1} with $g\equiv0$, $u$, satisfies, for a.e. $\tau\in(0,T)$ and for all
$\varphi\in\overset{\circ}{W}\,^{1,1}_{2}(\Omega_{T};\mathbb{R}^{N})$,
$$\int_{\omega}(u\varphi)(\cdot,\tau)dx-\int_{\Omega_{T}}\left(A^{\alpha\beta}_{ij}\partial_{\beta}u^{j}\partial_{\alpha}\varphi^{i}-u\varphi_{t}\right)dxdt
=\int_{\Omega_{T}}\partial_{\alpha}U^{i}\partial_{\alpha}\varphi^{i}dxdt.$$
By \eqref{lem81.1} and \eqref{U}, we conclude
that
$$\int_{\Omega_{T}}|\nabla{u}|^{p}dxdt\leq\,C\int_{\Omega_{T}}|DU|^{p}dxdt\leq\,C\int_{0}^{T}
\left(\int_{\omega}|f|^{2}dx\right)^{p/2}dt,$$ and this is the
required assertion.
\end{proof}

\begin{proof}[Proof of Theorem \ref{reverseholder}]
Combining the proof of Theorem \ref{reverseholder} with $f\equiv0$ and Lemma \ref{f}, the proof is
completed.
\end{proof}

\noindent{\bf{\large Acknowledgements.}} The first author is
grateful to Professor Jiguang Bao for helpful comments and
encouragement. Part of the work was completed while the first author
was visiting Rutgers University, he also thanks the mathematics
department and the Nonlinear Analysis Center for the hospitality.
The first author was partially supported by SRFDPHE (20100003120005),
NSFC (11071020) and (11126038). The work of the second author was
partially supported by NSF grant DMS-0701545. Both authors were
partially supported by Program for Changjiang Scholars and
Innovative Research Team in University in China.


\end{document}